\documentclass{conm-p-l}
\usepackage{amssymb, stmaryrd, amscd}

\newtheorem{theorem}[equation]{Theorem}

\newtheorem{lemma}[equation]{Lemma}

\newtheorem{corollary}[equation]{Corollary}

\newtheorem{proposition}[equation]{Proposition}

\newtheorem{remark}[equation]{Remark}

\newtheorem{definition}[equation]{Definition}

\newcommand{\cO}{\mathcal{O}}

\renewcommand{\gg}{\mathfrak{g}}
\newcommand{\gp}{\mathfrak{p}}
\newcommand{\gq}{\mathfrak{q}}

\makeatletter
\def\revddots{\mathinner{\mkern1mu\raise\p@
\vbox{\kern7\p@\hbox{.}}\mkern2mu
\raise4\p@\hbox{.}\mkern2mu\raise7\p@\hbox{.}\mkern1mu}}
\makeatother

\def\ga{\mathfrak{a}}
\def\gb{\mathfrak{b}}
\def\gc{\mathfrak{c}}
\def\gd{\mathfrak{d}}

\def\gg{\mathfrak{g}}

\def\gh{\mathfrak{h}}

\def\gj{\mathfrak{j}}
\def\gk{\mathfrak{k}}
\def\gl{\mathfrak{l}}
\def\gm{\mathfrak{m}}
\def\gn{\mathfrak{n}}

\def\gp{\mathfrak{p}}
\def\gq{\mathfrak{q}}
\def\gr{\mathfrak{r}}
\def\gs{\mathfrak{s}}

\def\gt{\mathfrak{t}}
\def\gu{\mathfrak{u}}
\def\gv{\mathfrak{v}}

\def\gz{\mathfrak{z}}

\def\Int{{\rm Int}}
\def\Ad{{\rm Ad}}

\def\rank{{\rm rank\,}}
\def\trace{{\rm trace\,\,}}

\def\Ind{{\rm Ind\,}}

\def\Car{{\rm Car}}

\def\C{\mathbb{C}}

\def\E{\mathbb{E}}

\def\L{\mathbb{L}}

\def\R{\mathbb{R}}
\def\T{\mathbb{T}}

\def\Z{\mathbb{Z}}

\def\cH{\mathcal{H}}

\def\cO{\mathcal{O}}

\def\cR{\mathcal{R}}
\def\cS{\mathcal{S}}

\def\cU{\mathcal{U}}
\def\cV{\mathcal{V}}

\def\cZ{\mathcal{Z}}

\begin{document}

\title[Representations on Partially Holomorphic Cohomology Spaces]
{Representations on Partially Holomorphic Cohomology Spaces, Revisited}

\author{Joseph A. Wolf}

\address{Department of Mathematics, University of California,
Berkeley, CA 94720--3840}
\curraddr{}
\email{jawolf@math.berkeley.edu}
\thanks{Research partially supported by the Simons Foundation}

\date{31 July, 2017}

\begin{abstract} This is a semi--expository update and rewrite 
of my 1974 AMS AMS Memoir describing Plancherel formulae and
partial Dolbeault cohomology realizations for standard
tempered representations for general real reductive Lie groups.  
Even after so many years, much of that Memoir is up to date, but
of course there have been a number of refinements, advances and
new developments, most of which have applied to smaller classes
of real reductive Lie groups.  Here we rewrite that AMS Memoir in
in view of these advances and indicate the ties with some of the
more recent (or at least less classical)  approaches to geometric 
realization of unitary representations.
\end{abstract}

\maketitle

\setcounter{section}{-1}
\section{Introduction}\label{sec0} 
\setcounter{equation}{0}

In the 1960's I initiated the study of the orbit structure
for the action of a general real semisimple Lie group $G$ on a complex flag
manifold $X = G_\C/Q$ \cite{W1969}.  Here $Q$ is a parabolic subgroup of the
complexification $G_\C$ of $G$.  In the early 1970's I associated the 
various series of standard tempered representations of $G$ to certain 
$G$--orbits in $X$.  I showed how to construct the standard tempered 
representations as the natural action of $G$ on certain square integrable 
partially holomorphic cohomology spaces, corresponding to the CR structure 
of the orbit and to partially holomorphic complex vector bundles over the orbit
\cite{W1973}.  At the time these geometric realizations required the 
infinitesimal characters of the standard tempered representations be 
sufficiently nonsingular \cite{S1971}; but that condition is no longer needed
\cite{S1976}.  This simplifies the theory and clarifies
our exposition.

The advances in the theory include (i) getting rid of the ``sufficiently 
nonsingular'' condition, as mentioned above; (ii) a complete description of 
the irreducible constituents of the standard tempered representations
(\cite{L1973}, \cite{KZ1982a}, \cite{KZ1982b}) and their analytic continuation 
(\cite{V1983a}, \cite{V1983b});
(iii) a better approach to the Plancherel Formula (\cite{HW1986a},
\cite{HW1986b}); and (iv) the Atlas project \cite{ALTV2015}.  Most of the 
advances (ii), (iii) and (iv) apply to smaller classes of groups than the
ones we consider here, specifically to groups of Harish--Chandra class or
even the smaller class of real reductive linear algebraic groups.
In this paper we go over the material of the original AMS Memoir,
indicating the principal changes in view of those developments.  We
also fix one error and a few trivial typos.

Semisimple representation theory has benefited from increased use of
algebraic methods, with more emphasis on Harish--Chandra modules than on
representations of the group itself.  This has led to better results in 
many cases, such as the advances (i) and (ii) mentioned above, but has 
tended to sever the connection with differential geometry.  Papers that 
bridge the gap to some extent include \cite{HMSW1987}, \cite{SW1987},
\cite{SW1990}, \cite{W1995} and \cite{W1999}.  While one must be aware of
those results we do not attempt to describe them here, first because of
limitations of space, and second because many of them 
have not yet been extended to general real reductive Lie groups.

\subsection{}\label{ssec0a}  
\setcounter{equation}{0}
In this paper we work with the class of 
{\em general real reductive Lie groups}.  Those are the
Lie groups $G$ whose Lie algebra $\gg$ is reductive and such that
(a) if $g \in G$ then $\Ad(g)$ is an inner automorphism of the complexified
Lie algebra $\gg_\C$ and (b) $G$ has a closed normal abelian subgroup
$Z$ such that (i) $Z$ centralizes the identity component $G^0$, (ii) 
$|G/ZG^0| < \infty$ and (iii) $Z \cap G^0$ is co-compact in the center
$Z_{G^0}$ of $G^0$.  These conditions are inherited by reductive components
of parabolic subgroups, and the class of groups that satisfy them includes
both Harish--Chandra's class $\cH$ and all connected real semisimple Lie
groups.  See \cite[\S 0.3]{W1973} for a discussion.  
We work out geometric realizations  for all classes
in the unitary dual $\widehat{G}$ except for a set of Plancherel measure
zero, and we express the Plancherel formula in terms of them.

The first part is the construction and analysis of standard tempered
representations of general real reductive Lie groups.  Roughly speaking 
this is a matter of extending Harish--Chandra's work from groups of class
$\cH$, in other word from the case where $|G/G^0| < \infty$ and
$[G^0,G^0]$ has finite center.  We construct a series of representations 
for each conjugacy class of Cartan subgroups of $G$, describe the characters 
of those representations, and use them for a Plancherel formula.  This 
material is in \S\S \ref{sec2} through \ref{sec5}.

The second part is the geometric realization of the representations just
constructed.  Given a Cartan subgroup $H \subset G$ we construct partially
complex homogeneous spaces $Y$ of $G$, partially holomorphic $G$--homogeneous
vector bundles $\E \to Y$, and Hilbert spaces $\cH_2^{0,q}(Y;\E)$ of
square integrable partially harmonic $\E$--valued $(0,q)$--forms on $Y$.
If $[\pi] \in \widehat{G}$ belongs to the series for the $G$--conjugacy
class of $H$ then it is realized by a representation of $G$ by its natural
action on one of the $\cH_2^{0,q}(Y;\E)$.  If $G/Z$ is compact and connected
this reduces to the Bott--Borel--Weil Theorem.  If $H/Z$ is compact this
includes the discrete series realizations of Schmid and Narasimhan--Okamoto.
This material is in \S\S \ref{sec6} through \ref{sec8}.

In \S \ref{sec1} we  illustrate the geometric realization procedure by carrying it out
for the ``principal series'', corresponding to the conjugacy class of minimal
parabolic subgroups $B \subset G$.  We construct an Iwasawa decomposition
$G = KAN$ where $Z \subset K$, and $B = MAN$ where $M$ is the centralizer of
$A$ in $K$.  We construct certain subgroups $U \subset M$ that contain Cartan
subgroups of $M$, partially complex homogeneous spaces $Y = G/UAN$ and
$G$--equivariant fibrations $Y \to Y/M = G/B$ whose fibers are maximal
complex subvarieties of $Y$.  Given a unitary equivalence class 
$[\mu] \in \widehat{U}$ and a functional $\sigma \in \ga^*$ we construct a
$G$--homogeneous bundle $\E_{\mu,\sigma} \to Y$ that is holomorphic over
the fibers of $Y \to Y/M$.  We extend the Bott--Borel--Weil Theorem to
groups that are compact modulo their center and identify the sheaf cohomology 
representation $\eta_\mu^q$
of $M$ on $H^q(M/U;\cO(\E_{\mu,\sigma}|_{M/U}))$.  Given a unitary equivalence
class $[\eta] \in \widehat{M}$ we find all $[\mu] \in \widehat{U}$ and all
$q \geqq 0$ such that $\eta_\mu^q \in [\eta]$.  That done we show that the 
principal series class 
$$
[\pi_{\eta,\sigma}] = [\Ind_B^G(\eta \otimes e^{i\sigma})], 
  \text{ where } (\eta \otimes e^{i\sigma})(man) = e^{i\sigma}(a)\eta(m)
$$
is realized the natural action of $G$ on a certain square integrable partially
holomorphic cohomology space $\cH_2^{0,q}(Y; \E_{\mu,\sigma})$.  The 
representations of $U$ and $M$ come out of the Peter--Weyl Theorem and are
characterized by Cartan's highest weight theory.  In the general case we
substitute an extension of Harish--Chandra's theory of the discrete series,
based on results of Schmid.  In \S \ref{sec1} the $\eta^q$ are identified  by a mild
extension of the Bott--Borel--Weil Theorem.  In the general case the matter is
delicate and we need detailed information on distribution characters of
representations induced from parabolic subgroups.

\subsection{} \label{ssec0b} 
\setcounter{equation}{0}
Here is a more detailed description of the contents
of this paper and a comparison with the earlier AMS Memoir, except 
\S \ref{sec1}, which is described above.

In \S \ref{sec2} we record the basic facts on square integrable representations of
a locally compact unimodular group $G$ relative to a unitary character
$\zeta \in \widehat{Z}$,  Here $Z$ is a closed normal abelian subgroup.
Denote 
$$
L_2(G/Z,\zeta) = \{f:G \to \C \mid f(gz) = \zeta(z)^{-1} 
	\text{ and } \int_{G/Z} |f(g)|^2 d(gZ) < \infty\}
$$
and $\widehat{G}_\zeta = \{[\pi]\in\widehat{G} \mid \pi(gz)=\zeta(z)\pi(g)\}$.
Then $\widehat{G} = \bigcup_{\widehat{Z}}\, \widehat{G}_\zeta$ and 
$L_2(G) = \int_{\widetilde{Z}} L_2(G/Z,\zeta) d\zeta =
\int_{\widehat{G}_\zeta} \cH_\pi \widehat{\otimes} \cH_\pi^*\, d_\zeta[\pi]$
where $d_\zeta[\pi]$ is the Plancherel measure on $\widehat{G}_\zeta$.
We say {\em $\zeta$--discrete} for the classes in $\widehat{G}_{\zeta-disc}$\,.
The set of all these classes forms the {\em $\zeta$--discrete series}
$\widehat{G}_{\zeta-disc}$ of $G$\,, and the {\em relative discrete series} is 
the union $\widehat{G}_{disc} = \bigcup\, \widehat{G}_{\zeta-disc}$\,.  If
$[\pi] \in \widehat{G}_\zeta$ then $[\pi] \in \widehat{G}_{\zeta-disc}$ if
and only if its coefficient functions $\phi_{u,v}(g)=\langle u, \pi(g)v\rangle$
satisfy $|\phi_{u,v}| \in L_2(G/Z)$.  When $Z$ is compact, $\widehat{G}_{disc}$
is the usual discrete series of $G$.  

Let $U$ be a closed subgroup of $G$ with $Z \subset U$ and $U/Z$ compact.
We give a short proof that $\widehat{U} = \widehat{U}_{disc}$ and that every 
class in $\widehat{U}$ is finite dimensional.  Then we write down the relative 
Plancherel formula for the $L_2(G/Z,\zeta)$ and the absolute Plancherel
formula for $L_2(G)$.

In \S \ref{sec3} we see how Harish--Chandra's theory of the discrete 
series for his
class $\cH$ extends to our class of general real reductive Lie groups.  This
was originally done in \cite[\S 3]{W1973} by looking at central extensions
$1 \to S \to G[\zeta] \to ZG^0/Z \to 1$ where $S$ is the circle group
$\{s \in \C \mid |s| = 1\}$ and
$G[\zeta] = \{S \times ZG^0\}/\{\zeta(z)^{-1},z) \mid z \in Z\}$. 
That led to 
a bijection $\widehat{G[\zeta]}_1 \to (\widehat{ZG^0})_\zeta$.  We showed
that $G[\zeta]$ is a connected reductive Lie group with compact center
and verified that Harish--Chandra's discrete series theory applies to
such groups.  Here we go directly and apply results of R. Herb and
the author (\cite{HW1986a}, \cite{HW1986b}).  In particular we see that
$\widehat{G}_{disc}$ is nonempty if and only if $G/Z$ has a compact
Cartan subgroup, and in that case one has the expected infinitesimal and 
global character formulae.

In \S \ref{sec4} we construct a series of unitary representations of $G$ for
every conjugacy class of Cartan subgroups $H \subset G$.  For lack of a
better term we continue to refer to the series for 
$\{\Ad(g)H \mid g \in G\}$
as the ``$H$--series'' of $G$.  If $H/Z$ is compact then the
$H$--series is the relative discrete series.  If $H/Z$ is maximally
noncompact it is the principal series.  If $G/G^0$ and the center of
$[G^0,G^0]$ are finite, in other words if $G$ is of class $\cH$,
then the various $H$--series are just the standard tempered series
constructed by Harish--Chandra and used by him to decompose $L_2(G)$.  Our
constructions and results first came as straightforward extensions
of results of Harish--Chandra, Hirai and Lipsman, but now they follow
directly from Herb and the author \cite{HW1986a}.

If $H$ is a Cartan subgroup of $G$ we construct a ``Cartan involution''
$\theta$ of $G$ that leaves $H$ stable.  Thus $\theta^2 = 1$, $\theta(H) = H$,
$K = \{g \in G \mid \theta(g) = g\}$ contains $Z$, and $K/Z$ is a maximal
compact subgroup of $G/Z$.  Further $H = T \times A$ where $T = H \cap K$ and
$A = \exp(\ga)$, $\ga = \{\xi \in \gh \mid \theta(\xi) = -\xi\}$.  
Choose a positive $\ga$--root system $\Sigma_\ga^+$ on $\gg$, define
$\gn = \sum_{\phi \in \Sigma_\ga^+} \gg^{-\phi}$, and let $N$ be the
analytic subgroup of $G$ for $\gn$.  All this defines
$$
P = \{g \in G \mid \Ad(g)N = N\}, \text{ cuspidal parabolic 
   subgroup of } G.
$$
The centralizer $Z_G(A) = M\times A$ with $\theta(M) = M$, $P = MAN$,
$N$ is the unipotent radical of $P$, and $MA$ is the reductive part of $P$.  
The groups $M$ and $MA$ inherit our working conditions from $G$: they belong
to the class of general real reductive Lie groups.  Further, $T$ is a 
Cartan subgroup of $M$, and $T/Z$ is compact, so $\widehat{M}_{disc}$ is
not empty.  If $[\eta] \in \widehat{M}$ and $\sigma \in \ga^*$ then
$(\eta\otimes e^{i\sigma})(ma) = e^{i\sigma}(a)\eta(m)$ defines an irreducible
unitary representation of $P$.  The {\em $H$--series}  of $G$ consists of
the unitary equivalence classes
$
[\pi_{\eta,\sigma}] = [\Ind_P^G(\eta \otimes e^{i\sigma})]
$ 
where $[\eta] \in \widehat{M}_{disc}$ and $\sigma \in \ga^*$.

We compute central, infinitesimal and distribution characters of the classes
$[\Ind_P^G(\eta \otimes e^{i\sigma})]$.  In particular we see that the
$H$--series classes are finite sums of irreducibles, and that the $H$--series
of $G$ depends only on the conjugacy class of $H$, independent of the
choice of $\Sigma_\ga^+$\,.  We also see that if $H_1$ and $H_2$ are
non--conjugate Cartan subgroups of $H$ then the $H_1$--series and the
$H_2$--series are disjoint.

In \S \ref{sec5} we describe the Plancherel formula for $G$.   There is a 
sharp improvement (\cite{HW1986a} and \cite{HW1986b}) on our original argument
\cite[\S 5]{W1973}.  It gives a precise description of the character formula
and shows that the densities $m_{j,\zeta,\nu}$ in the Plancherel formula are 
restrictions of meromorphic functions.  However that proof is rather technical,
involving machinery that takes some space to describe, and we do not need the
meromorphicity properties of the $m_{j,\zeta,\nu}$.  For that reason we
simply state the result as proved in \cite[\S 5]{W1973}.
That ends the first part of this paper.

\S \ref{sec6} starts the second part of this paper, the geometric realization of
the representations involved in the Plancherel formula, based on the action
of $G$ on complex flag manifolds $X = \overline{G}_\C/Q$\,.  Here $\overline{G}
= G/Z_G(G^0)$ is the adjoint group, $\overline{G}_\C$ is its complexification,
and $Q$ is a parabolic subgroup of $\overline{G}_\C$\,.  Then $G$ acts naturally
on $X$, as in \cite{W1969}, through the lift of the action of $\overline{G}$.
The main points for us are the concept of holomorphic arc component and
measurable orbit.  The orbits over which we realize a general $H$--series
representation of $G$ are measurable, and the relative discrete series
representations of $M$ are constructed over holomorphic arc components of
measurable orbits.  We end \S \ref{sec6} with a classification and structure theory
for the orbits $G(x) \subset X$ over which we have geometric realizations
of $H$--series representations of $G$.

In \S \ref{sec7} we work out the partially holomorphic cohomology realizations of
relative discrete series representations of $G$.  Suppose that $G$ has a 
Cartan subgroup $H$ with $H/Z$ compact.  Then we have complex flag manifolds
$X = \overline{G}_\C/Q$ with orbits $Y = G(x) \subset X$ such that
$Y = G/U$ and $H \subset U$ with $U/Z$ compact, and all such $G$--orbits
on $X$ are open.  They are defined by a choice of positive $\gh_\C$--root 
system $\Sigma^+$ and a subset $\Phi$ of the corresponding simple root system,
such that $\Phi$ is the simple root system for $\gu_\C$.  For each 
$[\mu]\in \widehat{U}$ we have a $G$--homogeneous hermitian holomorphic
vector bundle $\E \to Y$.  Let $q \geqq 0$.  Then we have the Hilbert space
$\cH_2^{0,q}(Y;\E)$ of $\E$--valued square integrable harmonic $(0,q)$--forms 
on $Y$, and $G$ acts on $\cH_2^{0,q}(Y;\E)$ by a unitary representation
$\pi_\mu^q$\,.  Write $\Theta_{\pi_\mu^q,disc}$ for the sum of the distribution
characters of the irreducible summands of $\pi_\mu^q$\,.  Let $\rho$ denote
half the sum of the positive roots.  Thus $\pi_\mu^q = \pi_{\chi,\lambda+\rho}$
in the notation of \S \ref{sec3} where $[\chi] \in \widehat{Z_G(G^0)}_\zeta$ and
$\lambda + \rho \in i\gh^*$ integrates to a unitary character on $H^0$
that agrees with $\zeta$ on $Z_{G^0}$.  We can (and do) arrange this in such 
a way that $\lambda$ is $\gu$--dominant, i.e. $\langle \lambda, \phi \rangle
\geqq 0$ for every $\phi \in \Phi$.  Note 
$U = Z_G(G^0)U^0$ so $[\mu] = [\chi \otimes \mu^0]$.
$$
q(\lambda + \rho) = \#\{\alpha \text{ compact } \mid
	\langle \lambda + \rho,\alpha \rangle < 0\} +
	\#\{\alpha \text{ noncompact } \mid
        \langle \lambda + \rho,\alpha \rangle > 0\}.
$$
We prove

(i) $\sum_{q \geqq 0} (-1)^q \Theta_{\pi_\mu^q,disc} = 
	(-1)^{|\Sigma^+| + q(\lambda+\rho)} \Theta_{\chi,\lambda+\chi}$\,,

(ii) $\text{If } q \ne q(\lambda + \rho) \text{ then } 
	\cH_2^{0,q}(Y;\E) = 0$\,, and 
(iii) $[\pi_\mu^{q(\lambda +\rho)}]=[\pi_{\chi,\lambda+\rho}]\in\widehat{G}_{disc}$\,.

\noindent
This improvement over \cite[\S 7]{W1973} is mostly due to the improvement
of \cite{S1976} over \cite{S1971}, but it still relies on the Plancherel 
formula.

In \S \ref{sec8} we work out the geometric realization for all standard tempered
representations.  Fix a Cartan subgroup $H = T \times A$ and a corresponding
cuspidal parabolic subgroup $P = MAN \subset G$.  The $H$--series classes
are realized over measurable orbits $Y = G(x) \subset \overline{G}_\C/Q = X$
such that (i) The $G$--normalizer $N_{[x]}$ of the holomorphic arc
component $S_{[x]}$ through $x$ is an open subgroup of $P$ and (ii) 
$T \subset U = \{m \in M \mid m(x) = x\}$ with $U/Z$ compact.  (Such pairs 
always exist and were classified at the end of \S \ref{sec6}.)  Then
$G$ has isotropy subgroup $UAN$ at $x$, $N_{[x]} = M^\dagger AN$ where
$M^\dagger = Z_M(M^0)M^0$\,, and $U = Z_M(M^0)U^0$ with $U \cap M^0 = U^0$.

Let $\rho_\ga = \frac{1}{2}\sum_{\phi \in \Sigma_\ga^+} (\dim \gg^\phi)\phi$
where $\gm = \sum_{\phi \in \Sigma_\ga^+} \gg^{-\phi}$.  Given $[\mu] \in
\widehat{U}$ and $\sigma \in \ga^*$ we have the $G$--homogeneous vector
bundle $\E_{\mu,\sigma} \to G/UAN = Y$ associated to 
$[\mu \otimes e^{\rho_\ga + i\sigma}] \in \widehat{UAN}$.  $\E_{\mu,\sigma}$
carries a $K$--invariant hermitian metric and it is holomorphic over every
holomorphic arc component $S_{[x]}$ of $Y$.  That leads to the Hilbert spaces
$\cH_2^{0,q}(Y;\E_{\mu,\sigma})$ of all square integrable partially harmonic
$\E_{\mu,\sigma}$--valued $(0,q)$--forms on $Y$.  $G$ acts there by a unitary
representation $\pi_{\mu,\sigma}^q$\,, and we prove $[\pi_{\mu,\sigma}^q]
= [\Ind_{N_{[x]}}^{G} (\eta_\mu^q \otimes e^{i\sigma})]$ where $\eta_\mu^q$
is the representation of $M^\dagger$ on 
$\cH_2^{0,q}(S_{[x]};\E_{\mu,\sigma}|_{S_{[x]}})$.

Decompose $[\mu] = [\chi \otimes \mu^0]$ where $[\chi] \in \widehat{Z_M(M^0)}$.
Let $\nu$ be the highest weight of $[\mu^0]$.  Then $\chi = e^\nu$ on the
center of $M^0$.  Our main result, Theorem \ref{8.3.4}, is
\medskip
{\em

(i) The $H$--series constituents of $\pi_{\mu,\sigma}^q$ are just its 
irreducible subrepresentations, and $\pi_{\mu,\sigma}^q$
has well defined distribution character $\Theta_{{\pi_{\mu,\sigma}^q}}$\,.
\smallskip

(ii) $\sum_{q \geqq 0} (-1)^q \Theta_{{\pi_{\mu,\sigma}^q}} =
	(-1)^{|\Sigma_\gt ^+|+q_M(\nu + \rho_\gt)}
	\Theta_{\pi_{\chi, \nu+\rho_\gt ,\sigma}}$
where $[\pi_{\chi, \nu+\rho_\gt ,\sigma}]$ is the $H$--series class as
denoted in \S \ref{sec4}.
\smallskip

(iii) If $q \ne q_M(\nu + \rho_\gt)$ then $\cH_2^{0,q}(Y; \E_{\mu,\sigma}) = 0$
\smallskip

(iv) $[\pi_{\mu,\sigma}^{q_M(\nu + \rho_\gt)}]$ is the $H$--series
representation class $[\pi_{\chi,\nu+\rho_\gt,\sigma}]$.
}
\medskip

\noindent
That gives explicit geometric realizations for the $H$--series classes of
unitary representations of $G$.  The improvement over \cite[\S 8]{W1973}
comes from the improvement of \S \ref{sec7} over \cite[\S 7]{W1973}.

\section{The Principal Series}\label{sec1}
\setcounter{equation}{0}
The first example of geometric realization for general real reductive 
Lie groups is given by the
``principal series'', combining the Bott--Borel--Weil Theorem with Mackey's
theory of unitary induction.  

\subsection{}\label{ssec1a}  
\setcounter{equation}{0}
Let $M$ be a reductive Lie group, i.e.
$\gm = \gz_\gm \oplus [\gm,\gm]$ where $\gz_\gm$ is the center of $\gm$.
We assume that $M$ has a closed normal abelian subgroup $Z$ such that 
(i) $M/Z$ is compact and (ii) if $m \in M$ then $\Ad(m)$ is an inner 
automorphism on $\gg_\C$\,.  From (ii), $ZM^0$ has finite index in $M$
and $Z\cap M^0$ is co-compact in the center $Z_{M^0}$ of $M^0$.

As usual, if $\eta$ is a unitary representation of $M$ then $[\eta]$ is
its unitary equivalence class.  The set of all such equivalence classes
$[\eta]$, with $\eta$ irreducible,
is the {\em unitary dual} $\widehat{M}$.  If $E$ is a close central subgroup
of $M$ and $\xi \in \widehat{E}$ then $\widehat{M}_\xi
:= \{[\eta] \in \widehat{M} \mid \eta|_E \text{ is a multiple of } \xi\}$.

\begin{proposition} \label{1.1.3}
Let $Z_M(M^0)$ denote the centralizer of $M^0$ in $M$.  Then

\noindent
{\rm 1.} $Z_M(M^0)\cap M^0 = Z_{M^0}$ and $M = Z_M(M^0)M^0$.

\noindent
{\rm 2.} If $[\eta] \in \widehat{M}$ then $\eta$ is finite dimensional.

\noindent
{\rm 3.} If $[\eta] \in \widehat{M}$ then there exist unique 
$[\xi] \in \widehat{Z_{M^0}}$, $[\chi] \in \widehat{Z_M(M^0)}_\xi$ and
$[\eta^0] \in (\widehat{M^0})_\xi$ such that $[\eta] = [\chi \otimes \eta^0]$.

\noindent
{\rm 4.} Let $\gt$ be a Cartan subalgebra of $\gm$, $\Sigma^+_\gt$ a positive
root system, $\rho_\gt$ half the sum of the positive roots, $T^0 = \exp(\gt)$,
and 
$$
L_\gm^+ = \{\nu \in i\gt^* \mid e^{\nu - \rho_\gt} \in \widehat{T^0}
\text{ is well defined and } \langle \nu, \phi \rangle > 0 
 \text{ for every } \phi \in \Sigma_\gt^+\}.
$$
Then there is a bijection $\nu \mapsto [\eta_\nu^0]$
of $L_\gm^+$ onto $\widehat{M^0}$ where $\nu - \rho_\gt$ is the highest
weight of $\eta_\nu^0$\,.  Further, $[\eta_\nu^0] \in \widehat{(M^0)}_\xi$
where $\xi = e^{\nu - \rho_\gt}|_{Z_{M^0}}$\,.
\end{proposition}

\begin{proof} If $m \in M$ then $\Ad(m)$ is an inner automorphism on $M^0$,
so  $M = Z_M(M^0)M^0$\,, in particular $Z_{M^0}$ is central in $M$.

Let $[\eta] \in \widehat{M}$.  Then $\eta|_{Z_{M^0}}$ is a multiple of a
unitary character $\xi$ on $Z_{M^0}$\,.  As $M^0/Z_{M^0}$ and 
$Z_M(M^0)/Z_{M^0}$ are compact, $M^0$ and $Z_M(M^0)$ are of type I with
every irreducible representation finite dimensional; see \S \ref{ssec2d}
below.  Now $[\eta|_{Z_M(M^0)}]$ is a multiple of a finite dimensional class
$[\chi] \in \widehat{Z_M(M^0)}_\xi$ and $[\eta|_{M^0}]$ is a multiple
of a finite dimensional class $[\eta^0] \in \widetilde{M^0}_\xi$\,.
Assertions (2) and (3) follow, and assertion (4) boils down to the highest
weight theory of $\gm$.
\end{proof}

The Bott--Borel--Weil Theorem extends from compact connected Lie groups to 
give geometric realizations of the classes in $\widehat{M}$.

The homogeneous K\"ahler manifolds of $M$ are the manifolds $S_\Phi$ given by
\begin{equation}\label{1.1.4}
\begin{aligned}
& \Pi_\gt: \text{ simple $\gt_\C$--root system on 
	$\gm_C$ for $\Sigma^+_\gt$\,,} \\
& \Phi: \text{ arbitrary subset of } \Pi_\gt\,\, \\
& \gz_\Phi = \{x \in \gt \mid \Phi(x) = 0\} \text{ and }
	Z_\Phi^0 = \exp(\gz_\Phi) \subset T^0 \\
& U_\Phi \text{ is the $M$--centralizer of 
	$Z_\Phi^0$ and } S_\Phi = M/U_\Phi \\
\end{aligned}
\end{equation}
Note that $\Phi$ is a simple $\gt_\C$--root system for $\gu_\Phi$\,.
Using Proposition \ref{1.1.3}(1),  construct
\begin{equation}\label{1.1.6}
\begin{aligned}
& \overline{M}:= M/Z_M(M^0) = M^0/Z_{M^0} \text{ compact connected Lie group.}\\
& \overline{T}:= T/Z_M(M^0) = T^0/Z_{M^0} 
	\text{ maximal torus in }\overline{M}.\\
\end{aligned}
\end{equation}
This leads to the homogeneous K\"ahler structure on $S_\Phi$.  $\overline{M}$
has complexification $\overline{M}_\C = \Int(\gm_\C)$, inner automorphism 
group of $\gm_\C$\,.  The group $\overline{U}_\Phi := U_\Phi/Z_M(M^0)$ is
connected, and we define
\begin{equation}\label{1.1.7}
\begin{aligned}
&\gr_\Phi := \overline{\gu}_{\Phi,\C} + 
	{\sum}_{\phi\in\Sigma_\gt^+} \overline{\gm}^{-\phi}
	\text{ subalgebra of } \overline{\gm}_\C\,, \\
& R_\Phi \text{ is the complex analytic subgroup of } \overline{M}_\C
	\text{ for } \gr_\Phi\,.
\end{aligned}
\end{equation}
Then $R_\Phi$ is closed in $\overline{M}_\C$\,, and $M$ acts on 
$\overline{M}_\C/R_\Phi$ by the projection $m \mapsto \overline{m}$ of $M$
onto $\overline{M}$.  The orbit $M(1R_\Phi)$ is closed because $\overline{M}$
is compact, and one checks that it has the same real dimension as
$\overline{M}_\C/R_\Phi$\,.  Thus $M$ is transitive on 
$\overline{M}_\C/R_\Phi$\,, and $mU_\Phi \mapsto \overline{m}R_\Phi$ is a
covering space projection.  But $\overline{M}_\C/R_\Phi$ is simply connected.

\begin{lemma}\label{1.1.8}
The map $mU_\Phi \mapsto \overline{m}R_\Phi$ is an $M$--equivariant 
bijection of $S_\Phi = M/U_\Phi$ onto $\overline{M}_\C/R_\Phi$\,.
\end{lemma}
Now the complex presentation $S_\Phi = \overline{M}_\C/R_\Phi$ defines
an $M$--homogeneous complex structure on $S_\Phi$\,, and the coboundary
of any $\gu_\Phi$--regular element of $\gz_\Phi^*$ is an $M$--homogeneous
K\"ahler metric on $S_\Phi$\,.

The irreducible $M$--homogeneous holomorphic vector bundles $\E_\mu \to S_\Phi$
are constructed as follows.  Let $[\mu] \in \widehat{U}_\Phi$ and let $E_\mu$
denote its representation space.  Let $\E_\mu \to S_\Phi$ denote the
associated complex vector bundle.  The group $M^{(1)}
= [M^0,M^0]$ is compact, connected and semisimple, so the projection.
$M \to \overline{M}$ restricts to a finite covering $M^{(1)} \to \overline{M}$,
and that complexifies to a finite holomorphic covering 
$p: M^{(1)}_\C \to \overline{M}_\C$\,.  Denote $R_\Phi^{(1)} = p^{-1}(R_\Phi)$.
It is connected.  Now $\mu^{(1)}$ has a unique completely reducible
holomorphic extension $\mu^{(1)}$ to $R_\Phi^{(1)}$\,.  That defines 
an $M_\C^{(1)}$--homogeneous holomorphic vector bundle
$\E_\mu = \E_{\mu_\C^{(1)}} \to M_\C^{(1)}/R_\Phi^{(1)} = \overline{M}_\C/R_\Phi
= S_\Phi$\,.  That bundle structure is stable under the action of $M$, proving
\begin{lemma} \label{1.1.9}
$\E_\mu \to S_\Phi$ is an $M$--homogeneous holomorphic vector bundle.
\end{lemma}
The sheaf $\cO(\E_\mu) \to S_\Phi$ of germs of holomorphic 
sections of $\E_\mu \to S_\Phi$ defines
\begin{equation}\label{1.1.10}
\begin{aligned}
& H_2^{0,q}(S_\Phi,\E_\mu): L_2 \text{ harmonic $(0,q)$--forms on $S_\Phi$ 
	with values in $\E_\Phi$}\,,\\ 
& \eta_\mu^q: \text{ representation of $M$ on }
	H^q(S_\Phi;\cO(\E_\mu)) \text{ and on } H_2^{0,q}(S_\Phi,\E_\mu)
\end{aligned}
\end{equation}
Simple connectivity of $S_\Phi$ implies $U_\Phi \cap M_\Phi^0 = U_\Phi^0$\,.
With Proposition \ref{1.1.3},
\begin{equation}\label{1.1.11}
\begin{aligned}
&\widehat{U_\Phi} = {\bigcup}_{\xi \in \widehat{Z_{M^0}}}
	(\widehat{U_\Phi})_\xi \text{ where } \\
&\widehat{(U_\Phi)}_\xi = \{[\chi\otimes\mu^0] \mid 
	[\chi] \in \widehat{Z_M(M^0)}_\xi \text{ and } 
	[\mu^0] \in  \widehat{(U_\Phi^0)}_\xi\}.
\end{aligned}
\end{equation}
The Bott--Borel--Weil Theorem \cite{B1957} extends to our setting as follows.

\begin{proposition}\label{1.1.12}
Let $[\mu] = [\chi\otimes\mu^0] \in \widehat{U}_\Phi$ as in 
{\rm (\ref{1.1.11})} and let $\beta$ be the highest weight of $\mu^0$.

\noindent
{\rm 1.} If $\langle \beta + \rho_\gt,\phi \rangle = 0$ for some 
$\phi \in \Sigma_\gt^+$ then $H^q(S_\Phi;\cO(\E_\mu)) = 0$ for all integers $q$.

\noindent
{\rm 2.} If $\langle \beta + \rho_\gt,\phi \rangle \ne 0$ for every
$\phi \in \Sigma_\gt^+$ define $q_0$ to be the number of roots $\phi \in 
\Sigma_\gt^+$ for which $\langle \beta + \rho_\gt,\phi \rangle < 0$, and
let $\nu \in L_\gm^+$ that is $W(M^0,T^0)$--conjugate to $\beta + \rho_\gt$\,.
Then $[\eta_\mu^{q_0}] = [\chi\otimes \eta_\nu^0]$ and
$H^q(S_\Phi;\cO(\E_\mu)) = 0$ for every integer $q \ne q_0$\,.
\end{proposition}

\begin{proof} $[\mu] = [\chi\otimes\mu^0]$ has representation space
$E_\mu = E_\chi \otimes E_{\mu^0}$.  As $Z_M(M^0)$ acts trivially on $S_\Phi$ 
the associated bundle $\E_\mu = E_\chi \otimes \E_{\mu^0}$\,.  This reduces 
the proof to the case where $M$ is connected.

Now that $M$ is connected, $M = Z_M^0\,M^{(1)}$ where $M^{(1)} = [M,M]$
is compact, connected and semisimple and where $F:= Z_M^0\cap M^{(1)}$
is finite.  Now $U_\Phi = Z_M^0\, U_\Phi^{(1)}$ where $U_\Phi^{(1)} =
U_\Phi\cap M^{(1)}$ and $Z_M^0\cap U^{(1)} = F$.  Split 
$[\mu] = [\varepsilon \otimes \mu^{(1)}]$  where $[\varepsilon] \in 
\widehat{Z_M^0}$ and $[\mu^{(1)}] \in \widehat{U_\Phi^{(1)}}$ give the same
character $\varepsilon|_F$ on $F$.  As above 
$[\eta_\mu^q] = [\varepsilon \otimes \eta_{\mu^{(1)}}^q]$ where
$\eta_{\mu^{(1)}}^q$ induces the representation of $M^{(1)}$ on
$H^q(S_\Phi;\cO(\E_\mu))$.  We have reduced the proof to the case where
$M$ is compact, connected and semisimple, which is \cite{B1957}.
\end{proof}

\subsection{}\label{ssec1b} 
\setcounter{equation}{0}
Let $G$ be a general real reductive Lie group.  So
$\gg = \gc \oplus [\gg,\gg]$ where $\gc$ is the center and the derived
algebra $[\gg,\gg]$ is semisimple.  We recall the global conditions:
\begin{equation}\label{1.2.2}
\begin{aligned}
&{\rm (i) }\text{ if } m \in M \text{ then } \Ad(g) 
	\text{ is an inner automorphism on } \gg_\C\,, \text{ and }\\
&{\rm (ii)}\quad G \text{ has a closed normal abelian subgroup } Z \text{ such that}\\
&\phantom{X} Z \text{ centralizes } G^0\,, |G/ZG^0| < \infty, \text{ and }
	Z\cap G^0 \text{ is co-compact in } Z_{G^0}
\end{aligned}
\end{equation}
Fix a Cartan involution $\theta$ of $G$.  Its fixed point set $K = G^\theta$
contains $Z_G(G^0)$ and
\begin{equation}\label{1.3.2}
K/Z_G(G^0) \text{ is a maximal compact subgroup of } G/Z_G(G^0)\,.
\end{equation}
Let $\ga$ be a maximal abelian subspace of $\{x \in \gg\mid \theta(x) = -x\}$.
Any two choices are $\Ad_G(K)$--conjugate, $\Ad(\ga)$ is diagonalizable on
$\gg$, and $\gg$ is the direct sum
the joint eigenspaces $\gg^\phi = \{x \in \gg \mid [a,x] = \phi(a)x
\text{ for all } a \in \ga\}$.   The {\em $\ga$--root system} of $\gg$ is
$\Sigma_\ga := \{\text{joint eigenvalues } \phi \mid \gg^\phi \ne 0\}$.
Choose a positive subsystem $\Sigma_\ga^+$; any two choices are conjugate
by an element of the $K$--normalizer of $\ga$.  The pair $(\ga,\Sigma_\ga^+)$
specifies
\begin{equation}\label{1.2.5}
\begin{aligned}
&N: \text{ analytic subgroup of $G$ for } \gn = 
	{\sum}_{\alpha \in \Sigma_\ga^+} \gg^{-\alpha}\,, \\
&A = \exp(\ga), M = Z_K(A), \text{ and }B = \{g \in G \mid Ad(g)N = N\}.
\end{aligned}
\end{equation}
Then $MA = M\times A = Z_G(A)$, $B = MAN$ is a {\em minimal parabolic} 
subgroup of $G$ and $G = KAN$ is the {\em Iwasawa decomposition}.  Further, 
$M$ and $MA$ satisfy (\ref{1.2.2}).

$G$ has a Cartan subgroup $H = T \times A$ where $T$ is a Cartan subgroup of
$M$.  The corresponding positive root system $\Sigma^+$ satisfies 
$\Sigma_\ga^+ = \{\gamma|_\ga \mid \gamma \in \Sigma^+
\text{ and } \gamma|_\ga \ne 0\}$ and $\Sigma_\gt^+ =
\{\gamma|_\gt \mid \gamma \in \Sigma^+ \text{ and } \gamma|_\ga = 0\}$.

Let $[\chi \otimes \eta_\nu^0] \in \widehat{M}$ and $\sigma \in \ga^*$.
Together they specify 
\begin{equation}
[\alpha_{\chi,\nu,\sigma}] \in \widehat{B} \text{ by }
	\alpha_{\chi,\nu,\sigma}(man) = 
	(\chi \otimes \eta_\nu^0)(m)e^{i\sigma}(a).
\end{equation}
The corresponding {\em principal series representation} of $G$ is
\begin{equation}\label{1.2.8}
\pi_{\chi,\nu,\sigma} = \Ind_B^G(\alpha_{\chi,\nu,\sigma}), 
	\text{ unitarily induced representation.}
\end{equation}

We construct partially holomorphic cohomology spaces $Y_\Phi$ and use them
for geometric constructions of the principal series representations
$[\pi_{\chi,\nu,\sigma}]$.  Let $\Pi$ be the simple $(\gt + \ga)$ root
system for $\Sigma^+$  Then $\Phi \subset \Pi_\gt \subset \Pi$.  Denote
\begin{equation}\label{1.2.9}
\overline{G} = G/Z_G(G^0), \text{ so } \overline{G}_\C = \Int(\gg_\C)
\text{ and } \overline{\gg} = \gg/\gc.
\end{equation} 
$\gq_\Phi = (\gu_\Phi/\gc)+ \ga_\C +{\sum}_{\gamma \in \Sigma^+}\,\gg^{-\gamma},
\text{ is a parabolic subalgebra of $\overline{\gg}_\C$\,.}$
Let $Q_\Phi$ denote the corresponding parabolic subgroup of $\overline{G}_\C$\,;
it is the $\overline{G}_\C$--normalizer of $N$. 
That defines the complex homogeneous projective variety,
\begin{equation}
X_\Phi = \overline{G}_\C/Q_\Phi\,,  \text{ complex flag manifold.} 
\end{equation}
As $Q_\Phi$ is its own normalizer in $\overline{G}_\C$ we can identify $X_\Phi$
with the set of all $\overline{G}_\C$--conjugates of $Q_\Phi$\,.
As discussed in \S 1.1, $G$ acts on $X_\Phi$ by holomorphic diffeomorphisms,
through the homomorphism $G \to \overline{G}$.  Denote 
$$
x_\Phi = 1\cdot Q_\Phi \in X_\Phi \text{ and } Y_\Phi = G(x_\Phi).
$$
Then, as in Lemma \ref{1.1.8}, 
\begin{equation}\label{1.2.10}
U_\Phi AN = \{g \in G \mid g(x_\Phi) = x_\Phi\}, \text{ so } Y_\Phi = U_\Phi AN.
\end{equation}
In particular $Y_\Phi$ contains
$$
S_\Phi := M(x_\Phi) = (MAN)(x_\Phi), 
	\text{ compact complex submanifold of } X_\Phi\,.
$$
Since the minimal parabolic subgroup $B$ is its own $G$--normalizer one
can look on the Lie algebra level to see that $B = \{g \in G \mid gS_\Phi
= S_\Phi\}$.
\begin{lemma}\label{1.2.12}
If $g \in G$ then $gS_\Phi$ is a complex submanifold of $X_\Phi$ contained
in $Y_\Phi$\,.  If $S \subset Y_\Phi$ is a connected  complex submanifold of
$X_\Phi$ then $S$ is contained in one of the $gS_\Phi$\,,
$Y_\Phi \to G/B = K/M$
is a well defined equivariant fibration, and $gS_\Phi$ is the fiber over $gB$.
\end{lemma}

\begin{proof} In the terminology of \S \ref{ssec6d} the topological 
component of $x_\Phi$ in $S_\Phi$ is the holomorphic arc component of 
$Y_\Phi$ through $x_\Phi$\,.  Let $\tau$ denote complex conjugation.  Then
$\gq_\Phi + \tau\gq_\Phi = (\gm/\gc)_\C + \ga_\C + \gn_\C = (\gb/\gc)_\C$
subalgebra of $\overline{\gg}_\C$\,.  In other words $Y_\Phi$ is integrable
in the sense of \S \ref{ssec6d}.  Also, $\gq_\Phi$ has $\tau$--stable 
Levi component $(\gu_\Phi/\gc)_\C$\,.  Now the assertion is a special 
case of some results from \S \ref{ssec6d}.
\end{proof}

We now construct the partially holomorphic bundles and cohomology spaces for
the principal series.  Let $[\mu] \in \widehat{U_\Phi}$, $\sigma \in \ga^*$
and $\rho_\ga = \tfrac{1}{2}\sum_{\phi \in \Sigma_\ga^+} 
(\dim \gg^{-\phi})\phi \in \ga^*$.  Then we have a representation 
$\gamma_{\mu,\sigma}$ of $U_\Phi AN$ on $E_\mu$\,, and the associated
$G$--homogeneous vector bundle, given by
\begin{equation}\label{1.2.13}
\gamma_{\mu,\sigma}(uan) = e^{\rho_\ga +i\sigma}(a)\mu(u) \text{ and }
\E_{\mu,\sigma} \to G/U_\Phi AN = Y_\Phi\,.
\end{equation}
Note that $\E_{\mu,\sigma}|_{S_\Phi} = \E_{\mu}$\,.  If $g \in G$ then
$\E_{\mu,\sigma}|_{gS_\Phi} \to gS_\Phi$ is an $\Ad(g)B$--homogeneous 
holomorphic vector bundle.  As $[\mu]$ is unitary and $K$ permutes
$\{gS_\Phi\}$ transitively, $\E_{\mu,\sigma}$ carries a $K$--invariant 
hermitian metric, which we will use without comment.

If $y \in Y_\Phi$\,, say $y = g(x_\Phi)$\,, we view the holomorphic tangent
space to $gS_\Phi$ at $g(x_\Phi)$ as a subspace $T_{g(x_\Phi)}$ of the
complexified tangent space of $Y_\Phi$\,.  These subspaces define a sub-bundle
$\T \to Y_\Phi$ of the complexified tangent bundle of $Y_\Phi$\,.  Note
that $\T \to Y_\Phi$ is $G$--homogeneous and is holomorphic over every
$gS_\Phi$\,.  Then
\begin{equation}\label{1.2.15}
A^{0,q}(Y_\Phi,\E_{\mu,\sigma}) = \{C^\infty \text{ sections of }
	\E_{\mu,\sigma}\otimes \Lambda^q(\overline{\T}^*) \to Y_\Phi\}
\end{equation}
is the space of {\em $C^\infty$ partially $(0,q)$--forms on $Y_\Phi$ with values
in $\E_{\mu,\sigma}$}.  From any $K$--invariant hermitian metric on $\T$,
thus also on $\E_{\mu,\sigma}\otimes \Lambda^q(\overline{\T}^*)$, we
have Kodaira--Hodge operators
\begin{equation}
\Lambda^{0,q}(Y_\Phi,\E_{\mu,\sigma}) \overset{\#}{\longrightarrow}
	\Lambda^{n,n-q}(Y_\Phi,\E_{\mu,\sigma}^*) 
	\overset{\widetilde{\#}}{\longrightarrow}
	\Lambda^{0,q}(Y_\Phi,\E_{\mu,\sigma})
\end{equation}
where $n = \dim_\C S_\Phi$\,.  The $\overline{\partial}$ operator of
$X_\Phi$ induces the $\overline{\partial}$ on each $gS_\Phi$.  They
fit together to define 
operators $\overline{\partial}: A^{0,q}(Y_\Phi,\E_{\mu,\sigma})
\to A^{0,q+1}(Y_\Phi,\E_{\mu,\sigma})$.  Each $A^{0,q}(Y_\Phi,\E_{\mu,\sigma})$
is a pre--Hilbert space with inner product
\begin{equation}
\langle \alpha, \beta \rangle = 
  \int_{K/M}d(kM)\int_{kS_\Phi} \alpha\, \bar\wedge \, \#\beta
\end{equation}
where $\bar\wedge$ signifies exterior product followed by contraction
$E_\mu \otimes E_\mu^* \to \C$.  Define
\begin{equation}\label{1.2.16}
L_2^{0,q}(Y_\Phi,\E_{\mu,\sigma}):\text{\, Hilbert space completion of }
A^{0,q}(Y_\Phi,\E_{\mu,\sigma}).
\end{equation}
Then $\overline{\partial}$ has formal adjoint $\overline{\partial}^* =
-\widetilde{\#} \overline{\partial} \#$ there, and that leads to the
essentially self adjoint {\rm partial Kodaira--Hodge operators}
\begin{equation}\label{1.2.17}
\square = (\overline{\partial} + \overline{\partial}^*)^2 
  = \overline{\partial} \overline{\partial}^* 
	+ \overline{\partial}^* \overline{\partial}
\text{ on } L_2^{0,q}(Y_\Phi,\E_{\mu,\sigma}).
\end{equation}
The {\em partially harmonic $(0,q)$--forms on $Y_\Phi$ with values in 
$\E_{\mu,\sigma}$} are the elements of
\begin{equation}\label{1.2.18}
H^{0,q}(Y_\Phi,\E_{\mu,\sigma}) = \{\omega \in L_2^{0,q}(Y_\Phi,\E_{\mu,\sigma})
	\mid \square\,\omega = 0\}.
\end{equation}
Note that $H^{0,q}(Y_\Phi,\E_{\mu,\sigma})$ consists of all Borel--measurable
sections $\omega$ of 
$\E_{\mu,\sigma}\otimes \Lambda^q(\overline{\T}^*) \to Y_\Phi$ such that
(i) for almost all $k \in K$, $\omega|_{kS_\Phi}$ is harmonic in the ordinary
sense and (ii) the $L_2$ norms satisfy 
$\int_{K/M} ||\omega|_{kS_\Phi}^2 d(kM) < \infty$.  In particular
$H^{0,q}(Y_\Phi,\E_{\mu,\sigma})$ is a closed subspace of 
$L_2^{0,q}(Y_\Phi,\E_{\mu,\sigma})$ and the natural action of $G$ on
$H^{0,q}(Y_\Phi,\E_{\mu,\sigma})$ is a continuous representation
$\pi_{\mu,\sigma}^q$\,.

\begin{theorem}\label{1.2.19}
The representation $\pi_{\mu,\sigma}^q$ of $G$ on
$H^{0,q}(Y_\Phi,\E_{\mu,\sigma})$ is unitary.  Let 
$[\mu] = [\chi\otimes \mu^0]$ 
as in {\rm (\ref{1.1.11})} and let $\beta$ be the highest weight of $\mu^0$\,.

{\rm 1.}\phantom{X} If $\langle \beta + \rho_\gt, \phi \rangle = 0$ for some
$\phi \in \Sigma_\gt^+$ then $H^{0,q}(Y_\Phi,\E_{\mu,\sigma}) = 0$ for all $q$.

{\rm 2.}\phantom{X} If $\langle \beta + \rho_\gt, \phi \rangle \ne 0$ for
every $\phi \in \Sigma_\gt^+$\,, let $q_0$ be the number of roots
$\phi \in \Sigma_\gt^+$ for which $\langle \beta + \rho_\gt, \phi \rangle < 0$,
and let $\nu$ be the unique element of $L_\gm^+$ conjugate to 
$\beta + \rho_\gt$ by an element of the Weyl group $W(M^0,T^0)$.  Then
$
[\pi_{\mu,\sigma}^{q_0}] = [\pi_{\chi,\nu,\sigma}], \text{ principal
	series class, }
$
and $H^{0,q}(Y_\Phi,\E_{\mu,\sigma}) = 0$ for every $q \ne q_0$\,.

{\rm 3.}\phantom{X} In particular, given a principal series class
$[\pi_{\chi,\nu,\sigma}]$, we can realize it on 
$H^{0,0}(Y_\Phi,\E_{\mu,\sigma})$ where $\mu = [\chi \otimes \mu^0]$ and
$\mu^0$ has highest weight $\nu - \rho_\gt$\.,.
\end{theorem}

\begin{proof} Let $\tilde \pi_{\mu,\sigma}^{q}$ denote the representation
of $G$ on $L_2^{0,q}(Y_\Phi,\E_{\mu,\sigma})$.  Factor $\gamma_{\mu,\sigma} = 
{}' \gamma_{\mu,\sigma}\cdot e^{\rho_\ga}$ where
${}' \gamma_{\mu,\sigma}(uan) = e^{i\sigma}(a)\mu(u)$ is unitary and
$e^{\rho_a}$ compensates non--unimodularity of $U_\Phi AN$\,.  
Thus $\tilde \pi_{\mu,\sigma}^{q} = \Ind_{U_\Phi AN}^G(\gamma_{\mu,\sigma})$
is unitary, so its subrepresentation 
$\pi_{\mu,\sigma}^{q} = \Ind_{U_\Phi AN}^G(\gamma_{\mu,\sigma})$ is unitary.

Write $E_\mu^q$ for $E_\mu\otimes\Lambda^q(\overline{T}_{x_\Phi}^*)$,
so $L_2^{0,q}(Y_\Phi,\E_{\mu,\sigma})$ is the space of all measurable
$f:G \to E_\mu^q$ such that $f(guan)$ = $(\gamma_{\mu,\sigma}(uan)^{-1}
\otimes \Lambda^q\Ad(uan)_\gn^{-1})f(g)$ and 
$\int_{K/U_\Phi}||f(k)||^2 d(kU_\Phi) < \infty$.  Note 
$\tilde \pi_{\mu,\sigma}^{q} = \Ind_B^G(\tilde \psi)$ where $\tilde \psi
= \Ind{U_\Phi AN}^B('\gamma_{\mu,\sigma})$.  The representation space
elements annihilated by the partial Laplacian $\square$ correspond to the
subspace of the representation space of $\tilde \psi$ annihilated by the
full Laplacian of $\E_{\mu,\sigma}|_{S_\Phi}$\,, thus
$\pi_{\mu,\sigma}^q = \Ind_B^G(\psi)$ where $\psi$ represents $B$ on
$H^{0,q}(S_\Phi,\E_{\mu,\sigma}|_{S_\Phi})$.  Denote the
representation of $M$ on $H^{0,q}(S_\Phi,\E_\mu)$ by $\eta_\mu^q$\,.  Now
$\psi(man) = \eta_\mu^q(m)e^{i\sigma}(a)$.  The theorem now follows from
Proposition \ref{1.1.12}.
\end{proof}

\section{General Notion of Relative Discrete Series}\label{sec2}
\setcounter{equation}{0}

We recall some basic facts on the relative discrete series of a locally compact
group $G$.  Left Haar measure is denoted $dg$, and the left regular
representation of $G$ on $L_2(G)$ is $\ell(g)f](x) = f(g^{-1}x)$.  If $\pi$
is an irreducible unitary representation of $G$ then
$\cH_\pi$ is its representation space and $[\pi]$ its unitary equivalence class.
Its coefficients are the functions $\phi_{u,v}(g) = \langle u, \pi(g)v\rangle$.
The set of all equivalence classes of irreducible unitary representations
is $\widehat{G}$.

\subsection{}\label{ssec2a}  
\setcounter{equation}{0}
Let $Z$ be a closed normal abelian subgroup of
$G$. It has left regular representation 
$\ell^Z = \int_{\widehat{Z}}\zeta\, d\zeta$.  The left regular
representation of $G$ decomposes as
\begin{equation}\label{2.1.1}
\ell = \Ind_{\{1\}}^G(1) = \Ind_Z^G(\ell_Z) = 
  \int_{\widehat{Z}} \ell_\zeta d\zeta = 
  \int_{\widehat{Z}} \Ind_Z^G(\zeta)d\zeta
\end{equation}
where $\ell_\zeta := \Ind_Z^G(\zeta)$ is the left regular representation of
$G$ on the Hilbert space
{\footnotesize
\begin{equation}\label{2.1.2}
L_2(G/Z,\zeta) = \{f:G \to \C \mid f(gz) = \zeta(z)^{-1}f(g) \text{ and }
	\int_{G/Z} |f(g)|^2 dg < \infty \}.
\end{equation}
}
Thus $L_2(G) = \int_{\widehat{Z}} L_2(G/Z,\zeta)$.

Given $\zeta \in \widehat{Z}$ let $\widehat{G}_\zeta = \{[\pi] \in
\widehat{G} \mid \zeta \text{ is a subrepresentation of } \pi|_Z\}$.
Thus $\ell_\zeta$ is a direct integral over $\widehat{G}_\zeta$.  We
say that $[\pi] \in \widehat{G}$ is {\em $\zeta$--discrete} if $\pi$
is a subrepresentation of $\ell_\zeta$\,.  The $\zeta$--discrete
classes form the {\em $\zeta$--discrete series} $\widehat{G}_{\zeta-disc}$\,.
The {\rm relative (to $Z$) discrete series} is 
$\widehat{G}_{disc} = \bigcup \widehat{G}_{\zeta-disc}$\,.

\subsection{}\label{ssec2b}
\setcounter{equation}{0}
Let the closed normal abelian subgroup $Z$ be central in $G$.
If $[\pi] \in \widehat{G}$ there is a character $\zeta_\pi \in \widehat{Z}$
such that $\pi|_Z$ is a multiple of $\zeta_\pi$\,.  So
$\widehat{G}_{disc} = \bigcup \widehat{G}_{\zeta-disc}$ is disjoint.

If $[\pi] \in \widehat{G}$ and $\zeta \in \widehat{Z}$ the following are 
equivalent: (i) there exists $0 \ne u \in \cH_\pi$ with $\phi_{u,u}
\in L_2(G/Z,\zeta)$, (ii) if $u, v \in \cH_\pi$ then $\phi_{u,v} \in
L_2(G/Z,\zeta)$, (iii) $[\pi] \in \widehat{G}_{\zeta-disc}$\,.  Under
these conditions there is a number $\deg(\pi) > 0$, the {\em formal degree}
of $[\pi]$, such that
\begin{equation}\label{2.2.2}
\begin{aligned}
&\langle \phi_{u,v},\phi_{u',v'} \rangle = \tfrac{1}{\deg(\pi)}
		\langle u, u'\rangle \overline{\langle v, v'\rangle}
		\text{ for } u,u',v,v' \in \cH_\pi \\
&\langle \phi_{u,v},\phi_{u',v'} \rangle = 0 \text{ for }
	u,v \in \cH_\pi \text{ and } u', v' \in \cH_{\pi'}
	\text{ with } [\pi] \ne [\pi]. 
\end{aligned}
\end{equation}
If $G$ is compact, $Z = \{1\}$ and we normalize $\int_G dg = 1$, then
$\widehat{G} = \widehat{G}_{disc}$\,, $\deg{\pi}$ is the degree in
the usual sense, and (\ref{2.2.2}) reduces to the Frobenius--Schur
Relations.  

Define $L_p(G/Z,\zeta)$ in the same way, integration over $G/Z$,
for $1 \leqq p < \infty$.  Since $Z$ is central 
$(f*h)(x) = \int_{G/Z} f(g)h(g^{-1}x) dx$ gives a well defined
convolution product $L_1(G/Z,\zeta) \times L_p(G/Z,\zeta)
\to L_p(G/Z,\zeta)$.  The first line of (\ref{2.2.2}) can be
expressed $\phi_{u,v}*\phi_{u',v'} = \frac{1}{\deg(\pi)}
\langle u, v' \rangle \phi_{u',v} \in L_2(G/Z,\zeta)$.  

These results are due to Godement \cite{G1947} for compact $Z$,
to Harish--Chandra \cite{HC1956b} for semisimple $G$.
See Dixmier's exposition 
\cite[\S 14]{D1964} of Godement or apply Rieffel's
results \cite{R1969} to the convolution algebra
$L_1(G/Z,\zeta)\cap L_2(G/Z,\zeta)$.

\subsection{}\label{ssec2c} 
\setcounter{equation}{0}
We can relax the condition that
$Z$ be central in $G$, requiring only that $G$ have a finite
index subgroup $J$ that centralizes $Z$.  Then we may assume 
$Z \subset J$ and $J$ normal in $G$.  If $\zeta \in \widehat{Z}$
denote $\ell_\zeta^J = \Ind_Z^J(\zeta)$, so $\ell_\zeta =
\Ind_J^G(\ell_\zeta^J)$.  Since $|G/J| < \infty$ the restriction
$\ell_\zeta|_J = \sum_{xJ \in G/J}\,\ell_{\zeta_x}^J$\,, finite
sum, where $\zeta_x(z) = \zeta(x^{-1}zx)$.  Now
\begin{equation}\label{2.3.2}
\begin{aligned}
&\widehat{G}_{disc} = \{[\pi] \in \widehat{G} \mid \pi|_J 
 \text{ has a subrepresentation } [\psi] \in \widehat{J}_{disc}\} 
 \text{ and }\\
&\widehat{J}_{disc} = \{[\psi] \in \widehat{J} \mid \psi
 \text{ is a subrepresentation of }\pi|_J \text{ for some }
 [\pi] \in \widehat{G}_{disc}\}
\end{aligned}
\end{equation}
In our applications, if $[\psi] \in \widehat{J}_{disc}$ then
$\Ind_J^G(\psi)|_J$ is a finite sum of mutually inequivalent 
representations, so $[\psi] \mapsto [\Ind_J^G(\psi)]$ maps
$\widehat{J}_{disc}$ onto $\widehat{G}_{disc}$.

\subsection{}\label{ssec2d} 
\setcounter{equation}{0}
Later we will construct the relative 
discrete series $\widehat{G}_{disc}$ on homogeneous vector bundles
$\E \to G/U$ where $Z \subset U$ and $U/Z$ is compact.  For that
we will need a mild extension of the Peter--Weyl Theorem.
\begin{lemma}\label{2.4.1}
Let $U$ be a locally compact group and $Z$ a closed central
subgroup such that $U/Z$ is compact.  Then every class $[\chi]
\in \widehat{U}$ is finite dimensional.
\end{lemma}
\begin{proof}
Let $S$ be the circle group $\{s \in \C \mid |s|=1\}$ and
$1_S \in \widehat{S}$ the character $1_S(s) = s$.  Given
$\zeta \in \widehat{Z}$ define
$U[\zeta] = \{S\times U\}/\{(\zeta(z)^{-1},z)\mid z \in Z\}$.
If $[\chi] \in \widehat{U}_\zeta$ now 
$[1_S\otimes \chi] \in \widehat{S\times U}$
factors through $U[\zeta]$.  But $U[\zeta]$ is an extension
$1 \to S \to U[\zeta] \to U/Z \to 1$ of a compact group by a
compact group, so it is compact.  Thus $1_S\otimes \chi$
is finite dimensional.
\end{proof}

\begin{lemma}\label{2.4.2}
If $\zeta \in \widehat{Z}$ then 
$\widehat{U}_\zeta = \widehat{U}_{\zeta-disc}$\,.
\end{lemma}
\begin{proof} Coefficients of $[\chi] \in \widehat{U}_\zeta$
are in $L_2(U/Z,\zeta)$ because $U/Z$ is compact.
\end{proof}

\begin{lemma}\label{2.4.3}
Let $\zeta \in \widehat{Z}$, $[\chi] \in \widehat{U}_\zeta$\,,
$E_\chi$ the representation space of $\chi$, and $L_\chi$ the
space of coefficients.  Then $L_2(U/Z,\zeta)$ is the orthogonal direct
sum $\sum_{\widehat{U}_\zeta}L_\chi 
= \sum_{\widehat{U}_\zeta} E_\chi \otimes E_\chi^*$ where $U\times U$
acts on $L_\chi = E_\chi \otimes E_\chi^*$ by 
$(\chi \otimes 1)\boxtimes (1\otimes \chi^*)$.
\end{lemma}
\begin{proof} Let $q:U \to U[\zeta]$ be the restriction to $U$ of the
projection $S \times U \to U[\zeta]$ in the proof of Lemma \ref{2.4.1}.  The 
assertions follow from the part concerning $\widehat{U[\zeta]}_{1_S}$ in
the standard Peter--Weyl Theorem for $U[\zeta]$
\end{proof}

Normalize Haar measures by
$\int_U du = \int_Z dz\int_{U/Z}d(uZ)$ and $\int_{U/Z} d(uZ) = 1$.
$C_c(U)$ denotes the space of continuous compactly supported functions on $U$.
If $f \in C_c(U)$ then $\chi(f):E_\chi \to E_\chi$ by
$\chi(f)v = \int_U f(u)\chi(u)v\, du$.  

\begin{proposition}\label{2.4.4}
Let $\zeta \in \widehat{Z}$, $[\chi] \in \widehat{U}_\zeta$\,.  If
$f \in C_c(U)$ then 
$\trace\,\chi(f) = \int_U f(u)\, \trace\,\chi(u)\,du$
and $\dim\,\chi (= \dim E_\chi)$ is the formal degree $\deg \chi$\,.
Further, orthogonal projection $L_2(U/Z,\zeta) \to L_\chi$ is given by right
(or left) convolution with $(\dim\,\chi)\trace\overline{\chi}$.
\end{proposition}
\begin{proof} Let $\{v_1,\dots,v_n\}$ be an orthonormal basis of $E_\chi$
and $\phi_{i,j} = \langle v_i,\chi(v_j)\rangle$.  Then $\trace \chi(f)
=\sum\langle \chi(f)v_i,v_i\rangle 
= \sum\int_U\langle f(u)\chi(u)v_i,v_i\rangle du 
= \int_U f(u)\trace \chi(u) du$.  That is the first assertion.  For the
second, use $q:U \to U[\zeta]$ as in the proof of Lemma \ref{2.4.3} to
see $\dim\chi = \dim(1_S\otimes\chi) = \deg(1_S\otimes\chi) = \deg\chi$.
For the last, use $\phi_{u,v}*\phi_{u',v'} = \frac{1}{\deg(\pi)}
\langle u, v' \rangle \phi_{u',v} \in L_2(G/Z,\zeta)$ (from the first line 
of (\ref{2.2.2})) to calculate $1 = \langle \trace\chi,\trace\chi\rangle
= \sum_{i,j}\langle \phi_{i,i}, \phi_{j,j}\rangle = n/\deg(\chi)$.
\end{proof}

Combining (\ref{2.1.2}), Lemma \ref{2.4.3} and Proposition \ref{2.4.4} we have
the Plancherel Formula for $U$.  If $f \in C_c(U)$ and $\zeta \in \widehat{Z}$
we denote $f_\zeta(u) = \int_Z f(uz)\zeta(z)dz$.

\begin{proposition}\label{2.4.5}
$L_2(U) = \int_{\widehat{Z}} \bigl ( \sum_{\widehat{U}_\zeta}
E_\chi \otimes E_\chi^* \bigr )\,d\zeta$.  If $f \in C_c(U)$ and $u \in U$
then $f(u) = \int_{\widehat{Z}} \bigl ( \sum_{\widehat{U}_\zeta}
\trace \chi((r(u)f)_\zeta) \dim\chi\bigr )d\zeta$ where $r(u)f(u') = f(uu')$.
\end{proposition}

\section{Relative Discrete Series for Reductive Groups}\label{sec3}
\setcounter{equation}{0}
Harish--Chandra's theory of the discrete series (\cite{HC1965}, \cite{HC1966})
extends from his class $\cH$ to our class of reductive Lie groups that 
contains all connected reductive groups and has certain hereditary properties.
We describe that class in \S \ref{ssec3a}.  Certain points of Harish--Chandra's
general character theory are recalled in \S 3.2.  In \S 3.3 we show how
to reduce certain questions from connected reductive groups in general to
the case of compact center.  We use that in \S 3.4 to extend Harish--Chandra's
theory of the discrete series to all connected reductive Lie groups, in
particular to all connected semisimple Lie groups.  Then in \S 3.5 we obtain 
the discrete series for our class of groups described in \S 3.1.

\subsection{} \label{ssec3a}  
\setcounter{equation}{0}
From now on, $G$ is a reductive Lie group in
the sense that its Lie algebra is reductive: $\gg = \gc \oplus [\gg,\gg]$
where $\gc$ is the center and $\gg' = [\gg,\gg]$ is semisimple.
\begin{definition}\label{3.1.1}
{\rm Suppose that $G$ satisfies the condition
$$
\text{if $g \in G$ then $\Ad(g)$ is an inner automorphism on }\gg_\C
$$
and that $G$ has a closed normal abelian subgroup $Z$ such that
$$
Z \text{ centralizes } G^0\,,\, G/ZG^0 \text{ is finite, and }
	ZG^0/(Z\cap G^0) \text{ is compact.}
$$
Then we say that $G$ is a {\em general real reductive Lie group}, i.e.,
that $G$ of class $\widetilde{\cH}$.
}
\end{definition}
If $G/G^0$ is finite then $Z_{G^0}$ satisfies the conditions for $Z$ in 
the definition.  The Harish--Chandra class $\cH$ consists of the groups
$G \in \widetilde{\cH}$ for which $Z$ is finite.  In any case, the 
discrete series relative to $Z$ is independent of choice of group $Z$ that
satisfies the conditions of Definition \ref{3.1.1}.  We use the notation
\begin{equation}\label{3.1.3}
G^\dagger = \{g \in G \mid \Ad(g) \text{ is an inner automorphism on } G^0.
\end{equation}
Then $G^\dagger = Z_G(G^0)G^0$, $Z_G(G^0)$ is compact and $G^\dagger/ZG^0$
is finite.

\subsection{}\label{ssec3b}
\setcounter{equation}{0}
Using the second condition of Definition(\ref{3.1.1}) we choose a 
maximal compact subgroup $K/Z$
of $G/Z$.  Then $K^0 = K\cap G^0$\,, $K$ meets every topological component
of $G$, and $Z_{G^0} \subset K$.  The basis of Harish--Chandra's 
character theory is
\begin{lemma}\label{3.2.2} There is an integer $n_G \geqq 1$ such that,
if $[\kappa] \in \widehat{K}$ and $[\pi] \in \widehat{G}$, then the
multiplicity $mult(\kappa,\pi|_K) \leqq n_G\dim(\kappa) < \infty$.
\end{lemma}
\begin{proof}
This was proved by Harish--Chandra for connected reductive Lie groups
\cite{HC1954b}, in particular for $G^0$.  If $[\kappa_1] \in \widehat{ZK^0}$
and $[\pi_1] \in \widehat{ZG^0}$ we have $\zeta, \zeta' \in \widehat{Z},
[\kappa_0] \in \widehat{K^0}\text{ and } [\pi_0] \in \widehat{G^0}$ such
that $\kappa_1 = \zeta\otimes\kappa_0$ and $\pi_1 = \zeta'\otimes\pi_o$\,.
Now $mult(\kappa_1,\pi_1|_{ZK^0}) \leqq mult(\kappa_0,\pi_0|_{K^0})
\leqq n_{G^0}\dim(\kappa_0) = n_{G^0}\dim(\kappa)$, so the lemma follows for
$ZG^0$.  Finally, 
if $[\kappa] \in \widehat{K}$ and $[\pi] \in \widehat{G}$, we decompose
$\kappa|_{ZK^0} = \sum \kappa_i$ and $\pi|_{ZG^0} = \sum \pi_j$ into irreducible
constituents, and $mult(\kappa,\pi|_K) \leqq 
\sum_{i,j} mult(\kappa_i,\pi_j|_{ZK^0}) \leqq \sum_{i,j} n_{G^0}\dim(\kappa_i)
= \sum_i n_{G^0} \dim(\kappa) \leqq (n_{G^0}|G/ZG^0|)\dim(\kappa)$.
\end{proof}
\begin{remark} Using Casselman--Mili\v{c}i\'c \cite{CM1982} one can
see that $n_G \leqq |G/ZG^0|$.
\end{remark}

The first consequence of Lemma \ref{3.2.2} is that the group $G$ is CCR.
In other words, if $[\pi] \in \widehat{G}$ and $f \in L_1(G)$ then
$\pi(f) = \int_Gf(g)\pi(g)dg$ is a compact operator on $\cH_\pi$\,.
In particular $G$ is of type I.  The second consequence is that $\pi(f)$ 
is of trace class for $f \in C_c^\infty(G)$ and that
\begin{equation}\label{3.2.3}
\Theta_\pi: C_c^\infty \to \C \text{ by } \Theta_\pi(f) = \trace \pi(f)
\text{ is a Schwartz distribution on } G.
\end{equation}
$\Theta_\pi$ is the {\em global character} or {\em distribution character}
of $\pi$.  Classes $[\pi] = [\pi']$ if and only if $\Theta_\pi = \Theta_{\pi'}$.

A differential operator $z$ on $G$ has {\em transpose} given by
$\int_G ^tz(f)(g)h(g)dg = \int_G f(g)z(h)(g)dg$.  The operator $z$ acts on
distributions by $(z\Theta)(f) = \Theta(^tz(f))$.  Given $\Theta$ now
$z \mapsto z\Theta$ is linear in $z$.  A distribution on $G$ is
{\em invariant} if $\Theta(f) = \Theta(f\cdot \Ad(g))$ for $f \in C_c^\infty(G)$
and $g \in G$.

The universal enveloping algebra $\cU(\gg)$ is the associative algebra of
all left--invariant differential operators on $G$.  The center $\cZ(\gg)$
consists of the bi--invariant operators.  That uses (\ref{3.1.1}).  A
distribution $\Theta$ on $G$ is an {\em eigendistribution} if 
$\dim \cZ(\gg)(\Theta) \leqq 1$.  In that case, using commutativity of 
$\cZ(\gg)$, we have an algebra homomorphism $\chi_\Theta:\cZ(\gg) \to \C$
defined by $z\Theta = \chi_\Theta(z)\Theta$.  If $[\pi] \in \widehat{G}$
its distribution character $\Theta_\pi$ is an invariant eigendistribution
on $G$, and the associated homomorphism is
\begin{equation}\label{3.2.4}
\chi_\pi: \cZ(\gg) \to \C\,,\, z\Theta_\pi = \chi_\pi(z)\Theta_\pi\,,
\text{ is the {\em infinitesimal character} of } [\pi].
\end{equation}

Choose a Cartan subalgebra $\gh \subset \gg$ and let $I(\gh_\C)$ denote the 
algebra of all polynomials on $\gh_\C^*$ invariant under the Weyl group
$W(\gg_\C,\gh_\C)$.  Let $\gamma: \cZ(\gg) \to I(\gh_\C)$ denote the
Harish--Chandra homomorphism \cite{HC1956c}.  If $\lambda \in \gh_\C^*$
then
\begin{equation}\label{3.2.5}
\chi_\lambda: \cZ(\gg) \to \C \text{ by } 
	\chi_\lambda(z) = [\gamma(z)](\lambda)
\end{equation}
is a homomorphism, every homomorphism $\cZ(\gg) \to \C$ is a $\chi_\lambda$\,,
and $\chi_\lambda = \chi_{\lambda'}$ if and only if $\lambda$ and
$\lambda'$ are $W(\gg_\C,\gh_\C)$--equivalent.

Harish--Chandra \cite{HC1956c} used the equations 
$z\Theta_\pi = \chi_\pi(z)\Theta_\pi$ and the description (\ref{3.2.5})
of $\chi_\pi$ to show that $\Theta_\pi$ is a locally $L_1$ function that
is analytic on the {\em regular set} $G'$.  Here $G'$ is the dense open
subset of all $g \in G$ for which $\{\xi \in \gg \mid \Ad(g)\xi = \xi\}$
is a Cartan subalgebra.  The complement of $G'$ has measure zero in $G$.

The differential equations also show that at most finitely many classes 
in $\widehat{G}$ can have the same infinitesimal character.

\subsection{}\label{ssec3c}  
\setcounter{equation}{0}
We describe a method for reducing questions of 
harmonic analysis on $ZG^0$ and $G^0$ to the same questions on connected
reductive groups with compact center.  With that we extends some of
Harish--Chandra's discrete series results \cite{HC1966} to $G^0$ in
\S 3.4 and then to $G$ in \S 3.5.  This uses the Mackey central extension
$1 \to S \to G[\zeta] \to ZG^0/Z \to 1$ for $\delta\zeta$ as a normalized
multiplier on $ZG^0$.  Thus
$G[\zeta] = \{S\times ZG^0\}/\{(\zeta(z)^{-1},z)| z \in Z\}$ and $S$ is
the circle group $\{s \in \C \mid |s| = 1\}$.
$G[\zeta]$ is a connected reductive group with Lie algebra 
$\gs \oplus (\gg/\gz)$ and compact center.

\begin{lemma}\label{3.3.2}
Let $p:ZG^0 \to G[\zeta]$ be the restriction of the projection
$S\times ZG^0 \to G[\zeta]$.  Then $f \mapsto f\cdot p$ gives a
$G$--equivariant isometry of $L_2(G[\zeta]/S,1_S)$ onto $L_2(ZG^0/Z,\zeta)$.
\end{lemma}

\begin{proof}  View $f \in L_2(G[\zeta]/S,1_S)$ as a function on $S\times ZG^0$.
If $g \in ZG^0$ and $z \in Z$ then
$(f\cdot p)(gz) = f(1,gz) = f(\zeta(z),g) = \zeta(z)^{-1}f(1,g)
= \zeta(z)^{-1}[f\cdot p](g)$, and
$\int_{ZG^0/Z} |(f\cdot p)(g)|^2d(gZ) 
= \int_{(S\times ZG^0)/(S\times Z)}
|f(s,g)|^2 d(sS \times gZ) 
= \int_{G[\zeta]/S} |f(\overline{g})|^2d(\overline{g}S).$
Thus $f \mapsto f\cdot p$ is an isometric injection of $L_2(G[\zeta]/S,1_S)$
into $L_2(ZG^0/Z,\zeta)$.  It is surjective because every 
$f' \in L_2(ZG^0/Z,\zeta)$ has form $f\cdot p$ with $f$ defined on 
$S \times ZG^0$ by $f(s,g) = s^{-1}f'(g)$.
\end{proof}

\begin{theorem}\label{3.3.3}
$\varepsilon[\psi] = [\psi\cdot f]$ defines a bijection
$\varepsilon = \varepsilon_\zeta :
\widehat{G[\zeta]}_{1_S} \to (\widehat{ZG^0})_\zeta$\,.
It maps $\widehat{G[\zeta]}_{1_S-disc}$
onto $\widehat{ZG^0}_{\zeta-disc}$ and carries Plancherel measure of
$\widehat{G[\zeta]}_{1_S}$ to Plancherel measure of $(\widehat{ZG^0})_\zeta$\,.
Distribution characters satisfy $\Theta_{\varepsilon[\psi]} = 
\Theta_{[\psi]}\cdot p$.
\end{theorem}

\begin{proof}
View $[\psi] \in \widehat{G[\zeta]}_{1_S}$ as a representation of 
$S\times ZG^0$.  If $z \in Z$ and $g \in ZG^0$ then $\psi\cdot p)(gz)
= \psi(1,gz) = \psi(\zeta(z),g) = \zeta(z)\psi(1,g) = \zeta(z)[\psi\cdot p](g)$.
Thus the central character of $\psi\cdot p$ restricts to $\zeta$ on $Z$,
and $[\psi\cdot p] \in (\widehat{ZG^0})_\zeta$\,.

Let $[\psi], [\psi'] \in \widehat{G[\zeta]}_{1_S}$ and $b:\cH_\psi \to
\cH_{\psi'}$ an isometry.  If $\psi' = b\psi b^{-1}$ the above calculation 
shows $(\psi' p) = b (\psi p)b^{-1}$.  If $b$ intertwines $\psi p$ with 
$\psi' p$ the same calculation shows that it intertwines $\psi|_{ZG^0}$
with $\psi'|_{ZG^0}$\,.  Then it intertwines $\psi$ and $\psi'$ because
$\psi'(s,g) = s\psi'(1,g) = sb\psi(1,g)b^{-1} = b\psi(s,g)b^{-1}$.  Now
$\varepsilon_\zeta : \widehat{G[\zeta]}_{1_S} \to (\widehat{ZG^0})_\zeta$
is a well defined bijection.

We reduce the proof of Theorem \ref{3.3.3} to the case where $Z \subset G^0$,
i.e. where $ZG^0$ is connected.  Let $\zeta^0 = \zeta|_{ZG^0}$\,.
Then $G^0 \hookrightarrow G$ induces a commutative diagram
{\tiny $$ \begin{CD}
1 @>>> S @>>> G^0[\zeta^0] @>>> G^0/(Z\cap G^0) @>>> 1 \\
@.   @VVV          @VV{a}V            @VVV          \\
1 @>>> S @>>> G[\zeta] @>>> ZG^0/Z @>>> 1
\end{CD} $$}
where $a$ is an isomorphism because $S \to S$ and $G^0/(Z\cap G^0) \to 
ZG^0/Z$ are.  That results in a bijection $\varepsilon$ that is the
composition $\widehat{G[\zeta]}_{1_S} \overset{a^*}{\rightarrow}
\widehat{G^0[\zeta^0]}_{1_S} \overset{\varepsilon^0}{\rightarrow}
(\widehat{G^0})_{\zeta^0} \overset{b}{\rightarrow} 
(\widehat{ZG^0})_\zeta$
where $a^*$ is induced by $a$ and $b[\pi^0] = [\zeta\otimes\pi^0]$.
Plancherel measure and relative discrete series are transported by $b$.
If $[\pi^0] \in (\widetilde{G^0})_{\zeta^0}\,, z \in Z \text{ and } g \in G^0$
then $\Theta_{b[\pi^0]}(zg) = \zeta(z)\Theta_{\pi^0}(g)$.  Thus if
$\varepsilon: \widehat{G^0[\zeta^0]}_{1_S} \to (\widehat{G^0})_{\zeta^0}$ 
has the properties asserted in Theorem \ref{3.3.3} those properties pass over
to $\varepsilon : \widehat{G[\zeta]}_{1_S} \to (\widehat{ZG^0})_\zeta$\,.

Now we assume $Z \subset G^0$, and we further reduce the proof to the case
where $G^0$ is simply connected.  Let $\tau:\widetilde{G} \to G^0$ be
the universal cover, $\widetilde{Z} = \tau^{-1}(Z)$, and $\widetilde{\zeta}$
the corresponding lift of $\zeta$.  Then we have 
{\tiny $\begin{CD}
\widehat{\widetilde{G}[\widetilde{\zeta}]}_{1_S} @>{\widetilde{\varepsilon}}>>
	\widehat{\widetilde{G}}_{\widetilde{\zeta}} \\
@A{i^*}AA @AA{j}A \\
\widehat{G[\zeta]}_{1_S} @>{\varepsilon}>> (\widehat{G^0})_{\zeta_0}
\end{CD} $}
linking $\varepsilon$ to $\widetilde{\varepsilon}$.
Everything is preserved by $i^*$, and $j$ transports Plancherel measure and 
relative discrete series.

We check that $\Theta_{j[\psi]} = \Theta_{[\psi]}\tau$ for $[\psi] \in
(\widehat{G^0})_\zeta$.  It suffices to test this on functions 
$f \in C^\infty_c(\widetilde{G})$ in a component $U$ of $\tau^{-1}(U_1)$
where $U_1$ is an open subset of $G^0$ admissible for the covering.  Now
$f$ is the lift to $U$ of $f_1 \in C_c^\infty(U_1)$ and we calculate
$\Theta_{j[\psi]}(f) = \trace \int_U f(x)(\psi\cdot\tau)(x)dx
= \trace \int_U f_1(\tau x)dx = \trace \int_{U_1} f_1(x_1)\psi(x_1)dx_1
= \Theta_{[\psi]}(f_1)$.  That shows $\Theta_{j[\psi]} = \Theta_{[\psi]}\tau$\,.

We have shown that if 
$\widetilde{\varepsilon}: \widehat{\widetilde{G}[\widetilde{\zeta}]}_{1_S}
\to \widehat{\widetilde{G}}_{\widetilde{\zeta}}$ satisfies Theorem \ref{3.3.3}
then the same is true for $\varepsilon: \widehat{G[\zeta]}_{1_S} \to
\widehat{G^0}_\zeta$\,.  This reduces the proof of Theorem \ref{3.3.3} to
the case where $Z \subset G^0$ and $G^0$ is simply connected.  There
$G^0 = V \times G_{ss}$ where $V$ is a vector group and $G_{ss}$ is
semisimple.  We can enlarge $Z$ to $VZ$ and assume $Z = V \times D$ where
$D$ is a subgroup of finite index in the center $Z_{ss}$ of $G_{ss}$\,.
That done, $\zeta = \nu \boxtimes \delta$ accordingly, and
$G_{ss} \hookrightarrow G^0$ gives an isomorphism
$a: G_{ss}[\delta] \cong G^0[\zeta]$.  This results in a commutative
diagram
{\tiny $\begin{CD}
\widehat{G_{ss}[\delta]}_{1_S} @>{\varepsilon_{ss}}>>
        (\widehat{G_{ss}})_\delta \\
@A{a^*}AA @VV{b}V \\
\widehat{G^0[\zeta]}_{1_S} @>{\varepsilon}>> (\widehat{G^0})_{\zeta_0}
\end{CD} $}
where $b[\pi_{ss}] = [\nu \otimes \pi_{ss}]$.  As in the reduction to
$Z \subset G^0$, if Theorem \ref{3.3.3} holds for $G_{ss}$ it holds for $G$.

We are finally reduced to the case where $Z \subset G^0$ and $G^0$ is
semisimple.  Then $Z$ is discrete, so $S\times G^0 \to G^0[\zeta]$ is a
Lie group covering.  The method of reduction to simply connected $G^0$ 
completes the proof of Theorem \ref{3.3.3}.
\end{proof}

\subsection{} \label{ssec3d}  
\setcounter{equation}{0}
We extend Harish--Chandra's description
\cite{HC1966} of the discrete series of a connected semisimple Lie group.
In fact his analysis extends without change to connected reductive Lie 
groups with compact center, as in the case of $G^0$ when $Z\cap G^0$ is
compact.  We formulate the results.  Recall that $G^0$ is an arbitrary
connected reductive Lie group and that $Z\cap G^0$ is co-compact in $Z_{G^0}$\,.

\begin{theorem}\label{3.4.1}
$G^0$ has a relative discrete series representation if and only if
$G^0/(Z\cap G^0)$ has a compact Cartan subgroup.
\end{theorem}

The compact Cartan subgroups of $G^0/(Z\cap G^0)$ have form $H^0/(Z\cap G^0)$.
The unitary characters on $H^0$ are in bijective correspondence with
\begin{equation}\label{3.4.2}
L = \{\lambda \in i\gh^* \mid e^\lambda \text{ is well defined on } H^0.
\end{equation}
Choose a positive root system $\Sigma^+$ and use the standard notation
$\rho = \tfrac{1}{2}\sum_{\phi \in \Sigma^+}\phi$,
$\varpi(\lambda) = \prod_{\phi \in \Sigma^+}\langle\phi,\lambda\rangle$ and
$\Delta = \prod_{\phi \in \Sigma^+}(e^{\phi/2}-e^{-\phi/2})$ where
$\langle\cdot,\cdot\rangle$ comes from the Killing form.  Then $\Sigma^+
\subset L$ so $2\rho \in L$.  Passing to a $2$--sheeted covering group of
$G^0/(Z\cap G^0)$, if necessary, we may assume $\rho \in L$ and 
$e^\rho(Z\cap G^0) = 1$.  $\Delta$ is an
analytic function on $H^0$, well defined on $H^0/(Z\cap G^0)$.
The {\em regular set} is $L' = \{\lambda \in L \mid \varpi(\lambda) \ne 0\}$.
Note $\rho \in L'$.  If $\lambda \in L'$ we define
{\footnotesize
\begin{equation}\label{3.4.6}
q(\lambda) = |\{\text{compact } \phi\in\Sigma^+ \mid 
	\langle\phi,\lambda\rangle < 0\}| +
|\{\text{ noncompact } \phi\in\Sigma^+ \mid 
        \langle\phi,\lambda\rangle > 0\}|.
\end{equation}
}
Thus $(-1)^{q(\lambda)} = (-1)^q\text{sign\,}\varpi(\lambda)$.

\begin{theorem}\label{3.4.7}
If $\lambda \in L'$ there is a unique class $[\pi_\lambda] \in
\widehat{G^0}_{disc}$ whose distribution character satisfies
$\Theta_{\pi_\lambda}|_{H^0\cap G'} = (-1)^{q(\lambda)}\tfrac{1}{\Delta}
	\sum_{w \in W_{G^0}} \det(w)e^{w\lambda}$.  Every class in
$\widehat{G^0}_{disc}$ is one of these $[\pi_\lambda]$.  Classes
$[\pi_\lambda] = [\pi_{\lambda'}]$ precisely when 
$\lambda' \in W_{G^0}(\lambda)$.  With a certain normalization of
Haar measure on $G^0$, $[\pi_\lambda]$ has formal degree $|\varpi(\lambda)|$.
\end{theorem}

\begin{corollary}\label{3.4.9} $[\pi_\lambda] \in (\widehat{G^0})_{disc}$
has dual $[\pi_\lambda^*] = [\pi_{-\lambda}]$, central character
$e^{\lambda - \rho}|_{Z_{G^0}}$ and infinitesimal character $\chi_\lambda$ 
as in {\rm (\ref{3.2.5})}; $\chi_\lambda(Casimir)
= ||\lambda||^2 - ||\rho||^2$.
\end{corollary}

When $Z_{G^0}$ is compact, Theorems \ref{3.4.1} and \ref{3.4.7} reduce to
Harish--Chandra's celebrated results \cite[Theorems 13 and 16]{HC1966}.
We describe the reduction.

Let $[\pi] \in (\widehat{G^0})_\zeta$\,.  By Theorem \ref{3.3.3} there exists
$[\psi] \in \widehat{G^0[\zeta]}_{1_S}$ such that $\varepsilon_\zeta[\psi]
= [\pi]$.  In particular \cite[Theorem 13]{HC1966} $G^0[\zeta]$ has a compact
Cartan subgroup, and it must have form $H^0[\zeta]$ where $H^0$ is a Cartan
subgroup of $G^0$.  Since $H^0[\zeta]$ is compact, so is $H^0/(Z\cap G^0)$.
That proves the ``only if'' part of Theorem \ref{3.4.1}.

Conversely,
let $H^0/(Z\cap G^0)$ be a compact Cartan subgroup of $G^0/(Z\cap G^0)$
and $\zeta \in \widehat{Z\cap G^0}$, so $H^0[\zeta]$ is a compact Cartan
subgroup of $G^0[\zeta]$.  Denote 

$L[\zeta] = \{\nu \in i\gh[\zeta]^* \mid
e^\nu \text{ is well defined on } H^0[\zeta]\}$,

$L[\zeta]'_{1_S} = \{\nu \in L[\zeta] \mid \varpi(\nu) \ne 0
\text{ and } e^\nu|_S = 1_S\}$, and

$L'_\zeta = \{\lambda \in L \mid \varpi(\lambda) \ne 0 \text{ and }
e^\lambda|_{Z\cap G^0} = \zeta\}$.

\noindent
Since $G^0[\zeta]$ is a connected reductive Lie group with compact center, and
$e^\rho(S) = 1$, \cite[Theorem 16]{HC1966} gives a map 
$\omega_1: L[\zeta]'_{1_S} \to \widehat{G^0[\zeta]}_{1-disc}$\,, by
$\nu \mapsto [\psi_\nu]$, that satisfies the assertions of Theorem \ref{3.4.7}.
We construct the corresponding 
$\omega_\zeta: L'_\zeta \to \widehat{G^0}_{\zeta-disc}$ by
$\omega_\zeta\cdot\delta = \varepsilon\cdot\omega_1$ where $\varepsilon$ is
the bijection of Theorem \ref{3.3.3} and $\delta :
L[\zeta]'_{1_S} \to L'_\zeta$ is defined as follows.

Let $\nu \in L[\zeta]'_{1_S}$.  The distribution character of
$\varepsilon\omega_1(\nu) = \omega_\zeta \delta(\nu)$ must have
$(H^0\cap G')$--restriction $(-1)^{q(\delta \nu)}\tfrac{1}{\Delta}
\sum \det(w)e^{w\delta \nu} = (-1)^{q(\nu)}\tfrac{1}{\Delta}
\sum \det(w) e^{w\nu}\cdot p$.  For that, define $\delta$ by
$e^{\delta\nu} = e^\nu\cdot p$, i.e. $\delta\nu = p^*\nu$ under
$p:\gg \to \gg[\zeta]$.  Since $p$ restricts to an isomorphism of derived
algebras, $\delta$ bijects $L[\zeta]'_{1_S}$ to $L'_\zeta$ equivariantly
for $W$.  Our assertions now go over from $\omega_1$ to $\omega_\zeta$\,.

That completes the derivation of Theorems \ref{3.4.1} and \ref{3.4.7}
from \cite{HC1966}.

\subsection{}\label{ssec3e}
\setcounter{equation}{0}
We extend the description of the relative discrete series from connected
reductive groups to the class of real reductive Lie groups 
specified in \S \ref{ssec3a}.

\begin{lemma}\label{3.5.1}
$ZZ_{G^0}$ has finite index in $Z_G(G^0)$.  Every class 
$[\chi] \in \widehat{Z_G(G^0)}$ has dimension 
$\dim \chi \leqq |Z_G(G^0)/ZZ_{G^0}|$.
\end{lemma}

\begin{proof} The second condition of Definition \ref{3.1.1} shows that every 
$[\chi] \in \widehat{Z_G(G^0)}$ is a summand of 
$\Ind_{\widetilde{Z}}^{Z_G(G^0)}(\beta)$ for some 
$\beta \in \widehat{\widetilde{Z}}$.
\end{proof}

\begin{proposition}\label{3.5.2}
$\widehat{G^\dagger}$ is the disjoint union of the sets
\begin{equation}\label{3.5.3}
(\widehat{G^\dagger})_\xi = \{[\chi\otimes\pi] \mid 
	[\chi] \in \widehat{Z_G(G^0)}_\xi \text{ and }
	[\pi] \in (\widehat{G^0})_\xi \text{ where } \xi \in \widehat{Z_{G^0}}\}.
\end{equation}
Here $[\chi\otimes\pi]$ has the same infinitesimal character $\chi_\pi$
and $[\pi]$ and has distribution character $\Theta_{\chi\otimes\pi}(zg) =
(\trace \chi(z))\Theta_\pi(g)$ for $z \in Z_G(G^0)$ and $g \in G^0$\,.
Further, $[\chi\otimes\pi] \in (\widehat{G^\dagger})_{disc}$ if and only if
$[\pi] \in (\widehat{G^0})_{disc}$\,.
\end{proposition}

\begin{proof}
$\widehat{G^\dagger}$ is the disjoint union of the $(\widehat{G^\dagger})_\xi$\,,
$\xi \in \widehat{Z_{G^0}}$, because $Z_{G^0}$ is central in $G^\dagger$.
Now fix $\xi \in \widehat{Z_{G^0}}$, $[\chi] \in \widehat{Z_G(G^0)}_\xi$
and $[\pi] \in (\widehat{G^0})_\xi$\,.  Note $[\chi\otimes\pi] \in
(\widehat{G^\dagger})_\xi$\,, and $(\chi\otimes\pi)|_{G^0} = m\pi$ where
$m = \dim \chi < \infty$.  Thus $\chi\otimes\pi$ has infinitesimal character
$\chi_\pi$\,, and is discrete relative to $Z$ exactly when $\pi$ is
discrete relative to $Z\cap G^0$\,.

To prove the formula for $\Theta_{\chi\otimes\pi}$ we need only consider
test functions $f \in C^\infty_c(G^\dagger)$ supported in a single coset
$z_0G^0$\,, and there we compute\\
$\trace(\chi\otimes\pi)(f)$ 
= $\trace \int_{G^\dagger}f(zg)\chi(z)\otimes\pi(g)dg
= \trace \int_{G^0} f(z_0g)\chi(z_0)\otimes\pi(g)dg\\
= \trace \bigl ( \chi(z_0)\otimes\int_{G^0}f(z_0g)\pi(g)dg \bigr )
= (\trace \chi(z_0))\bigl (\trace \int_{G^0}f(z_0g)\pi(g)dg \bigr ) \\
= (\trace \chi(z_0)) \int_{G^0} f(z_0g)\Theta_\pi(g)dg 
= \int_{G^0} f(z_0g) (\trace\chi(z_0))\Theta_\pi(g)dg\\
= \int_{G^\dagger}f(zg)(\trace\chi(z)) \Theta_\pi(g)d(zg)$,
so $\Theta_{\chi\otimes\pi}(zg) = (\trace\chi(z)) \Theta_\pi(g)$, as asserted.

Finally let $[\gamma] \in (\widehat{G^\dagger})_\xi$\,.  Since $Z_G(G^0)$
acts trivially on $\widehat{G^0}$ and $G^0$ is of type I, now $\gamma|_{G^0}
= mn$ where $[\chi] \in \widehat{Z_G(G^0)}_\xi$\,. Thus $[\gamma] = 
[\chi\otimes\pi]$ because $[\chi\otimes\pi]$ is a subrepresentation.  That
proves (\ref{3.5.3}). Proposition \ref{3.5.2} is proved.
\end{proof}

Proposition \ref{3.5.2} gives the relative discrete series of $G^\dagger$ in
terms of those of $Z_G(G^0)$ and $G^0$.  The following lets us go on to $G$.

\begin{proposition}\label{3.5.5}
Let $[\gamma] = [\chi\otimes\pi] \in \widehat{G^\dagger}$ and
define $\psi = \Ind_{G^\dagger}^G(\psi)$.  Then
{\rm (1)} $[\psi]$ has the same infinitesimal character $\chi_\pi$ as $[\pi]$
and {\rm (2)} $[\pi]$ has distribution character that is a locally $L_1$
function supported in $G^\dagger$ and given there by
\begin{equation}\label{3.5.6}
\Theta_\psi(zg) = {\sum}_{xG^\dagger \in G/G^\dagger}\, (\trace \chi(x^{-1}zx))
	\Theta_\pi(x^{-1}gx) 
\end{equation}
for $z \in Z_G(G^0)$ and $g \in G^0$\,.  In particular $\Theta_\psi$ is analytic
on the regular set $G'$ and satisfies $\Theta_\psi|_{G^\dagger} =
\Theta|_{\psi|_{G^\dagger}}$.  Further {\rm (3)} 
if $[\pi] \in (\widehat{G^0})_{disc}$ then $[\psi] \in \widehat{G}_{disc}$\,, 
and every class in $\widehat{G}_{disc}$ is obtained this way. 
\end{proposition}

\begin{proof}
We follow an argument \cite[Lemma 4.3.3]{W1967} of Frobenius for (1) and (2).
As $G^\dagger$ is normal and has finite index in $G$, $\Theta_\psi$
exists and is supported in $G^\dagger$, where $\Theta_\psi|_{G^\dagger} =
\Theta|_{\psi|_{G^\dagger}}$.  Note $\psi|_{G^\dagger} = 
\sum_{G/G^\dagger} \gamma\cdot\Ad(x^{-1})$.  If $z \in Z_G(G^0)$ and
$g \in G^0$ now $\Theta_\psi(zg) = \sum \Theta_{\gamma\cdot\Ad(x^{-1})}(zg)
= \sum \Theta_\gamma(z^{-1}zx\cdot x^{-1}gx)$\,.  Assertion (2) follows 
from Proposition \ref{3.5.2}.

If $x \in G$ then $\Ad(x)$ is an inner automorphism on $\gg_\C$\,, hence 
trivial on $\cZ(\gg)$, so all the $\gamma\cdot\Ad(x^{-1})$ are the same
on $\cZ(\gg)$.  Now $\psi$ has infinitesimal character $\chi_\psi =
\chi_\gamma = \chi_\pi$\,.

Every class in $\widehat{G}_{disc}$ is a subrepresentation of an 
$[\Ind_{\widehat{G^\dagger}}^G(\gamma)]$, 
$[\gamma] \in (\widehat{G^\dagger})_{disc}$\,, because $|G/G^\dagger| < \infty$.
If $[\gamma] = [\chi\otimes\pi]$ as in Proposition \ref{3.5.2} then
$[\gamma] \in (\widehat{G^\dagger})_{disc}$ is equivalent to
$[\pi] \in (\widehat{G^0})_{disc}$\,.  To prove (3) now we need only
check that $\psi = \Ind_{G^\dagger}^G(\gamma)$ is irreducible whenever
$[\pi] \in (\widehat{G^0})_{disc}$\,.  

Choose a Cartan subgroup
$H^0 \subset G^0$ with $H^0/(Z\cap G^0)$ compact.  The corresponding
Cartan subgroup of $G$ is the centralizer $H$ of $\gh$.  Hypothesis
(\ref{3.1.1}) says that the Weyl group $W_G$ is a subgroup of the complex
Weyl group $W(\gg_\C,\gh_\C)$.  As any two compact Cartan subgroups of
$G^0/(Z\cap G^0)$ are conjugate we have a system $\{x_1,\dots ,x_r\}$
of representatives of $G$ modulo $G^\dagger$ such that each $\Ad(x_i)\gh = \gh.$
Now $W_G = \bigcup (x_jH)W_{G^0} \subset W(\gg_\C,\gh_\C)$.  Let $[\gamma]
= [\chi\otimes\pi]$ with $[\pi] \in (\widehat{G^0})_{disc}$\,.  Express
$[\pi] = [\pi_\lambda]$ with $\lambda \in L'$.  Then $[\pi\cdot\Ad(x_j^{-1})]
= [\pi_{\lambda_j}]$ where $\lambda_j = \Ad(x^{-1}_j)^*(\lambda)$.
Since $\lambda \in L'$ the $\lambda_j$ are distinct modulo the action of
$W_{G^0}$.  Theorem \ref{3.4.7} now says that the $\pi\cdot\Ad(x_j^{-1})$
are mutually inequivalent.  It follows that $\psi$ is irreducible.
\end{proof}

We formulate the extensions of Theorems \ref{3.4.1} and \ref{3.4.7}
from $G^0$ to $G$.
\begin{theorem}\label{3.5.8}
$G$ has a relative discrete series representation if and only if $G/Z$
has a compact Cartan subgroup.
\end{theorem}

Let $H/Z$ be a compact Cartan subgroup of $G/Z$.  Retain the notation of
\S \ref{ssec3d} for $(\widehat{G^0})_{disc}$ and the notation 
$G = \bigcup x_jG^\dagger$ where the $x_j$ normalize $H^0$.  Write
$w_j$ for the element of $W_G$ represented by $x_j$\,.

\begin{theorem}\label{3.5.9}
Let $\lambda \in L'$ and $[\chi] \in \widehat{Z_G(G^0)}_\xi$ where
$\xi = e^{\lambda - \rho}|_{Z_{G^0}}$\,.  Let $[\pi] \in (\widehat{G^0})_{disc}$
as in {\rm Theorem \ref{3.4.7}}.  Then 
$[\pi_{\chi,\lambda}]:= [\Ind_{G^\dagger}^G(\chi\otimes\pi_\lambda)]$ is the
unique class in $\widehat{G}_{disc}$ whose distribution character satisfies
\begin{equation}\label{3.5.10}
\Theta_{\pi_{\chi,\lambda}}(zh) = 
\sum_{1\leqq j\leqq r} (-1)^{q(w_j\lambda)}\trace \chi(x_j^{-1}zx_j) \cdot
\sum_{w \in W_{G^0}} \det(ww_j)e^{ww_j\lambda}(h)
\end{equation}
for $z \in Z_G(G^0)$ and $h \in H^0\cap G'$.  Every class in 
$\widehat{G}_{disc}$ is one of these $[\pi_{\chi,\lambda}]$.  Classes
$[\pi_{\chi,\lambda}] = [\pi_{\chi',\lambda'}]$ precisely when 
$([\chi'],\lambda') \in W_G([\chi],\lambda)$.  For appropriate normalizations
of Haar measures the formal degree 
$\deg(\pi_{\chi,\lambda}) = r\cdot\dim(\chi)\cdot |\varpi(\lambda)|$.
Finally, $[\pi_{\chi,\lambda}]$ has dual $[\pi_{\chi^*,-\lambda}]$ and has
infinitesimal character $\chi_\lambda$ as in {\rm (\ref{3.2.5})}, so in
particular $\chi_{[\pi_{\chi,\lambda}]}(Casimir) = ||\lambda||^2-||\rho||$.
\end{theorem}

\begin{proof} If $\widehat{G}_{disc}$ is nonempty then Propositions \ref{3.5.2}
and \ref{3.5.5} show that $(\widehat{G^0})_{disc}$ is not empty, so 
$G^0/(Z\cap G^0)$ has a compact Cartan subgroup by Theorem \ref{3.4.1}.  As
$G^0/(Z\cap G^0)$ has finite index in $G/Z$ the latter also has a compact 
Cartan subgroup.  If $H/Z$ is a compact Cartan subgroup of $G/Z$ then
Theorem \ref{3.5.9} follows directly from Theorem \ref{3.4.7} and
Propositions \ref{3.5.2} and \ref{3.5.5}.
\end{proof}

\section{Tempered Series Representations of Reductive Lie Groups}
\label{sec4}
\setcounter{equation}{0}
$G$ is a reductive Lie group of the class described in \S\ref{ssec3a}.
In \S\ref{sec3} we used the conjugacy class of Cartan subgroups $H$ of $G$,
with $H/Z$ compact, to construct the relative discrete series
$\widehat{G}_{disc}$\,.  Here we construct a series of unitary representations
for every conjugacy class of Cartan subgroups of $G$.

In \S 4.1 and 4.2 we work out the relation between Cartan involutions
$\theta$ of $G$, Cartan subgroups $H$ of $G$, and cuspidal parabolic
subgroups $P=MAN$ of $G$.  Here $H = T\times A$, $T/Z$ compact and 
$A$ split/$\R$, and $Z_G(A) = M\times A$ where $M$ is in the class of
\S\ref{ssec3b} and $T$ is a Cartan subgroup of $M$.  Then $H \mapsto P$
gives a bijection from the set of all 
conjugacy classes of Cartan subgroups of $G$ to the set of all 
``association classes'' of cuspidal parabolic subgroups of $G$.

In \S 4.3 we describe these representations, calculating infinitesimal and
distribution characters.  They form the ``$H$--series'' $\widehat{G}_H$\,.  
Then in \S 4.5  we examine the correspondence 
$\widehat{M}_{disc} \to \widehat{G}_H$\,.  Finally in \S 4.5 we look at
questions of irreducibility.

\subsection{}\label{ssec4a}\setcounter{equation}{0}
$G$ is a real reductive Lie group as in \S \ref{ssec3a}, $\gh$ is a Cartan
subalgebra of $\gg$, and $H = Z_G(\gh)$ is the corresponding Cartan
subgroup of $G$.  If $G^0$ is a linear group, or if $H/Z$ is compact, then
$H \cap G^0$ is commutative.  In general one only knows that $H^0 = 
\exp(\gh)$ is commutative.  We collect some information.

\begin{lemma}\label{4.1.1} If $K/Z$ is a maximal compact subgroup of $G/Z$
then there is a unique involutive automorphism $\theta$ of $G$ with
fixed point set $K$.
These automorphisms $\theta$ are the ``Cartan involutions'' of $G$, and
any two are $\Ad(G^0)$--conjugate. Every
Cartan subgroup of $G$ is stable under a Cartan involution.
\end{lemma}
\begin{lemma}\label{4.1.2} If $K/Z$ is a maximal compact subgroup of $G/Z$
then $K^0 = K\cap G^0$, $K$ meets every component of $G$, and
$K = \{g \in G \mid \Ad(g)K^0 = K^0\}$.
\end{lemma}
These lemmas are standard when $Z = \{1\}$ and $G$ is either 
linear or semisimple.
\begin{proof} $Z_{G^0} \subset K$ and $(K\cap G^0)/(Z_{G^0})^0$ is connected,
is its own $G^0/(Z_{G^0})^0$--normalizer, and is unique up to conjugacy.
The same follows for $K\cap G^0$ in $G^0$.  Let
$E = \{g \in G \mid \Ad(g)K^0 = K^0\}$; now $E \cap G^0 = K^0 = K \cap G^0$.
If $g \in G$ now some $g' \in G^0$ send $\Ad(g)K^0$ to $K^0$, so $E$ meets
$gG^0$.  Now $K \subset E \subset K$ and Lemma \ref{4.1.2} follows.  For
Lemma \ref{4.1.1} each simple ideal $\gg_i \subset \gg$ has a unique
involution $\theta_i$ with fixed point set $\gg_i\cap\gk$\,, and we define
$\theta$ as their sum with the identity map on the center of $\gg$.  Then
$\theta$ extends uniquely to the universal cover of $G^0$, and there its
fixed point set $\exp(\gk)$ contains the center, so $\theta$ extends
uniquely to $G^0$ with fixed point set $K^0$.  Now $\theta$ extends 
uniquely to $G = KG^0$ with fixed point set $K$, using Lemma \ref{4.1.2}.
As any two choices of $K/Z$ are $\Ad(G^0)$--conjugate, that completes the
proof of the first statement of Lemma \ref{4.1.1}.  For the second just note
that any two choices of $\gk$ are $\Ad(G^0)$--conjugate.
\end{proof}

Now fix the data:
a Cartan subgroup $H$ of $G$, 
a Cartan involution $\theta$ of $G$ with $\theta(H) = H$,
and $K = G^\theta$ fixed point set of $\theta$.
We decompose
\begin{equation}\label{4.1.4}
\begin{aligned}
&\gh = \gt + \ga \text{ into } (\pm 1)\text{--eigenspaces of } \theta|_\gh \\
&H = T \times A \text{ where } T = H \cap K \text{ and } A = \exp(\ga).
\end{aligned}
\end{equation}
The $\ga$--{\em root spaces} of $\gg$ are the $\gg^\phi = 
\{\xi \in \gg \mid [\alpha,\xi] = \phi(\alpha)\xi \text{ for } \alpha \in \ga\}$
with $0 \ne \phi \in \ga^*$ and $\gg^\phi \ne 0$.  The $\ga$--roots are
these functionals $\phi$, and $\Sigma_\ga$ denotes the set of all $\ga$--roots.
The corresponding $\ga$--root decomposition is
$\gg = \gz_\gg(\ga) + \sum_{\phi \in \Sigma_\ga} \gg^\phi$ where
$\gz_\gg(\ga)$ is the centralizer of $\ga$ in $\gg$.
Then it is not too difficult to see that
the centralizer $Z_G(A)$ has a unique splitting $Z_G(A)=M\times A$ with
$\theta(M) = M$.  In particular 
$\gg=\gm + \ga + \sum_{\phi \in \Sigma_\ga} \gg^\phi$ with $\theta(\gm) = \gm$.
The hereditary properties of \S 3.1 pass down from $G$ to $M$ as follows.

\begin{proposition}\label{4.1.6}
$M$ inherits the conditions of {\rm \S 3.1} from $G$: every $\Ad(m)$ is inner on
$\gm_\C$\,, $Z$ centralizes $M^0$, $|M/ZM^0| < \infty$, and 
$Z_{M^0}/(Z\cap M^0)$ is compact.  Further, $T/Z$ is a compact Cartan subgroup
of $M/Z$.
\end{proposition}
The proof of Proposition \ref{4.1.6} requires some information on the
$\ga$--root system.

Every $\phi \in \Sigma_\ga$ defines $\phi^\perp := \{\alpha \in \ga \mid 
\phi(\alpha) = 0\}$. The complement 
$\ga \setminus \bigcup_{\Sigma_\ga}\phi^\perp$ is a finite union of convex 
open cones, its topological components, the {\em Weyl chambers}.  A 
Weyl chamber $\gd \subset \ga$ defines a {\em positive root system}
$\Sigma_\ga^+ = \{\phi \in \Sigma_\ga \mid \phi(\gd) \subset \R^+\}$.
\begin{lemma}\label{4.1.7}
If $\Sigma_\ga^+$ is a positive $\ga$--root system on $\gg$ and $\Sigma_\gt^+$
is a positive $\gt_\C$--root system on $\gm_\C$ then there is a unique
positive $\gh_\C$--root system $\Sigma^+$ on $\gg_\C$ such that
$\Sigma_\ga^+ = \{\gamma|_\ga \mid \gamma \in \Sigma^+ \text{ and }
\gamma|_\ga \ne 0\}$ and $\Sigma_\gt^+ = 
\{\gamma|_\gt \mid \gamma \in \Sigma^+ \text{ and } \gamma|_\ga  = 0\}$.
\end{lemma}
\begin{proof}  Choose ordered bases $\beta_\ga$ of $\ga^*$ and $\beta_{\gt}$
of $i\gt^*$ whose associated lexicographic orders give $\Sigma_\ga^+$ and
$\Sigma_\gt^+$\,.  Then the ordered basis $\{\beta_\ga,\beta_{\gt}\}$ of
$\ga^* +i\gt^*$ gives a lexicographic order whose associated positive
$\gh_\C$--root system $\Sigma^+$ has the required properties.  Uniqueness
of $\Sigma^+$ is similarly straightforward.
\end{proof}

{\sc Proof of Proposition \ref{4.1.6}.}  $H \subset M \times A$ is a Cartan
subgroup of $G$, hence also of $M\times A$, so $T$ is a Cartan subgroup
of $M$.  $T/Z$ is compact because $K/Z$ is compact.

Let $m \in M$ with $\Ad(m)$ outer on $\gm_\C$.  We may move $m$ within
$mM^0$ and assume $\Ad(m)\gt = \gt$. As the Weyl group of $(M^0,T^0)$ is
simply transitive on the Weyl chambers in $i\gt$, $\Ad(m)$ preserves and
acts nontrivially on some $\Sigma_\gt^+$\,.  Choose $\Sigma_\ga^+$\,; now
$\Ad(m)$ preserves and acts nontrivially on the corresponding positive
$\gh_\C$--root system of $\gg_\C$\,, in other words $\Ad(m)$ is outer on
$\gg_\C$\,.  That contradicts (\ref{3.1.1}).  Thus $\Ad(m)$ is inner on
$\gm_\C$\,.  In particular $M^\dagger := \{m \in M \mid \Ad(m) \text{ inner
on } M^0\}$ has finite index in $M$ and $M^\dagger = TM^0$.
Now $T/Z$ is a compact subgroup of $M/Z$ such that $M/T = (M/Z)/(T/Z)$ has
only finitely any components.  Thus $|M/ZM^0| < \infty$.  The center of
$M^0/(Z\cap M^0)$ is compact because it is a closed subgroup in the  torus
$T^0/(Z\cap M^0)$, so that center is $Z_{M^0}/(Z\cap M^0)$.
\hfill $\square$

\subsection{}\label{ssec4b}\setcounter{equation}{0}
We apply the considerations of \S\ref{ssec4a} to study cuspidal parabolic
subgroups.  

Retain the splittings (\ref{4.1.4}) and $Z_G(A) = M\times A$.  
A positive $\ga$--root system $\Sigma_\ga^+$ on $\gg$ defines
$$
\gn = \sum_{\phi \in \Sigma_\ga^+}\gg^{-\phi} \text{ nilpotent subalgebra of }
	\gg\,, N = \exp(\gn), P = \{g \in G \mid \Ad(g)N = N\}.
$$
\begin{lemma}\label{4.2.2}
$P$ is a real parabolic subgroup of $G$.  It has unipotent radical
$P^{unip} = N$ and Levi (reductive) part $P^{red} = M\times A$.  Thus
$(m,a,n) \mapsto man$ is a real analytic diffeomorphism of $MAN$ onto $P$.
\end{lemma}
\begin{proof} We follow the idea of the proof for the case where
$G = G^\dagger$ and $G$ is linear, which follows from the complex case.
Let $\pi: G \to \overline{G} = G/ZZ_{G^0}$\,.  That is a real linear
algebraic group, so $\pi(P)$ is a parabolic subgroup of $\overline{G}$,
normalizer of $\pi(N)$, and $\pi(N) = \pi(P)^{unip}$.

Note $ZZ_{G^0} \subset MA \subset P$.  Thus (i) we can choose $\pi(P)^{red}$
to contain $\pi(MA)$ and (ii) $P = P^{red}\cdot N$ semidirect where
$P^{red} = \pi^{-1}(\pi(P)^{red})$.  Now $MA \subset P^{red}$ and, by 
dimension, $M^0A = (P^{red})^0$.  $A$ is normal in $P^{red}$ by uniqueness
from (\ref{4.1.4}).

Let $V = \{w \in W(\gg_\C,\gh_\C) \mid w(\ga) = \ga\,, w|_\ga \ne 1\}$.
Choose $x \in \ga \setminus \bigcup_{v \in V} \{y \in \ga \mid v(y)=y\}$.
Let $g \in P^{red}$.  Then $x, \Ad(g)x$ are conjugate by an inner 
automorphism of $\gg_\C$, thus \cite[Theorem 2.1]{R1972} conjugate by an
inner automorphism on $\gg$.  Now $\Ad(g')x = x$ for some $g' \in gG^0$,
and $\Ad(g')M = M$.  We may assume $\Ad(g')H = H$.  Now $g' \in MA$.  Thus
$P^{red} = MA(P^{red}\cap G^0)$.  As $P^{red}\cap G^0 \in Z_G(A)$ now
$P^{red} = MA$ as asserted.
\end{proof}

We say that two parabolics in $G$ are {\em associated} if their
reductive parts are $G$--conjugate.  Thus the association class of
$P = MAN$ is independent of $N$, i.e. is independent of $\Sigma_\ga^+$\,.
We say that a parabolic $Q \subset G$ is {\em cuspidal} if
$[(Q^{red})^0,(Q^{red})^0]$ has a Cartan subgroup $E$ such that
$E/(E\cap Z_{G^0})$ is compact.

\begin{proposition}\label{4.2.3}
Let $Q$ be a parabolic subgroup of $G$.  Then the following are equivalent.

{\rm (i)} $Q$ is a cuspidal parabolic subgroup of $G$.

{\rm (ii)} There exist a Cartan subgroup $H = T\times A$ of $G$ and a positive
$\ga$--root system\\
\phantom{XXXXI} $\Sigma_\ga^+$ such that $Q$ is the group $P = MAN$
of {\rm Lemma \ref{4.2.2}}.

{\rm (iii)} $Q^{red}$ has a relative discrete series representation.

{\rm (iv)} $(Q^{red})^0$ has a relative discrete series representation.

In particular, the construction $H \mapsto P = MAN$ induces a bijection from
the set of all conjugacy classes of Cartan subgroups of $G$ onto the 
set of all association classes of cuspidal parabolic subgroups of $G$.
\end{proposition}

\begin{proof}
Let $\pi: G \to \overline{G} = G/ZZ_{G^0}$ as before.  Note
$ZZ_{G^0} \subset Q^{red}$.  Thus (i), (ii), (iii) each holds for $Q$
exactly when it holds for $\pi(Q)$.  Also, (iii) and (iv) are equivalent
because $Q^{red}/ZZ_{G^0}$ has only finitely many components.  Thus we need
only prove the equivalence of (i), (ii) and (iv) when $G$ is a connected
centerless semisimple group.

Decompose $Q^{red} = M_Q \times A_Q$\,, stable under $\theta$, where $A_Q$
is the $\R$--split component of the center of $Q^{red}$.  Thus $Q$ is 
cuspidal if and only if $M_Q$ has a compact Cartan subgroup $T_Q$\,.
That is the case just when $G$ has a Cartan subgroup $H = T_Q \times A_Q$
from which $Q$ is constructed as in Lemma \ref{4.2.2}.  Thus (i) and (ii)
are equivalent.  Apply Theorem \ref{3.4.1} to $(Q^{red})^0$.  Then (ii) 
implies (iv) by Proposition \ref{4.1.6} and (iv) implies (i) directly.
Now (i), (ii), (iii) and (iv) are equivalent, and the bijection statement
follows.
\end{proof}

Two Cartan subalgebras of $\gg$ are conjugate by an inner automorphism of
$\gg$ precisely when they are conjugate by an inner automorphism of $\gg_\C$
\cite[Corollary 2.4]{R1972}.  By Proposition \ref{4.2.3} the same holds
for association classes of cuspidal parabolic subgoups of $G$.  Thus we
could use $G^0$--conjugacy, $G^0$--association, or both, in the bijection
of Proposition \ref{4.2.3}.

\subsection{}\label{ssec4c}\setcounter{equation}{0}
We define a series of unitary representations of $G$ for each conjugacy class
of Cartan subgroups.  Then we work out some generalities on the character
theory for that series.  The precise character theory is in \S \ref{ssec4d}.

Retain the notation of \S\S \ref{ssec4a} and \ref{ssec4b}, including
$H = T \times A$, $Z_G(A) = M\times A$ and $P = MAN$.  The general unitary
equivalence class in $\widehat{P^{red}} = \widehat{M\times A}$ has form
$[\eta\otimes e^{i\sigma}]$ where $[\eta] \in \widehat{M}$ and
$\sigma \in \ga^*$.  That extends to a class $[\eta\otimes e^{i\sigma}]
\in \widehat{P}$ that annihilates $N$: $(\eta\otimes e^{i\sigma})(man) =
e^{i\sigma}(a)\eta(m)$.  Then we have the unitarily induced representation
\begin{equation}\label{4.3.1}
\pi_{\eta,\sigma} = \Ind_P^G(\eta\otimes e^{i\sigma})
\end{equation}
of $G$.  The {\em $H$--series} of $G$ is $\{[\pi_{\eta,\sigma}] \mid
[\eta] \in \widehat{M}_{disc} \text{ and } \sigma \in \ga^*\}$.  If
$H/Z$ is compact then $M = G$ and the $H$--series is just the relative 
discrete series $\widehat{G}_{disc}$\,.  If $H/ZZ_{G^0}$ is maximally
$\R$--split, i.e. if $P$ is a minimal parabolic subgroups of $G$, then
the $H$--series is the principal series.  We refer to any $H$--series
as a {\em nondegenerate series}.  Later we will see that the $H$--series
depends only on the conjugacy class of $H$.

Given $\Sigma_\ga^+$ define 
$\rho_\ga = \tfrac{1}{2}\sum_{\phi \in \Sigma_\ga^+}\,(\dim \gg^\phi)\phi$.
Then $\ga$ acts (under the adjoint representation of $\gg$) on $\gn$ and
on $\gp$ with trace $-2\rho_\ga$\,.  Thus $P = MAN$ has modular function
$\delta_P$\,, $\int_P f(xy^{-1})dx = \delta_P(y)\int_P f(x)dx$, given by
\begin{equation}\label{4.3.2}
\delta_P(man) = e^{2\rho_a}(a) \text{ for } m \in M,\, a \in A,\, n \in N.
\end{equation}
Let $[\eta] \in \widehat{M}$ with representation space $E_\eta$\,, and let
$\sigma \in \ga^*$.  Then we have the Hilbert space bundle
$\E_{\eta,\sigma} \to G/P = K/(K\cap M)$ associated to the non--unitary
representation $\eta \otimes e^{\rho_\ga + i\sigma}$ of $P$.  Here $G$
acts on the bundle but the hermitian metric is invariant only under $K$.
We have the $K$--invariant probability measure $d(kZ)$ on $G/P = K/(K\cap Z)$.
Thus we have a well defined  space of square integrable sections of
$\E_{\eta,\sigma} \to G/P$ given by
\begin{equation}\label{4.3.3}
\begin{aligned}
L_2(G/P;&\E_{\eta,\sigma}) = \text{ all Borel--measurable } f:G \to E_\eta
	\text{ such that } \\
&f(gp) = (\eta\otimes e^{\rho_\ga + i\sigma})(p)^{-1}f(g) \text{ and }
	{\int}_{K/Z} ||f(k)||^2 d(kZ) < \infty.
\end{aligned}
\end{equation}
It is a Hilbert space with inner product $\langle f, f'\rangle =
{\int}_{K/Z} \langle f(k),f'(k)\rangle d(kZ)$, and $G$ acts unitarily on
it by the representation 
$(\pi_{\eta,\sigma}(g)(f))(g') = f(g^{-1}g')$ of (\ref{4.3.1}).

\subsection{}\label{ssec4d}\setcounter{equation}{0}
We now describe the distribution character $\Theta_{\pi_{\eta,\sigma}}$
of $\pi_{\eta,\sigma}$ in terms of the character $\Psi_\eta$ of $\eta$.
This is based on a minor variation $C^\infty_c(G) \to C^\infty_c(MA)$ 
of the Harish--Chandra transform.

Let $J$ be a Cartan subgroup of $G$.  Then these are equivalent: (i) 
$J \subset MA$, (ii) $J$ is a Cartan subgroup of $MA$, and (iii)
$J = J_M \times A$ where $J_M = J \cap M$ is a Cartan subgroup of $M$.
Without loss of generality we may assume $J_M$ stable under the Cartan
involution $\theta|_M$ of $M$.  Choose a positive $(\gj_M)_\C$--root
system $\Sigma_{j_M}^+$ on $\gm_\C$.  As in Lemma \ref{4.1.7} there is a 
unique positive $\gj_\C$--root system $\Sigma_{\gj}^+$ on $\gg_\C$ such that
$\Sigma_\ga^+ = \{\gamma|_\ga \mid \gamma \in \Sigma_\gj^+ \text{ and }
\gamma|_\ga \ne 0\}$ and $\Sigma_{\gj_m}^+ = \{\gamma|_{\gj_M} \mid
\gamma \in \Sigma_\gj^+ \text{ and } \gamma|_\ga = 0\}$.  Then let
$\rho_\gj = \tfrac{1}{2}\sum_{\gamma \in \Sigma_\gj^+}\,\gamma$\,; so
$\rho_\ga = \rho_\gj|_\ga$\,.  Now define
\begin{equation}\label{4.3.4}
\Delta_{G,J} = {\prod}_{\gamma \in \Sigma_\ga^+} (e^{\gamma/2} - 
	e^{-\gamma/2}) \text{ and } \Delta_{M,J_M} =
	{\prod}_{\phi \in \Sigma_{\gj_M}^+} (e^{\phi/2} - e^{-\phi/2}).
\end{equation}
\begin{lemma}\label{4.3.5}
If $\gamma$ is a $\gj_\C$ root then $e^\gamma$ is well defined on $J$, 
unitary on $J\cap K$, and $e^\gamma(Z_G(G^0)) = 1$.  If $\phi$ is a
$(\gj_M)_\C$--root, say $\phi = \gamma|_{\gj_M}$\,, then $e^\phi =
e^\gamma|_{J\cap M}$\,.
\end{lemma}
\begin{lemma}\label{4.3.6}
We can replace $Z$ by a subgroup of index $\leqq 2$, or replace $G$ by a
$\Z_2$ extension, so that the following holds.  If $L$ is any Cartan subgroup
of $G$ then for any positive $\gl_\C$--root system (i) $e^{\rho_\gl}$
is well defined on $L$ with $e^{\rho_\gl}(Z) = 1$, and (ii) $\Delta_{G,L}$
is a well defined analytic function on $L$.  In particular, then, 
$e^{\rho_\gj}$ and $\Delta_{G,J}$ are well defined on $J$, so 
$e^{\rho_{\gj_M}}$ and $\Delta_{M,J_M}$ are well defined on $J_M$\,.
\end{lemma}

{\sc Proof of Lemmas.} If $\gamma$ is a $\gj_\C$ root then $e^\gamma$ is 
well defined on the Cartan subgroup $(J/Z_G(G^0))_\C$ of the inner 
automorphism group $\Int(\gg_\C)$, because that Cartan is connected.
Lemma \ref{4.3.5} follows.  

With the adjustments of Lemma \ref{4.3.6}
we can factor $\Ad: G \to \Int(\gg_\C)$ as $G \to G/Z \to Q \to \Int(\gg_\C)$
where $q:Q \to \Int(\gg_\C)$ is a $1$-- or $2$--sheeted covering with
$e^{\rho_\gj}$ well defined on $q^{-1}(J/Z_G(G^0))_\C$ and
$e^{\rho_\gj}(Z) = 1$.  As any two Cartan subgroups of $\Int(\gg_\C)$ are
conjugate we have $\overline{x} \in \Int(\gg_\C)$ such that 
$\Ad(\overline{x})(L/Z_G(G^0))_\C = (J/Z_G(G^0))_\C$ and 
$\Ad(\overline{x})^*\rho_\gj = \rho_\gl$\,.  Now $e^{\rho_\gl}$ is
well defined on $q^{-1}(L/Z_G(G^0))_\C$\,, thus is well defined on $L$
with $Z$ in its kernel.  So also $\Delta_{G,L} = 
e^{-\rho_\gl}\cdot \prod_{\Sigma_\gl^+} (e^\gamma - 1)$ is well defined
on $L$.

$Q$ was defined so that it has a faithful irreducible holomorphic
representation $\psi$ of highest weight $\rho_\gj$ relative to
$(\gj,\Sigma_\gj^+)$.  Realize $\psi$ as a subrepresentation of the
left multiplication action $\lambda$ of the Clifford algebra on the 
Lie algebra of $\Int(\gg_\C)$.  The Clifford subalgebra for $M$ is
stable under $\lambda(q^{-1}(\Ad_\gg(M)))$, and the corresponding 
representation of $M$ has an irreducible summand of highest weight
$\rho_{\gj_M}$\,.  Now 
$e^{\rho_{\gj_M}}$ and $\Delta_{M,J_M}$ are well defined on $J_M$\,.
\hfill $\square$

$G'$ denotes the $G$--regular set in $G$, $M''A$ is the
$MA$--regular set in $MA$, and 
\begin{equation}\label{car}
\begin{aligned}
&\Car(G): \text{ the $G$--conjugacy classes of Cartan subgroups of $G$ and}\\
&\Car(MA): \text{ the $MA$--conjugacy classes of Cartan subgroups of $MA$}.
\end{aligned}
\end{equation}
Then
\begin{equation}\label{4.3.7}
\begin{aligned}
&G' = {\bigcup}_{L \in \Car(G)} G'_L \text{ where }
	G'_L = {\bigcup}_{g \in G}\Ad(g)(L\cap G') \text{ and }\\
&M''A = {\bigcup}_{J \in \Car(MA)} M''_JA \text{ where }
	M''_JA = {\bigcup}_{m \in M} (\Ad(m)(J\cap M''))A.
\end{aligned}
\end{equation}
The following theorem unifies and extends various results of Bruhat
\cite[Ch. III]{B1956}, Harish--Chandra (\cite[p. 544]{HC1970} and
\cite[\S 11]{HC1971}), Hirai \cite[Theorems 1, 2]{H1968} and Lipsman
\cite[Theorem 9.1]{L1971}.  We assume the adjustment of Lemma \ref{4.3.6}.
The specialization to $H$--series is in \S\ref{ssec4e}.

\begin{theorem}\label{4.3.8}
Let $\zeta \in \widehat{Z}$, $[\eta] \in \widehat{M}_\zeta$ and $\sigma \in
\ga^*$.  Let $\chi_\nu$ be the infinitesimal character of $[\eta]$
relative to $\gt$ and let $\Psi_\eta$ denote the distribution character
of $[\eta]$.

{\rm 1.} $[\pi_{\eta,\sigma}]$ has infinitesimal character 
$\chi_{\nu +i\sigma}$ relative to $\gh$.

{\rm 2.} $[\pi_{\eta,\sigma}]$ is a finite sum of classes from 
$\widehat{G}_\zeta$\,.  In particular $[\pi_{\eta,\sigma}]$ has distribution
character $\Theta_{\pi_{\eta,\sigma}}$ that is a locally summable function
analytic on the regular set $G'$.

{\rm 3.} $\Theta_{\pi_{\eta,\sigma}}$ has support in the closure of
${\bigcup} G'_J$ where $J$ runs over a system of representatives of the
$G$--conjugacy classes of Cartan subgroups of $MA$.

{\rm 4.} Let $J \in \Car(MA)$ and $\Xi(J)$ consist of all $G$--conjugates
of $J$ contained in $\Car(MA)$.  Enumerate 
$\Xi(J) = \{J_1,\dots, J_\ell\}$ with $J_i = \Ad(g_i)J$.
If $h \in J\cap G'$ denote $h_i = \Ad(g_i)h$.  Then
$$
\Theta_{\pi_{\eta,\sigma}}(h) = \frac{1}{|\Delta_{G,J}(h)|}
\sum_{J_i \in \Xi(J)}\sum_{h' \in N_G(J_i)h_i}
\frac{|\Delta_{MA,J_i}(h')|}{|N_{MA}(J_i)h'|}
\Psi_\eta(h'_M) e^{i\sigma}(h'_A).
$$
\{Note: If $h \in J^0$ then the second sum runs over the Weyl group orbit 
$W_{G,J_i}(h_i)$.  \}
\end{theorem}
\begin{corollary}\label{4.3.9}
The class $[\pi_{\eta,\sigma}]$ is independent of the choice of parabolic
subgroup $P=MAN$ associated to the Cartan subgroup $H=T\times A$ of $G$.
\end{corollary}

The proof of Theorem \ref{4.3.8} is based on the following 
minor variation of the Harish--Chandra transform
$C^\infty_c(G) \to C^\infty_c(MA)$.

\begin{proposition}\label{4.3.10}
Let $b \in C^\infty_c(G)$ and define 
\begin{equation}\label{4.3.11a}
b_P(ma) = 
e^{-\rho_\ga}(a)\int_{K/Z}\left ( \int_N b(kmank^{-1})dn\right )d(kZ).
\end{equation}
Then $b_P(C^\infty_c(MA)$, $\pi_{\eta,\sigma}(b)$ is of trace class, and
\begin{equation}\label{4.3.11b}
{\rm \text{trace\,\,}} \pi_{\eta,\sigma}(b) = 
	\int_{MA} b_P(ma)\Psi_\eta(m)e^{i\sigma}(a)\, dm\, da.
\end{equation}
\end{proposition}

\begin{proof} Let $K_1$ denote the image of a Borel section of $K \to K/Z$.
If $f \in L_2(G/P;\E_{eta,\sigma})$ is continuous it is determined by 
$f|_{K_1}$\,.  Compute
$$
\begin{aligned}
(\pi_{\eta,\sigma}&(b)f)(k') = \int_G b(g)f(g^{-1}k')dg = 
			\int_G b(k'g)g(g^{-1}dg \\
&= \int_{K/K\cap M} d(kM)\int_{MAN}  b(k'mank^{-1})f(k(man)^{-1})
			e^{-2\rho_\ga}(a) dm\, da\, dn \\
&= \int_{K/K\cap M}\left ( \int_{MAN} b(k'mank^{-1}) e^{-2\rho_\ga}(a)
			\eta(m) dm\, da\, dn \right ) f(k)d(kM).
\end{aligned}
$$
Define $\Phi_b(k',k) = \int_{MAN} b(k'mank^{-1}) e^{-\rho_\ga +i\sigma}(a)
\eta(m) dm\, da\, dn : E_\eta \to E_\eta$\,.  Then $\Phi_b(k',km_1)f(km_1)
= \Phi_b(k',k)f(k)$ for $m_1 \in M\cap K$.  Thus
$
(\pi_{\eta,\sigma}(b)f)(k') = \int_{K/K\cap M} \Phi_b(k',k) f(k) d(kM).
$
As $Z \subset K\cap M$ and $K/Z$ is compact we write this as
$(\pi_{\eta,\sigma}(b)f)(k') = \int_{K/Z} \Phi_b(k',k) f(k) d(kZ)$; so
$
\text{trace\,} \pi_{\eta,\sigma}(b) 
	= \int_{K/Z} \text{trace\,} \Phi_b(k,k) d(kZ).
$
Set $\varphi_b(k,m) = \int_{NA}e^{-\rho_\ga +i\sigma}b(kmank^{-1}) dn\, da$.
Then $\varphi_b \in C^\infty_c((K/Z)\times M)$.  Noting that we always have
absolute convergence we calculate
$$
\begin{aligned}
\int_{K/Z}\text{trace\,}& \Phi_b(k,k) d(kZ)
	= \int_{K/Z}\left ( \text{trace\,} \int_M \Phi_b(k,m)\eta(m)dm
		\right ) d(kz) \\
&= \int_{K/Z}\left ( \int_M \Phi_b(k,m)\Psi_\eta(m)dm\right ) d(kZ)\\
&= \int_{K/Z}\left ( \int_M \left ( \int_{NA} e^{-\rho_\ga +i\sigma}(a)
	b(kmank^{1})dn\,da\right )\Psi_\eta(m)dm\right ) d(kZ)\\
&= \int_M\left ( \int_A b_P(ma)e^{i\sigma}(a)\Psi_\eta(m)da\right ) dm.
\end{aligned}
$$
That completes the proof of Proposition \ref{4.3.10}
\end{proof}

Let $\cZ(\gg)$ and $\cZ(\gm+\ga)$ denote the respective centers of the 
enveloping algebras $\cU(\gg)$ and $\cU(\gm+\ga)$.  Recall the canonical
homomorphisms $\gamma_G$ and $\gamma_{MA}$ to Weyl group invariant polynomials.
Then \cite[\S 12]{HC1965a} $\mu_{MA} = \gamma_{MA}^{-1}\cdot\Gamma_G
\cZ(\gg) \to \cZ(\gm+\ga)$ has the property that $\cZ(\gm+\ga)$ has finite
rank over its subalgebra $\mu_{MA}(\cZ(\gg))$.

Proposition \ref{4.3.10} says trace\,$\pi_{\eta,\sigma}(b)$ =
trace\,$(\eta\otimes e^{i\sigma})(b_P)$ for $b \in C^\infty_c(G)$.  Now
Harish--Chandra's \cite[Lemma 52]{HC1965a} says 
trace\,$\pi_{\eta,\sigma}(zb)$ = trace\,$(\eta\otimes e^{i\sigma})
(\mu_{MA}(z)b_P)$ for $z \in \cZ(\gg)$.  From that, the infinitesimal
character $\chi_{\pi_{\eta,\sigma}}(z)$ = $\chi_{\eta\otimes e^{i\sigma}}(
\mu_{MA}(z))$ = $\chi_{\nu + i\sigma}(\mu_{MA}(z))$ ($\chi$ for 
$((\gm + \ga)_\C,\gh_\C)$) = $\gamma_{MA}(\mu_{MA}(z))(\nu + i\sigma)$
= $\gamma_G(z)(\nu + i\sigma)$ = $\chi_{\nu + i\sigma}(z)$ 
(for $(\gg_\C,\gh_\C)$).  That proves the first assertion of 
Theorem \ref{4.3.8}.

To see that $[\pi_{\eta,\sigma}]$ is a direct integral over $\widehat{G}_\zeta$
we set $G^1 = ZG^0$ and $M^1 = M\cap G^1$.  Then $[\eta] \in \widehat{M}_\zeta$
gives $[\eta^1] \in \widehat{M^1}_\zeta$ such that $[\eta]$ is a 
subrepresentation of $[\Ind_{M^1}^1 (\eta^1)]$.  Thus $\pi_{\eta,\sigma}$
is a subrepresentation of $\Ind_P^G((\Ind_{M^1}^M (\eta^1))\otimes e^{i\sigma})$
= $\Ind_{G^1} ^G( \Ind_{M^1AN}^{G^1} (\eta^1 \otimes e^{i\sigma})).$
Here $[\eta^1] \in \widetilde{M^1}_\zeta$\,, so 
$\Ind_{M^1AN}^{G^1} (\eta^1 \otimes e^{i\sigma})|_Z$ is a multiple of $\zeta$.
Then $[\pi_{\eta,\sigma}]$ is a direct integral over $\widehat{G}_\zeta$.

We now prove that there is an integer $n > 0$ such that
\begin{equation}\label{4.3.12}
\pi_{\eta,\sigma}|_K = {\sum}_{\widehat{K_\zeta}} m_\kappa\, \kappa
	\text{ where } 0 \leqq m_\kappa \leqq n\cdot\dim\kappa < \infty.
\end{equation}
It suffices to prove this for the finite index subgroup $ZG^0$, so we may assume
$Z$ central in $G$.  Then (\ref{4.3.3}) and the discussion just above
give us a $K$--equivariant injective isometry $r_K: L_2(G/P;\E_{\eta,\sigma}) 
\to L_2(K/Z;\zeta)$ by $r_K(f) \ f|_K$\,.  As $L_2(K/Z;\zeta) =
\sum_{\widehat{K}_\zeta} V_\kappa \otimes V_\kappa^*$ the multiplicity of
$V_\kappa$ here is $\dim((V_\kappa^*\otimes E_\eta)^{M\cap K})$.  But 
$\eta|_{M\cap K} = \sum_{\widehat{(M\cap K)}_\zeta} m_i\mu_i$ where
$0 \leqq m_i \leqq _M\dim \mu_i < \infty$.  If $\kappa \in \widehat{K}_\zeta$
then $\dim \kappa < \infty$ so $\kappa_{M\cap K}$ is a finite sum
$\sum_{\widehat{M\cap K}_\zeta} m_{\kappa,i}\mu_i$\,.  Now
$$
\dim((V_\kappa^*\otimes E_\eta)^{M\cap K}) 
	= \sum m_{\kappa,i}\mu_i
	\leqq n_M \sum m_{\kappa,i}\dim \mu_i
	= n_M\dim \kappa < \infty,
$$
proving (\ref{4.3.12}).  Note from the proof that $n \leqq n_M|G/ZG^0|$.

We show that $[\pi_{\eta,\sigma}]$ is a finite sum from $\widehat{G}_\zeta$\,,
following Harish--Chandra.  The discussion of (\ref{3.2.2}) shows that we
need only consider the case where $G$ is connected.  Then $Z$ is central so
$[\pi_{\eta,\sigma}]$ has central character $\zeta$ and infinitesimal character
$\chi_{\nu + i \sigma}$\,.  By (\ref{4.3.12}) $\pi_{\eta,\sigma}$ has no
nontrivial subrepresentation of infinite multiplicity.  Thus it is 
quasi--simple in the sense of Harish--Chandra \cite[p. 145]{HC1956c}.
Consequently it has distribution character $\Theta_{\pi_{\eta,\sigma}}$ that
is a locally summable function analytic on the regular set $G'$ 
\cite[Theorem 6]{HC1956c}, and $\pi_{\eta,\sigma} = \sum \pi_j$ discrete 
direct sum of irreducibles \cite[Lemma 2]{HC1954c}.  Each $[\pi_j] \in 
\widehat{G}_\zeta$ and each $\chi_{\pi_j} = \chi_{\nu + i\sigma}$\,.  Further,
the differential equations $z\Theta = \chi_{\nu + i\sigma}(z)\Theta$ 
$(z \in \cZ(\gg))$ constrain
the $\Theta_{\pi_j}$ to a finite dimensional space of functions on $G$.
Inequivalent classes in $\widehat{G}_\zeta$ have linearly independent 
distribution characters, so $\pi_{\eta,\sigma} = \sum \pi_j$ involves only
finitely many classes from $\widehat{G}_\zeta$\,.  Since the multiplicities 
$m(\pi_j,\pi_{\eta,\sigma}) < \infty$, $\pi_{\eta,\sigma}$ is a finite sum
from $\widehat{G}_\zeta$.  That is the second assertion of Theorem \ref{4.3.8}.

We now calculate $\Theta_{\pi,\sigma}$ by extending Lipsman's argument
\cite[Theorem 9.1]{L1971} to our more general situation.  Recall (\ref{4.3.7})
and the definition (\ref{4.3.11a}) of $b_P$\,.
\begin{lemma}\label{4.3.13}
Let $L \in \Car(G)$ not conjugate to a Cartan subgroup of $MA$ and
$b \in C^\infty_c(G'_L)$.  Then $b_P = 0$.  On the other hand, 
if $J \in \Car(MA)$ and
$b \in C^\infty_c(G'_J)$ then $b_P \in C^\infty_c(MA\cap G'_J)
\subset C^\infty_c((MA)''_J)$.
\end{lemma}
\begin{proof}
If $ma \in MA$ and $(\Ad(ma)-1)^{-1}$ is nonsingular on $N$ then
\cite[Lemma 11]{HC1966a} gives us 
$
\int_N b(kmank^{-1})dn=|\det(\Ad(ma)^{-1} - 1)_\gn|\int_N b(knman^{-1}k^{-1})dn.
$
If $b \in C^\infty_c(G'_L)$ where $G'_L$ doesn't meet $MA$ then 
$b(knman^{-1}k^{-1})$ is identically zero, so $b_P = 0$.  If
$b \in C^\infty_c(G'_J)$ where $J \in \Car(MA)$ there is a compact set
$S \subset G'_J$ such that, if $\int_N b(kmank^{-1})dn \ne 0$ for some
$k \in K$ then $ma \in S$.  Thus $b_P$ is supported in
$S \cap MA \subset G'_J\cap MA \subset (MA)_J''$\,.
\end{proof}

Let $L \in Car(G) \setminus \Car(MA)$.  Let
$b \in C^\infty_c(G'_L)$. Combine Proposition \ref{4.3.10} and
Lemma \ref{4.3.13} to see 
$
\Theta_{\pi_{\eta,\sigma}}(b) =
\int_{MA} b_P(ma)\Psi_\eta(m) e^{i\sigma}(a) dmda = 0.
$  
Thus $\Theta_{\pi_{\eta,\sigma}}|_{G'_L} = 0$.  That is the third
assertion of Theorem \ref{4.3.8}.

Fix $J = J_M \times A \in \Car(MA)$.  To compute 
$\Theta_{\pi_{\eta,\sigma}}|_{G'_J}$ we need a variation on the
Weyl Integration Formula.  The center $Z_J$ of $J$ is open in $J$
so it inherits Haar measure $dh$.  Normalize measure on
$G/Z_J$ by $\int_G f(g)dg = \int_{G/Z_J}\bigl ( \int_{Z_J} f(gh)dh
\bigr ) d(gZ_J)$ and on $MA/Z_J$ by
$\int_{MA} F(x)dx = \int_{MA/Z_J}\bigl ( \int_{Z_J} F(xh)dh
\bigr ) d(xZ_J)$.
Extending Harish--Chandra's extension 
\cite{HC1966} of Weyl's argument, 
\begin{lemma}\label{4.3.14} If $b \in C_c(G'_J)$ and
$B \in C_c((MA)_J'')$ then 

\noindent
$\int_G b(g)dg = \int_{J\cap G'} |N_G(J)(h)|^{-1}\bigl (
\int_{G/Z_J} b(ghg^{-1})d(gZ_J)\bigr ) |\Delta_{G,J}(h)|^2 dh$ and

\noindent
$\int_{MA} B(x)dx = \int_{J\cap (MA)''} |N_{MA}(J)(h)|^{-1}\bigl (
\int_{MA/Z_J} b(xhx^{-1})d(xZ_J)\bigr ) |\Delta_{MA,J}(h)|^2 dh$.
\end{lemma}
\begin{proof}
$(G/Z_J) \times (J\cap G')\to G'_J$, by $(gZ_J,h) \mapsto ghg^{-1}$,
is regular, surjective, and $|N_G(J)(h)|$ to one with Jacobian
determinant $|\det(\Ad(h)^{-1}-1)_{\gg/\gj}|
= |\prod_{\gamma \in \Sigma_\gj}(e^\gamma - 1)(h)|$ at
$(gZ_J,h)$.  But $\prod_{\gamma \in \Sigma_\gj}(e^\gamma - 1)(h)$ is
the product over $\Sigma_\gj^+$ times the product over $-\Sigma_\gj^+$,
which is $(-1)^n\Delta_{G,J}(h)^2$ where $n = |\Sigma_\gj^+|$, so that
Jacobian is $|\Delta_{G,J}(h)|^2$.  That proves the first equation;
the second is similar.
\end{proof}

Given $b \in C_c^\infty(G'_J)$, 
$\Theta_{\pi_{\eta,\sigma}}(b) = 
\int_{MA} b_P(ma)\Psi_\eta(m)e^{i\sigma}(a)dmda$
by Proposition \ref{4.3.10}.  Lemma \ref{4.3.13} ensures convergence.
Now, by Lemma \ref{4.3.14},
\begin{equation}\label{4.3.16}
\begin{aligned}
\Theta_{\pi_{\eta,\sigma}}(b) = &\int_{J\cap G'} |N_{MA}(J)(h)|^{-1} \\
&\Bigl ( \int_{MA/Z_J} b_P(xhx^{-1}d(xZ_J)\Bigr ) \Psi_\eta(h_M)
e^{i\sigma}(h_A)|\Delta_{MA,J}(h)|^2dh
\end{aligned}
\end{equation}
where $h = h_Ah_M$ along $J = J_M\times A$.  As $A$ is central in $MA$,
$\int_{MA/Z_J} b_P(xhx^{-1})d(xZ_J)$
= $e^{-\rho_\ga (h_A)}\int_{MA/Z_J}d(xZ_J)\cdot\int_{K/Z}d(kZ)\cdot
\int_N b(k.xhx^{-1}\cdot n\cdot k^{-1})dn.$  Unimodularity of $N$ and
\cite[Lemma 11]{HC1966} say 
$$
\int_N b(k\cdot maha^{-1}m^{-1}\cdot n\cdot k^{-1})dn =       
|\det(\Ad(h^{-1})-1)_\gn|\int_Nb(knma\cdot h\cdot (knma)^{-1}dn.
$$

We modify Harish--Chandra's evaluation \cite[Lemma 12]{HC1966}
of $|\det(\Ad(h^{-1})-1)_\gn|$ for the case $J = H$.  Choose $c \in
\Int(\gg_\C)$ with $c(\gg_\C) = \gh_\C$\,, $c(x) = x$ for 
$x \in \ga$, and $c^*\Sigma^+ = \Sigma^+_\gj$\,, so $c$ also
preserves $\gm_\C$ and $\gn_\C$\,.  Then
$\det(\Ad(h^{-1})-1)_\gn$ = $\det(\Ad(c(h^{-1}))-1)_\gn$
= $\prod_{\gamma \in \Sigma^+\setminus \Sigma^+_\gt} (e^\gamma(c(h^{-1}))-1$
= $\prod_{\beta \in \Sigma^+_\gj \setminus \Sigma^+_{\gj_M}} 
	(e^\beta(h))-1)$
=$e^{\rho_\gj - \rho_{\gj_M}}(h)\frac{\Delta_{G,J}(h)}{\Delta_{MA,J}(h)}$
= $e^{\rho_\ga}(h_A) \frac{\Delta_{G,J}(h)}{\Delta_{MA,J}(h)}$, so
$\int_{MA/Z_J} b_P(xhx^{-1})d(xZ_J)$ is equal to
$$
\begin{aligned}
\int_{MA/Z_J}d(xZ_J)\cdot & \int_{K/Z}d(kZ)\cdot 
	|\Delta_{G,J}(h)/\Delta_{MA,J}(h)| \int_N b(\Ad(knma)h)dn\\
&= |\tfrac{\Delta_{G,J}(h)}{\Delta_{MA,J}(h)}|\int_{K/K\cap M} \Bigl (
	\int_{MNA/Z_J} b(\Ad(kmna)h)dm\,da\,dn\Bigr ) d(kM) \\
&= |\Delta_{G,J}(h)/\Delta_{MA,J}(h)|\int_{G/Z_J} b(ghg^{-1}) d(gZ_J).
\end{aligned}
$$
Substituting this into (\ref{4.3.16}),
\begin{equation}\label{4.3.18}
\begin{aligned}
\Theta_{\pi_{\eta,\sigma}}&(b) = \int_{J\cap G'} |N_{MA}(J)(h)|^{-1} \\
&\Bigl ( \int_{G/Z_J} b(ghg^{-1})d(gZ_J)\Bigr ) \Psi_\eta(h_M)
e^{i\sigma}(h_A)|\Delta_{G,J}(h)\Delta_{MA,J}(h)|^2dh
\end{aligned}
\end{equation}
We extend $\Phi_{\eta,\sigma,J}(h) := \frac{1}{|\Delta_{G,J}(h)|}
\sum_{h' \in N_G(J)h} \frac{|\Delta_{MA,J}(h')|}{|N_{MA}(J)h'|}
\Psi_\eta(h'_M)e^{i\sigma}(h'_A)$ 
to a class function on $G'_J$ and substitute that into (\ref{4.3.18}).
Thus $\Theta_{\pi_{\eta,\sigma}}(b)$ is
$$
\int_{J\cap G'} |N_G(J)(h)|^{-1}|^{-1} Bigl ( \int_{G/Z_J}b(ghg^{-1})
\Phi_{\eta,\sigma,J}(ghg^{-1})d(gZ_J)\Bigr ) |\Delta_{G,J}(h)|^2 dh.
$$
From Lemma \ref{4.3.14} we see that $\Theta_{\pi_{\eta,\sigma}}|_{G'_J}$
is given by $\Phi_{\eta,\sigma,J}$\,.  That proves the character formula
and completes the proof of Theorem \ref{4.3.8}.

\subsection{}\label{ssec4e}\setcounter{equation}{0}
We specialize the results of \S \ref{ssec4d} to the $H$--series of $G$,
where $[\eta] \in \widehat{M}_{disc}$\,.
The Cartan subgroup $H = T\times A$ and the associated cuspidal parabolic
subgroup $P=MAN$ are fixed.  The two principal simplifications here are
(1) $\Xi(H) = \{H\}$ and (2) the character formulae for $H$--series
representations are explicit \cite{HW1986a}.

The choice of $H$ and $P$ specifies $\Sigma_\ga^+$ with
$\gn = \sum_{\alpha \in \Sigma_\ga^+} \gg^{-\alpha}$.  Choose $\Sigma_\gt^+$
and specify $\Sigma^+$ as in Lemma \ref{4.1.7}.  We have $\rho$, $\rho_\ga$,
$\rho_\gt$, $\Delta_{G,H}$ and $\Delta_{M,T}$ as in (\ref{4.3.4}).  Make the
adjustment of Lemma \ref{4.3.6} if needed, so that $e^\rho \in \widehat{H}$ 
and $e^{\rho_\gt} \in \widehat{T}$ are well defined, and $e^\rho(Z) = 
e^{\rho_\gt}(Z) = 1$.  Then $\Delta_{G,H}$ is well defined on $H$ and
$\Delta_{M,T}$ is well defined on $M$.  
Proposition \ref{4.1.6} says
that $M$ has relative discrete series as described in \S\S \ref{ssec3d} 
and \ref{ssec3e}.  It comes out as follows.  Let $\varpi_\gt(\nu) =
\prod_{\phi \in \Sigma^+_\gt} \langle \phi,\nu\rangle$ for
$\nu \in \gt^*_\C$ and $L''_\gt = \{\nu \in i\gt^* \mid e^\nu \in
\widehat{T^0} \text{ and } \varpi_\gt(\nu) \ne 0\}$.  Every 
$\nu \in L''_\gt$ specifies a class $[\eta_\nu] \in (\widehat{M^0})_{disc}$
whose distribution character satisfies
$$
\Psi_{\eta_\nu}|_{T^0\cap M''} = (-1)^{q_M(\nu)}
\tfrac{1}{\Delta_{M,T}} {\sum}_{w \in W(M^0,T^0)} \det(w)e^{w\nu}
$$
with $q_M$ defined on $L_\gt''$ as in (\ref{3.4.6}).  Every class in
$(\widehat{M^0})_{disc}$ is one of these $[\eta_\nu]$.  Classes
$[\eta_\nu] = [\eta_{\nu'}]$ if and only if $\eta' \in W(M^0,T^0)(\nu)$.
Finally, $[\eta_\nu]$ has central character $e^{\nu - \rho_\gt}|_{Z_M^0}$
and infinitesimal character $\chi_\nu$ relative to $\gt$.

If $\nu \in L''_\gt$ and $[\chi] \in \widehat{Z_M(M^0)}_\xi$ where
$\xi = e^{\nu - \rho_\gt}|_{Z_M^0}$ then we have
\begin{equation}\label{4.4.3a}
\eta_{\chi,\nu}
= \Ind_{M^\dagger}^M (\chi \otimes \eta_\nu) \in \widehat{M}_{disc}.  
\end{equation}
Here recall $M^\dagger := Z_M(M^0)M^0$.  Also,
$\eta_{\chi,\nu} \in \widehat{M}_{disc}$\,, and it is the only
class there with distribution character given on 
$Z_M(M^0)\cdot(T^0\cap M'')$ by
\begin{equation}\label{4.4.3b}
\begin{aligned}
&\Psi_{\eta_{\chi,\nu}}(zt) = \\
&= \sum_{1\leqq j\leqq r} (-1)^{q_M(w_j\nu)}\trace \chi(x_j^{-1}zx_j)\cdot
\frac{1}{\Delta_{M,T}} \sum_{w \in W_{M^0,T^0}} \det(ww_j)e^{ww_j\nu}(t)
\end{aligned}
\end{equation}
where there the $w_j = \Ad(x_jT)_\gt$ are representatives of $W_{M,T}$
modulo $W_{M^0,T^0}$\,.  Every class in $\widehat{M}_{disc}$ is one of the
$[\eta_{\chi,\nu}]$.  Classes $[\eta_{\chi,\nu}] = [\eta_{\chi',\nu'}]$ 
exactly when $([\chi'],\nu') \in W_{M,T}([\chi],\nu)$.  Finally,
$[\eta_{\chi,\nu}]$ has infinitesimal character $\chi_\nu$ relative
to $\gt$.  Now we combine this description with Theorem \ref{4.3.8}.
Recall that the normalizers $N_{MA}(H) = N_M(T) \times A$ and $N_G(H)$ have
all orbits finite on $H \cap G'$.

\begin{theorem}\label{4.4.4}  Let $\nu \in L''_\gt$, $\sigma \in \ga^*$ and 
$[\chi] \in \widehat{Z_M(M^0)}_\xi$ where 
$\xi = e^{\nu - \rho_\gt}|_{Z_M^0}$\,.  Define $\eta_{\chi,\nu}$
and $\Psi_{\eta_{\chi,\nu}}$ by {\rm (\ref{4.4.3a})} and
{\rm (\ref{4.4.3b})}.  Then $[\pi_{\chi,\nu,\sigma}] :=
[\Ind_P^G(\eta_{\chi,\nu}\otimes e^{i\sigma})]$ is the unique $H$--series
representation class on $G$ whose distribution character satisfies
$$
\Theta_{\pi_{\chi,\nu,\sigma}}(ta)  
= \frac{|\Delta_{M,T}(t)|}{|\Delta_{G,H}(ta)|}
\sum_{w(ta)\in N_G(H)(ta)} |N_{M}(T)(wt)|^{-1} 
\Psi_{\eta_{\chi,\nu}}(wt) e^{i\sigma}(wa)
$$
for $t\in T$, $a \in A$ and $ta \in G'$.
Every $H$--series class on $G$ is one of the $[\pi_{\chi,\nu,\sigma}]$,
and classes $[\pi_{\chi,\nu,\sigma}] = [\pi_{\chi',\nu',\sigma'}]$
if and only if $(\chi',\nu',\sigma') \in W_{G,H}(\chi,\nu,\sigma)$.
$[\pi_{\chi',\nu',\sigma'}]$ is a finite sum from $\widehat{G}_\zeta$ where
$[\eta_{\chi,\nu}] \in \widehat{M}_\zeta$\,.  The dual 
$[\pi_{\chi,\nu,\sigma}^*] = [\pi_{\bar\chi,-\nu,-\sigma}]$.
The infinitesimal character is $\chi_{\nu + i\sigma}$ relative to $\gh$,
so $[\pi_{\chi,\nu,\sigma}]$ sends the Casimir element of
$\cU(\gg)$ to $||\nu||^2 + ||\sigma||^2 - ||\rho||^2$.
\end{theorem}

\begin{proof} First note that $\Xi(H) = \{H\}$ because any two fundamental
Cartan subgroups of $MA$ are $\Ad(M^0)$--conjugate.  That eliminates the
sum over $\Xi(H)$ expected from Theorem \ref{4.3.8}.  Now we need only 
show that $\Theta_{\pi_{\chi,\nu,\sigma}}|_{H\cap G'}$
determines $(\chi,\nu,\sigma)$ modulo $W_{G,H}$.  Let
$\Theta_{\pi_{\chi,\nu,\sigma}}|_{H\cap G'} =
\Theta_{\pi_{\chi',\nu',\sigma'}}|_{H\cap G'}$.
By linear independence of characters $e^{i\sigma''}$ on $A$ we may
replace $\sigma'$ by any element of $N_G(H)\sigma'$ and assume
$\sigma' = \sigma$.  Thus, on $H\cap G'$,
$\sum |N_M(T)(wt)|^{-1} e^{i\sigma}(wa)\bigl ( \Psi_{\eta_{\chi,\eta}}(wt)
- \Psi_{\eta_{\chi',\eta'}}(wt)\bigr ) = 0$.
Here $|N_M(T)(wt)|$ is locally constant on $T\cap M''$ and the functions 
$\Psi_{\eta_{\chi'',\eta''}}$ are linearly independent on $T\cap M''$.
Thus $\Psi_{\eta_{\chi,\eta}} = \Psi_{\eta_{\chi',\eta'}}$\,, so
$(\chi,\eta) = (\chi',\eta')$, and thus 
$(\chi',\eta') \in W_{M,T}(\chi,\eta)$.
\end{proof}

\begin{corollary}\label{4.4.5}
The $H$--series classes $[\pi_{\chi,\nu,\sigma}]$ are independent of
the choice of cuspidal parabolic subgroups $P$ associated to $H$
\end{corollary}
 
The support of $\Theta_{\pi_{\chi,\nu,\sigma}}$ meets the interior of
$G'_H$\,, and by Theorem \ref{4.3.8}(3) it determines the conjugacy
class of $H$.  A stronger result, due to Lipsman \cite[Theorem 11.1]{L1971}
for connected semisimple groups with finite center, is
\begin{theorem}\label{4.4.6}
Let $H$ and $'H$ be non--conjugate Cartan subgroups of $G$.  Let 
$[\pi] \in \widehat{G}$ be $H$--series and let $['\pi] \in \widehat{G}$ 
be $'H$--series.  Then the infinitesimal characters 
$\chi_\pi \ne \chi_{'\pi}$\,, and $[\pi]$ and $['\pi]$ are disjoint
(no composition factors in common).
\end{theorem}

\begin{proof}
Take both Cartans are $\theta$--stable, $H = T \times A$
and $'H = {'T} \times {'A}$.  Express $[\pi] = [\pi_{\chi,\nu,\sigma}]$
using $H$ and $[{'\pi}] = [\pi_{{'\chi},{'\nu},{'\sigma}}]$ using $'H$.
Then $\chi_\pi = \chi_{\nu + i\sigma}$ using $\gh$ and
$\chi_{'\pi} = \chi_{{'\nu} +i {'\sigma}}$ using $'\gh$.

If $\chi_\pi = \chi_{'\pi}$ there exists $\beta \in \Int(\gg_\C)$ such
that $\beta(\gh_\C) = {'\gh}_\C$ and $\beta^*({'\nu}+{'\sigma})
= (\nu + i\sigma)$.  $\beta^*$ sends real span of roots to real span
of roots, so $\beta^*({'\nu}) = \nu$ and $\beta^*({'\sigma}) = \sigma$.
Further, we may suppose $\beta^*('\Sigma^+) = \Sigma^+$.  It follows
that $\beta(\gh) = {'\gh}$.  Consequently \cite[Corollary 2.4]{R1972}
there is an inner automorphism of $\gg$ that sends $\gh$ to $'\gh$,
contradicting nonconjugacy of $H$ and $'H$.  Thus
$\chi_\pi \ne \chi_{'\pi}$\,.  Now $[\pi]$ and
$[{'\pi}]$ are disjoint because common factors would have the same 
infinitesimal character.
\end{proof}

\subsection{}\label{ssec4f}\setcounter{equation}{0}
We discuss irreducibility for $H$--series representations.  As before
fox $H=T\times A$ and $P = MAN$.  Let 
$[\eta] \in \widehat{M}$ have infinitesimal character $\chi_nu$ relative 
to $\gt_\C$\,.  We say that $[\eta]$ has {\em real infinitesimal
character} if $\langle\phi,\nu\rangle$ is real for every 
$\phi \in \Sigma_\gt^+$\,.  The classes in $\widehat{M}_{disc}$ have
real infinitesimal character.

An element $\sigma \in \ga^*$ is $(\gg,\ga)$--{\em regular} if
$\langle\psi,\sigma\rangle \ne 0$ for all $\psi \in \Sigma_\ga$\,.
Choose a minimal parabolic subgroup $P_0 = M_0A_0N_0$ of $G$ with 
$A \subset A_0 = \theta A_0$\,.  The $\ga$--roots are just the
nonzero restrictions of the $\gh_\C$--roots, and so they are the
nonzero restrictions of the $\ga_0$--roots.  If $w \in W(\gg,\ga_0)$ and if
$\sigma \in \ga^*$ is $(\gg,\ga)$--regular, then $\ga$ is central in
the centralizer $\gg^\sigma$ and $w \in W(\gg^\sigma,\ga_0)$, so $w$
is generated by reflections in roots that annihilate $\ga$, forcing
$w|_\ga$ to be trivial.  In summary,
\begin{lemma}\label{4.5.1} If $\sigma \in \ga^*$ then the following
conditions are equivalent:
{\rm (i)} $\sigma$ is $(\gg,\ga)$--regular,
{\rm (ii)} If $\phi \in \Sigma^+$ and $\phi|_\ga \ne 0$ then
	$\langle\phi,\sigma\rangle \ne 0$,
{\rm (iii)} If $\psi_0$ is an $\ga_0$--root of $\gg$ and $\psi_0|_\ga 
\ne 0$ then
        $\langle\psi_0,\sigma\rangle \ne 0$,
{\rm (iv)} If $w \in W(\gg,\ga_0)$ and $w|_\ga \ne 1$ then
	$w(\sigma)\ne\sigma$.
\end{lemma}

The following theorem was proved by Harish--Chandra (unpublished):
\begin{theorem} \label{4.5.2}
Let $[\eta]\in\widehat{M}$ have real infinitesimal
character and let $\sigma \in \ga^*$ be $(\gg,\ga)$--regular.  Then
$[\pi_{\eta,\sigma}] = [\Ind_P^G(\eta\otimes e^{i\sigma})]$ is
irreducible.
\end{theorem}
\begin{corollary} \label{4.5.3}
If $\sigma \in \ga^*$ is $(\gg,\ga)$--regular then every $H$--series
class $[\pi_{\chi,\nu.\sigma}]$ is irreducible.
\end{corollary}
After that, irreducibility were settled by Knapp and Zuckerman
(\cite{KZ1982a}, \cite{KZ1982b}) for connected reductive real
linear algebraic groups (the case where $G$ is connected and is isomorphic
to a closed subgroup of some general linear group $GL(n;\R)$).  In view of 
Langlands theorem \cite{L1973}, that completed the classification of 
irreducible 
admissible representations for reductive real linear algebraic groups.
For those groups, and more generally
for groups of class $\cH$, Vogan's treatment of the
Kazhdan--Lusztig conjecture and construction and analysis
of the KLV polynomials (\cite{V1983a}, \cite{V1983b}) includes a complete
analysis of the composition series of any $H$--series representation
$[\pi_{\chi,\nu.\sigma}]$.  Finally, the {\em Atlas} software,
{\tt http://www.liegroups.org/software/}, allows explicit computation
of those composition factors; see \cite{ALTV2015}.

\section{The Plancherel Formula for General Real Reductive Lie Groups}
\label{sec5}
\setcounter{equation}{0}
As before, $G$ is a real reductive Lie group that satisfies (\ref{1.2.2}).
The Harish--Chandra class $\cH$ consists of all such groups for which 
$G/G^0$ is finite and the derived group $[G^0,G^0]$ has finite center.
We fix a Cartan involution $\theta$ of $G$ and a system
$\Car(G) = \{H_1, \dots , H_\ell\}$ of $\theta$--stable representatives of
the conjugacy classes of Cartan subgroups of $G$.  
Harish--Chandra's announced \cite[\S 11]{HC1971} a Plancherel formula
for groups of class $\cH$: there are unique continuous
functions $m_{j,\eta}$ on $\ga_j^*$\,, meromorphic on $(\ga_j^*)_\C$\,,
invariant under the Weyl group $W(G,H_j)$, such that
$$
f(x) = \sum_{1\leqq j\leqq\ell} \,\,\sum_{[\eta] \in (\widehat{M_j})_{disc}}
	\deg(\eta)\int_{\ga_j^*} \Theta_{\pi_{\eta,\sigma}}
	r_x(f) m_{j,\eta}(\sigma)d\sigma,
$$
absolutely convergent for $x \in G'$ and $f \in C_c^\infty(G).$ 
This was extended to our class in \cite{W1973} without consideration
of meromorphicity.  Later Harish--Chandra published a complete
treatment for $G$ of class $\cH$ and $f$ in the Harish--Chandra
Schwartz space $\cS(G)$ (\cite{HC1975},\cite{HC1976a},\cite{HC1976b}).  
Still later
Herb and I extended those results to general real reductive groups,
including explicit formulae for the various constants and functions 
that enter into the Plancherel measure (\cite{HW1986a}, \cite{HW1986b}).

Here, for lack of space or necessity, I'll only indicate the results from 
\cite{W1973}, because that is all that is needed in \S\S 7 and 8 below.

\subsection{}\label{ssec5a}\setcounter{equation}{0}
As above, we have $G$, $\theta$, $K = G^\theta$, 
$\Car(G) = \{H_1,\dots, H_\ell\}$, $H_j = T_j\times A_j$,
$\Sigma^+_{\ga_j}$ and $P_j = M_jA_jN_j$ with $M_j\times A_j = Z_G(A_j)$.
As in \S 4, $L_j = \{\nu \in i\gt^*_j \mid e^\nu \in \widehat{T_j^0}\}$
and $L_j''$ is its $M_j$--regular set.  Fix the $\Sigma^+_{\gt_j}$
and set $\varpi_{\gt_j}(\nu) = \prod_{\phi\in\Sigma^+_{\gt_j}}
\langle \phi,\nu\rangle$, so $L_j'' = \{\nu\in L_j \mid \varpi_{\gt_j}(\nu)
\ne 0\}$.  

If $\zeta\in\widehat{Z}$ then $L_{j,\zeta} =
\{\nu \in L_j \mid e^{\nu - \rho_{\gt_j}}|_{Z\cap M_j^0} =
\zeta|_{Z\cap M_j^0}\}$ and $L_j'' = L_{j,\zeta}\cap L_j''$\,.
Write $\xi_\nu$ for $e^{\nu - \rho_{\gt_j}}$.  Since $ZZ_{M_j^0}$
has finite index in $Z_{M_j}(M_j^0)$ we define finite subsets 
$S(\nu,\zeta) \subset \widehat{Z_{M_j}(M_j^0)}$ by
$S(\nu,\zeta) = \widehat{Z_{M_j}(M_j^0)}_{\zeta\otimes\xi_\nu}$ if 
$\xi|_{Z\cap M_j^0} = \xi_\nu|_{Z\cap M_j^0}$\,,  
$S(\nu,\zeta) = \emptyset$ otherwise.  When $\nu \in L_j''$ and 
$\sigma \in \ga_j^*$ the $H_j$--series classes $[\pi_{\chi,\nu,\sigma}]$
that transform by $\zeta$ are just the ones with $[\chi] \in S(\nu,\zeta)$. 
Thus we have finite sums
$$
\pi_{j,\zeta,\nu+i\sigma} = {\sum}_{S(\nu,\zeta)} (\dim \chi)
	\pi_{\chi,\nu,\sigma} \text{ and }
\Theta_{j,\zeta,\nu+i\sigma} = {\sum}_{S(\nu,\zeta)} (\dim \chi)
	\Theta_{\pi_{\chi,\nu,\sigma}}\,\,.
$$
If $\zeta|_{Z\cap M_j^0} \ne \xi_\nu|_{Z\cap M_j^0}$\,, in other words if
$\nu \not\in L_{j,\zeta}$\,, then $\Theta_{j,\zeta,\nu+i\sigma} = 0$.
Here is the 
extension \cite{W1973} of the Harish--Chandra Plancherel
Formula (\cite{HC1970}, \cite{HC1971}) to the $\widehat{G}_\zeta$\,.

\begin{theorem}\label{5.1.6}
Let $G$ be a general real reductive Lie group {\rm (\ref{1.2.2})} and 
$\zeta \in \widehat{Z}$.
Then there are unique Borel--measurable functions $m_{j,\zeta,\nu}$ on
$\ga_j^*$\,, $1\leqq j\leqq\ell$, defining the
{\rm Plancherel measure on} $\widehat{G}_\zeta$\,, as follows.

{\rm 1.} The $m_{j,\zeta,\nu}$ are $W_{G,H_j}$--equivariant:
$w^*m_{j,\zeta,\nu}(\sigma) = 
m_{j,w^*\zeta,w^*\nu}(w^*\sigma)$.

{\rm 2.} If $\nu \notin L_{j,\zeta}$ then $m_{j,\zeta,\nu} = 0$. 

{\rm 3.} Let $f \in L_2(G/Z,\zeta)$ be $C^\infty$ with support compact 
modulo $Z$.  If $x \in G$ define $(r_xf)(g) = f(gx)$.  Then
\begin{equation}\label{5.1.7}
\begin{aligned}
&\sum_{1\leqq j\leqq\ell} \sum_{\nu \in L''_{j,\zeta}} |\varpi_{\gt_j}(\nu)|
	\int_{\ga_j^*} |\Theta_{j,\zeta,\nu+i\sigma}(r_xf)
	m_{j,\zeta,\nu}(\sigma)|d\sigma < \infty \text{ and } \\
&f(x) = \sum_{1\leqq j\leqq\ell}\,\, \sum_{\nu \in L''_{j,\zeta}} 
	|\varpi_{\gt_j}(\nu)| \int_{\ga_j^*} 
	\Theta_{j,\zeta,\nu+i\sigma}(r_xf)
	m_{j,\zeta,\nu}(\sigma) d\sigma.
\end{aligned}
\end{equation}
\end{theorem}
The following corollary is used for realization of $H$--series
representations on spaces of partially harmonic spinors \cite{W1974a}.
\begin{corollary}\label{5.7.1}
Let $\omega \in \cU(\gg)$ be the Casimir element.  If $c \in \R$
and $\zeta \in \widehat{Z}$ then $\{[\pi] \in \widehat{G}_\zeta
\setminus \widehat{G}_{\zeta-disc} \mid \chi_\pi(\omega) = c\}$ has 
Plancherel measure $0$ on $\widehat{G}_\zeta$\,.
\end{corollary}

Corollary \ref{5.7.3} below is needed when we consider spaces of square 
integrable partially harmonic $(0,q)$--forms in \S\S 7 and 8.
It follows from Theorems \ref{4.5.2} and \ref{5.1.6}; or one can
also derive it from 

\begin{lemma}\label{5.7.2}
Let $[\pi]$ be an irreducible constituent of an $H$--series class
$[\pi_{\chi,\nu,\sigma}]$ where $\nu + i\sigma \in \gh^*$ is
$\gg$--regular.  If $G$ has relative discrete series representations,
and if $H/Z$ is noncompact, then $\Theta_\pi|_{K\cap G'} = 0$.
\end{lemma}

\begin{corollary}\label{5.7.3}
If $G$ has relative discrete series representations and if $\zeta \in
\widehat{Z}$ then $\{[\pi] \in \widehat{G}_\zeta \setminus 
\widehat{G}_{\zeta-disc} \mid \Theta_\pi|_{K\cap G'} \ne 0\}$
has Plancherel measure $0$ on $\widehat{G}_\zeta$\,.
\end{corollary}

\begin{corollary}\label{5.7.4}
Fix $\zeta \in \widehat{Z}$.  Let $\widehat{G}_{H_j,_\zeta}$ denote
the set of all $H_j$--series classes $[\pi_{\chi,\nu,\sigma}]$ for
$\zeta$ such that $\sigma$ is $(\gg,\ga_j)$--regular.  Then each
$\widehat{G}_{H_j,_\zeta} \subset \widehat{G}_\zeta$ and the
Plancherel measure on $\widehat{G}_\zeta$ is concentrated on
$\bigcup_{1\leqq j\leqq\ell} \widehat{G}_{H_j,_\zeta}$\,.
\end{corollary}

The ``absolute'' version of Theorem \ref{5.1.6} derives from 
$f(x) := \int_{\widehat{Z}} f_\zeta(x)\zeta(z)d\zeta$ where
$f_\zeta(x) = \int_Z f(xz)\zeta(z)dz$.  Given $f \in C^\infty_c(G)$ we
apply Theorem \ref{5.1.6} to each $f_\zeta$ and sum over $\widehat{Z}$.
The same holds for the corollaries.

\section{Real Groups and Complex Flags}
\label{sec6}
\setcounter{equation}{0}

While $G$ is a general real reductive Lie group (\ref{1.2.2}) the
adjoint representation takes $G$ to a real reductive semisimple Lie
group $\overline{G} := G/Z_G(G^0)$.  That group has complexification
$\overline{G}_\C = \Int(\gg_\C)$, the group of inner automorphisms
of $\gg_\C$\,.  Notice that $\overline{G}_\C$ is connected. 
Now $G$ acts on all complex flag manifolds
$X = \overline{G}_\C/Q$.  Here we recall the part of \cite{W1969}
needed for geometric realization of standard induced representations,
extending them from $\overline{G}^0$ to $G$ as needed.  We discuss
holomorphic arc components of $G$--orbits; consider measurable,
integrable and flag type orbits; and give a complete analysis of the
orbits on which our representations are realized in \S\S 7 and 8.

Notation: $Q$ is used for a (complex) parabolic subgroup of
$\overline{G}_\C$ and $P$ is reserved for cuspidal parabolics in $G$.
Roots are ordered so that $X = \overline{G}_\C/Q$ has holomorphic
tangent space spanned by positive root spaces.

\subsection{}\label{ssec6a}\setcounter{equation}{0}
It is standard that the following are equivalent for a closed complex
subgroup $Q \subset \overline{G}_\C$\,: (i) $X = \overline{G}_\C$ is
compact, (ii) $X$ is a compact simply connected K\"ahler manifold,
(iii) $X$ is a $\overline{G}_\C$--homogeneous projective algebraic
variety, (iv) $X$ is a closed $\overline{G}_\C$--orbit in a (finite
dimensional) projective representation, and (v) $Q$ contains a Borel
subgroup of $\overline{G}_\C$\,.  Under these conditions we say that
(1) $Q$ is a {\em parabolic subgroup} of $\overline{G}_\C$\,, (2)
$\gq$ is a {\em parabolic subalgebra} of $\overline{\gg}_\C$, and
(3) $X = \overline{G}_\C/Q$ is a {\em complex flag manifold} of 
$\overline{\gg}_\C$\,.  Given (i) through (v) $Q$ is the analytic
subgroup of $\overline{G}_\C$ for $\overline{\gq}_\C$\,, in fact is
the $\overline{G}_\C$--normalizer of $\overline{\gq}_\C$\,.

Recall the structure.  Choose a Cartan subalgebra $\overline{\gh}_\C$
of $\overline{\gg}_\C$ and a system $\Pi$ of simple 
$\overline{\gh}_\C$--roots on $\overline{\gg}_\C$\,.  Any subset
$\Phi \subset \Pi$ specifies
\begin{itemize}
\item $\Phi^r$: all roots that are linear combinations of elements 
of $\Phi$;
\item $\Phi^u$: all negative roots not contained in $\Phi^r$;
\item $\gq_\Phi^r = 
	\overline{\gh}_\C + \sum_{\Phi^r} \overline{\gg}^\phi$\,,\,\,
      $\gq_\Phi^u = \sum_{\Phi^u} \overline{\gg}^\phi$ and
      $\gq_\Phi = \gq_\Phi^r + \gq_\Phi^u$\,.
\end{itemize}
Then $\overline{G}_\C$ has analytic subgroups $Q_\Phi^r$ for $\gq_\Phi^r$\,,
$Q_\Phi^u$ for $\gq_\Phi^u$ and $Q_\Phi = Q_\Phi^u\rtimes Q_\Phi^r$ for 
$\gq_\Phi$\,.  $Q_\Phi^u$ and $\gq_\Phi^u$ are the nilradicals, and
$Q_\Phi^r$ and $\gq_\Phi^r$ are the Levi (reductive) components.  
$\Phi$ is a simple $\overline{\gh}_\C$--root system for $\gq_\Phi^r$\,. 
$Q_\Phi$
is a parabolic subgroup of $\overline{G}_\C$ and every parabolic subgroup
of $\overline{G}_\C$ is conjugate to exactly one of the $Q_\Phi$\,.
Any parabolic
$Q_\Phi$ is its own normalizer in $\overline{G}_\C$, so the complex
flag manifold $X = \overline{G}_\C/Q$ is in one-one correspondence
$x \leftrightarrow Q_x$ with the set of $\overline{G}_\C$--conjugates of
$Q$, by $Q_x = \{\overline{g} \in \overline{G}_\C \mid \overline{g}(x)=x\}$.
We will make constant use of this identification.

\subsection{}\label{ssec6b}\setcounter{equation}{0}
Let $\overline{G}$ be an open subgroup of a real form $\overline{G}_\R$
of $\overline{G}_\C$\,, so $\overline{G}^0$ is the real analytic subgroup
of $\overline{G}_\C$ for $\overline{\gg} = \overline{\gg}_\R$\,.  Denote
\begin{equation}\label{6.2.1}
\tau: \text{ complex conjugation of } \overline{G}_\C  \text{ over } 
	\overline{G}_\R \text{ and of } \overline{\gg}_\C  \text{ over }
	\overline{\gg}_\R\,.
\end{equation}
The isotropy subgroup of $\overline{G}$ at $x \in X$ is 
$\overline{G}\cap Q_x$\,.  The latter has Lie algebra 
$\overline{\gg}\cap \gq_x$\, which is a real form of $\gq_x \cap \tau\gq_x$.
The intersection of any two Borel subgroups contains a Cartan, and using
care it follows that we have
\begin{equation}\label{6.2.3}
\begin{aligned}
&\text{a Cartan subalgebra } \overline{\gh} \subset \overline{\gg}\cap \gq_x
	\text{ of } \overline{\gg}, \text{ a system } \Pi
	\text{ of simple } \overline{\gh}_\C\text{--roots }\\
&\text{of } \overline{\gg}_\C\,, \text{ and a subset } \Phi \subset \Pi
	\text{ such that } \gq_x = \gq_\Phi\,.
\end{aligned}
\end{equation}
Then we have the key decomposition to understanding $G$--orbits on $X$:
\begin{equation}\label{6.2.4}
\begin{aligned}
\gq_x &\cap \tau\gq_x = (\gq_x \cap \tau\gq_x)^r +
	(\gq_x \cap \tau\gq_x)^u \text{ where }
	(\gq_x \cap \tau\gq_x)^r = \overline{\gh}_\C + 
	\sum_{\Phi^r \cap \tau\Phi^r} \overline{\gg}^\phi \\
&\text{ and } (\gq_x \cap \tau\gq_x)^u = 
	\sum_{\Phi^r \cap \tau\Phi^u} \overline{\gg}^\phi +
	\sum_{\Phi^u \cap \tau\Phi^r} \overline{\gg}^\phi +
	\sum_{\Phi^u \cap \tau\Phi^u} \overline{\gg}^\phi \,.
\end{aligned}
\end{equation}
This shows that $\overline{G}(x)$ has real codimension
$|\Phi^u \cap \tau\Phi^u|$ in $X$, in particular that $\overline{G}(x)$
is open in $X$ if and only if $\Phi^u \cap \tau\Phi^u$ is empty, and
also that there are only finitely many $\overline{G}$--orbits on $X$.
This last shows that $\overline{G}$ has both open and closed 
orbits.

Recall $\overline{G} = G/Z_G(G^0) = \Int(\gg)_\C$\,, so $G$ acts on 
$X = \overline{G}/Q$ through $G \to \overline{G}$, specifically by
$Q_{g(x)} = \Ad(g)Q_x$\,.  Thus $G(x) = \overline{G}(x)$.  Now the results
on orbits and isotropy of $\overline{G}$ and $\overline{\gg}$ on $X$,
apply as well to orbits and isotropy of $G$ and $\gg$.

\subsection{}\label{ssec6c}\setcounter{equation}{0}
Let $H$ be a Cartan subgroup of $G$ and $\theta$ a Cartan involution with
$\theta(H) = H$.  Let $K$ be the fixed point set $G^\theta$.  As in
(\ref{4.1.4}) $\gh = \gt + \ga$ and $H = T\times A$ under the action of
$\theta$.  Thus \cite[Theorem 4.5]{W1969} the following conditions are 
equivalent: (i) $T$ is a
Cartan subgroup of $K$, (ii) $\gt$ contains a regular element of $\gg$,
and (iii) some simple system $\Pi$ of $\gh_\C$--roots satisfies
$\tau\Pi = -\Pi$.  Then those conditions hold, one says that $\gh$ is a
{\em fundamental Cartan subalgebra} of $\gg$ and $H$ is a
{\em fundamental Cartan subgroup} of $G$.  Equivalently, $\gh$ is a
{\em maximally compact Cartan subalgebra} of $\gg$ and $H$ is a
{\em maximally compact Cartan subgroup} of $G$.  From (\ref{6.2.4}),
\begin{lemma}\label{6.4.2} $G(x)$ is open in $X$ if and only if there
exist a maximally compact Cartan subalgebra $\gh \subset \gg$,
a simple $\gh_\C$--root system $\Pi$ with $\tau\Pi = -\pi$, and a
subset $\Phi \subset \Pi$, such that $\gq_x = \gq_\Phi$\,.
\end{lemma}

An open orbit $G(x)$ carries an invariant Radon measure if and only
if the isotropy subgroup at $x$ is reductive, i.e., if and only
if the choices in Lemma \ref{6.4.2} can be made so that
$\tau\Phi^r = \Phi^r$ and $\tau\Phi^u = -\Phi^u$.  Thus 
\cite[Theorem 6.3]{W1969} these
conditions are equivalent: (i) $G(x)$ is open in $X$ and has a
$G$--invariant positive Radon measure, (ii) $G(x)$ has a $G$--invariant
possibly--indefinite K\"ahler metric, (iii) $\gq_x\cap\tau\gq_x$ is
reductive, i.e., $\gq_x\cap\tau\gq_x = \gq_x^r \cap \tau\gq_x^r$ and
(iv) $\tau\Phi^r = \Phi^r$ and $\tau\Phi^u = -\Phi^u$.  Under those
conditions we say that the open orbit $G(x)$ is {\em measurable}.

A closely related set of equivalent conditions \cite[Theorem 6.7]{W1969} is
(a) some open $G$--orbit on $X$ is measurable, (b) every open 
$G$--orbit on $X$ is measurable and (c) if $\gq = \gq_\Phi$ then
$\tau\gq$ is conjugate to the opposite parabolic $\gq^r + \gq^{-u}$
where $\gq^{-u} = \sum_{\Phi^u}\overline{\gg}^{-\phi}$.
These conditions are automatic  \cite[Corollary 6.8]{W1969} if
$\rank K = \rank G$, i.e., if $G$ has relative discrete series 
representations.  In that regard we will need

\begin{lemma}\label{6.4.5}
Let $U$ be the isotropy subgroup of $G$ at $x \in X$.  Suppose that
$\gq$ does not contain any nonzero ideal of $\overline{\gg}_\C$\,.
Then the following are equivalent.

1. $U$ acts on the tangent space to $G(x)$ as a compact group.

2. $G(x)$ has a $G$--invariant positive definite hermitian metric.

3. $\overline{\gg}\cap \gq_x$ is contained in the fixed point set of a
Cartan involution of $\overline{\gg}$

\noindent
Under these conditions, $G(x)$ is open in $X$ and the maximal compact
subgroups $\overline{K} \subset \overline{G}$ satisfy
$\rank \overline{K} = \rank \overline{G}$.
\end{lemma}

Suppose $\rank K = \rank G$.  Let $G(x)$ be an open orbit, so 
$\overline{\gg}$ has a Cartan subalgebra $\overline{\gh} \subset
\overline{\gk}\cap\gq_x$ where $\overline{H} = H/Z_G(G^0)$ and
$\overline{K} = K/Z_G(G^0)$.  Let $W_K$\,, $W_{G_\C}$ and $W_{Q_x^r}$
denote Weyl groups relative to $\overline{\gh}$.  Then
\cite[Theorem 4.9]{W1969} the open $G$--orbits on $X$ are enumerated
by the double coset space $W_K\backslash W_{G_\C}/W_{Q_x^r}$\,.

\subsection{}\label{ssec6d}\setcounter{equation}{0}
We look at the maximal complex analytic pieces of a $G$--orbit on $X$.

Let $V$ be a complex analytic space and $D \subset V$.  By {\em holomorphic
arc} in $D$ we mean a holomorphic map $f:\{z \in C \mid |z|<1\} \to V$
with image in $D$.  A {\em chain of holomorphic arcs} in $D$ is a finite
sequence $\{f_1,\dots,f_m\}$ of holomorphic arcs in $D$ such that the
image of $f_{k-1}$ meets the image of $f_k$\,.  A {\em holomorphic arc
component} of $D$ is an equivalence class of elements of $D$ under
$u \sim v$ if there is a chain $\{f_1,\dots,f_m\}$ of holomorphic
arcs in $D$ with $u$ in the image of $f_1$ and $v$ in the image of $f_m$\,.
If $g$ is a holomorphic diffeomorphism of $V$, $g(D) = D$, and
$S$ is a holomorphic arc component of $D$ then $g(S)$ is 
holomorphic arc component of $D$.

Let $L$ be a group of holomorphic diffeomorphisms of $V$ that preserve $D$.
If $S$ is a holomorphic arc component of $D$ denote its $L$--normalizer
$\{\alpha \in L | \mid \alpha(S) = S\}$ by $N_L(S)$.  If $\alpha \in L$
and $\alpha(S)$ meets $S$ then $\alpha(S) = S$.  So if $D$ is an $L$--orbit
then $S$ is an $N_L(S)$--orbit.  It can happen that $D$ is a
real submanifold of $V$ but is not a complex submanifold; see 
\cite[Example 8.12]{W1969}.  

Now we turn to holomorphic arc components of an orbit 
$G(x) \subset X = \overline{G}_\C/Q$.  It is a finite union of $G^0$--orbits,
which are its topological components.  So we have
$$
S_{[x]}: \text{ holomorphic arc component of } G(x) 
	\text{ through } x.
$$
in the topological component of $x$ in $G(x)$, 
and its $G$-- and $\overline{G}$--normalizers 
$$
N_{[x]} = \{g \in G \mid gS_{[x]} = S_{[x]}\} \text{ and }
	\overline{N}_{[x]} = N_{[x]}/Z_G(G^0).
$$
The main general fact concerning these groups and their Lie algebras
\cite[Theorems 8.5 and 8.15]{W1969} is that $\gn_{[x],\C}$ is a $\tau$--stable
parabolic subalgebra of $\overline{\gg}_\C$\,, so $N_{[x],\C}$ is a parabolic 
subgroup of $\overline{G}$, and
$\gn_{[x]} = \overline{\gg}\cap \gn_{[x],\C}$ is a parabolic subalgebra 
of $\overline{\gg}$, so $N_{[x]}$ is a subgroup of finite index in a parabolic
subgroup of $\overline{G}$.  This ensures $G=KN_{[x]}$ and $\overline{G} =
\overline{K}\overline{N}_{[x]}$. In other words, $K$ and $\overline{K}$ are
transitive on the space $G/N_{[x]} = \overline{G}/\overline{N}_{[x]}$ of all
holomorphic arc components of $G(x)$.

With $x \in X$ fixed and $\gq_x = \gq_\Phi$ and in (\ref{6.2.4}) we consider
the real linear form 
$\delta_x = \sum_{\Phi^u\cap \tau\Phi^u}\phi : \overline{\gh} \to \R$.   That
defines
$\gq_{[x]} = \overline{\gh}_\C + \sum_{\langle\phi ,\delta_x\rangle \geqq 0}
	\gg^\phi$\,,
a $\tau$--stable parabolic subalgebra of $\gg_\C$\,. Then 
$(\gq_x^u \cap \tau\gq_x^u) \subset \gq_{[x]} \subset \{\overline{\gn}_{[x],\C}
\cap (\gq_x + \tau\gq_x)\}$.  
Let $\Gamma = \{\phi \in \Delta_\gg \mid
\langle\phi,\delta_x\rangle < 0, -\phi \notin \Phi^u\cap \tau\Phi^u,
\phi+\tau\phi \text{ not a root.}\}$  and define 
$\gm_{[x]} = \gq_{[x]} + \sum_{\Gamma} \overline{\gg}^\phi$.  Then
\cite[Theorem 8.9]{W1969} $\gm_{[x]} \subset \{\overline{\gn}_{[x],\C}
\cap (\gq_x + \tau\gq_x)\}$ and the following are equivalent:
(i) The holomorphic arc components of $G(x)$ are complex
submanifolds of $X$, (ii) $\overline{\gn}_{[x],\C} \subset (\gq_x + \tau\gq_x)$,
(iii) $\overline{\gn}_{[x],\C} = \gm_{[x]}$\,, and (iv) $\gm_{[x]}$ is a
subalgebra of $\gg_\C$\,.  When these hold, we say that the orbit
$G(x)$ is {\em partially complex.}

We will need stronger conditions.  An orbit $G(x)$ is {\em of flag type}
if the Zariski closure $\overline{N}_{[x],\C}$ of $S_{[x]}$ is a 
complex flag manifold, {\em measurable} if the holomorphic arc components
carry positive Radon measures invariant under their normalizers,
{\em integrable} if $(\gq_x + \tau\gq_x)$ is a subalgebra of 
$\overline{\gg}_\C$\,.  Given $\gq_x = \gq_\Phi$ we denote
$$
\gv_x^- = \sum_{\Phi^u \cap -\tau\Phi^u}\overline{\gg}^\phi,\,\,
\gv_x^+ = \sum_{-\Phi^u \cap -\tau\Phi^u}\overline{\gg}^\phi,\,\,
\gv_x = \gv_x^- + \gv_x^+\,.
$$
Then \cite[Theorem 9.2]{W1969} $G(x)$ is measurable if and only
if $\overline{\gn}_{[x],\C} = (\gq_x \cap \tau\gq_x) + \gv_x$\,.  It follows
that, in $G(x)$ is measurable, then (i) the invariant measure on $S_{[x]}$
comes from an $N_{[x]}$--invariant possibly--indefinite K\"ahler metric, (ii)
$G(x)$ is partially complex and of flag type, and (ii) $G(x)$ is integrable
if and only if $\tau\gq_x^r = \gq_x^r$\,.  

On the other hand, if 
$\tau\gq_x^r = \gq_x^r$ then \cite[Theorem 9.19]{W1969} $G(x)$ is measurable
$\Leftrightarrow G(x)$ is integrable $\Leftrightarrow G(x)$ is partially
complex and of flag type, and under those conditions
$\overline{\gn}_{[x],\C} =(\gq_x \cap \tau\gq_x) + \gv_x = \gq_{[x]}
= \gq_x + \tau\gq_x$\,.  Open orbits are obviously integrable, partially
complex and of flag type.  Closed orbits are another matter.  There is just one
closed $G$--orbit on $X$, every maximal compact subgroup of $\overline{G}$
is transitive on it, and it is connected.  There is a problem with 
\cite[Theorem 9.12]{W1969}, where it was asserted that the closed orbit 
always is measurable, hence partially complex (consider 
\cite[Example 8.12]{W1969} applied to $SU(m,m)$).  But if $Q$ is a
Borel subgroup of $\overline{G}$, then the closed orbit is measurable, hence
partially complex and of flag type, and in that case is integrable.

\subsection{}\label{ssec6e}\setcounter{equation}{0}
We now describe a class of orbits that plays a key role in the geometric 
realization of the various nondegenerate series of representations of $G$.
Fix a Cartan subgroup $H = T\times A$ of $G$ and an associated cuspidal
parabolic $P=MAN$.  We need complex flag manifolds
$X = \overline{G}_\C/Q$ and measurable integrable orbits $Y = G(x) \subset X$
such that the $G$--normalizers of the holomorphic arc components of $Y$ satisfy
\begin{equation}\label{6.7.1c}
N_{[x]} = \{g \in G \mid gS_{[x]} = S_{[x]} \text{ has Lie algebra } \gp\}.
\end{equation}
As $Y$ is to be measurable $S_{[x]}$ will be an open $M$--orbit 
in the sub--flag $\overline{M}_\C(x)$.  So $AN$ will act trivially on $S_{[x]}$
and the isotropy subgroup of $G$ at $x$ will have form $UAN$ with
$T\subset U\subset M$.  Finally, we need the condition that
\begin{equation}\label{6.7.1d}
U/Z_G(G^0) = \{m \in M \mid m(x) = x\}/Z_G(G^0) \text{ is compact.}
\end{equation}  
We look at some consequences of (\ref{6.7.1c}) and (\ref{6.7.1d}).
Write $\gz_\gg$ for the center of $\gg$, so $\gg \cong \gz_\gg  \oplus
\overline{\gg}$.  Since $Y$ is measurable and integrable,
\S \ref{ssec6d} would lead to
\begin{equation}\label{6.7.2}
\begin{aligned}
&\gp_\C = \gz_{\gg,\C} + \gq_x + \tau\gq_x\,,\,\, \tau\gq_x^r = \gq_x^r\,,\,\,
	\gn_\C = \gp_\C^u = \gq_x^u\cap \tau\gq_x^u\,,\\
&(\gm+\ga)_\C = \gp_\C^r = (\gz_{\gg,\C} +\gq_x^r) 
		+ (\gq_x^{-u} \cap \tau\gq_x^u)
		+ (\gq_x^u \cap \tau\gq_x^{-u})\,, \\
& (\gu + \ga)_\C = \gz_{\gg,\C} + \gq_x^r \text{ and }
	\gm_\C = \gu_\C + (\gq_x^{-u} \cap \tau\gq_x^u)
                + (\gq_x^u \cap \tau\gq_x^{-u}).
\end{aligned}
\end{equation}
Since $S_{[x]}$ a measurable open $M^0$--orbit in the flag $\overline{M}_\C(x)$,
(\ref{6.7.2}) ensures that
\begin{equation}\label{6.7.3}
\gr:= \gm_\C + (\gz_{\gg,\C} + \gq_x) \text{ is parabolic in } \gm_\C 
	\text{ with } \gr^r = \gu_\C \text{ and } 
	\gr^u = \gq_x^u\cap \tau\gq_x^{-u}\,.
\end{equation}
The following Proposition shows that (\ref{6.7.2}) and (\ref{6.7.3}) give us
the parabolics that we need for our geometric realizations.

\begin{proposition}\label{6.7.4}
Let $G$ be a real reductive Lie group in the class $\widetilde{\cH}$
of {\rm (\ref{1.2.2})}, $H = T\times A$, and $P=MAN$ an associated cuspidal
parabolic subgroup of $G$.  Suppose that {\rm (i)} $\gu \subset \gm$ is the
$\gm$--centralizer of a subalgebra of $\gt$ such that $U^0/(U^0\cap Z_G(G^0))$
is compact, {\rm (ii)} $\gr \subset \gm_\C$ is a parabolic subalgebra with 
$\gr^r \subset \gu_\C$\,, {\rm (iii)} $\gq$ is the 
$\overline{\gg}_\C$--normalizer
of $\gr^u + \gn_\C$ and $Q$ is the corresponding analytic subgroup of
$\overline{G}_\C$\,, and {\rm (iv)} $X = \overline{G}_\C/Q$ and $x = 1Q\in X$.
Then $Q$ is a parabolic subgroup of $\overline{G}_\C$, 
$\gq^u = \gr^u + \gn_\C$\,, and $G(x)$ is a measurable integrable orbit, and 
$(X,x)$ satisfies {\rm (\ref{6.7.1c})} and {\rm (\ref{6.7.1d})}.  
Conversely every pair $(X,x)$ satisfying {\rm (\ref{6.7.1c})} and 
{\rm (\ref{6.7.1d})}, $G(x)$ measurable and integrable, $U^0/(U^0\cap Z_G(G^0))$
compact, is constructed as above.
\end{proposition}

\begin{proof}
Denote $\overline{M} = M/Z_G(G^0)$.  Let $\overline{R}$ be the analytic 
subgroup of $\overline{M}_\C$ for $\overline{\gr} := \gr/\gz_{\gg,\C}$
and $S = \overline{M}_\C/\overline{R}$.  Then $S$ is a complex flag
manifold of $\overline{M}_\C$ by (ii).  The isotropy subalgebra of $\gm$ at
$s = 1\overline{R} \in S$ is $\gm\cap \gr$.  It has reductive part $\gu$
by (ii).  As $\gt \subset \gu$ by (i), and as $M(s)$ is measurable and 
open in $S$ by Lemmas \ref{6.4.2} and \ref{6.4.5}, we have 
$\gm\cap\gr = \gu$ and $\gr\cap\tau\gr = \gu_\C$\,.

Define $\gq$ and $Q$ as in (iii).  The contribution to $\gq$ from 
$\overline{\gm}_\C$ is $\overline{\gr}$, all of $(\ga + \gn)_\C$ from
$(\ga + \gn)_\C$\,, and $0$ from $\gn_\C^-$\,.  So $\gq = 
(\overline{\gu}_\C + \ga_\C) + (\gr^u + \gn_\C)$, thus is parabolic in
$\overline{\gg}_\C$\,.  Now $X = \overline{G}_\C/Q$ is a complex flag
manifold, $\gz_{\gg,\C} + \gq^r = \gu_\C + \ga_\C$\,, and 
$\gq^u = \gr^u + \gn_\C$\,.  In particular, $\gq + \tau\gq = \gp_\C/\gz_\gg$
is a subalgebra of $\gg_\C$\,, so the orbit $G(x)\subset X$ is integrable,
and $\tau\gq^r = \gr^r$ so $G(x)$ is measurable with 
$\overline{\gn}_{[x],\C} = \gq + \tau\gq = \gp_\C/\gz_{\gg,\C}$\,.  We conclude
$\overline{\gn}_{[x]} = \gp/\gz_\gg$ and so $\gn_{[x]} = \gp$.  We have
shown that $G(x)$ is a measurable integrable orbit, and 
$(X,x)$ satisfies {\rm (\ref{6.7.1c})} and {\rm (\ref{6.7.1d})}.

For the converse compare {\rm (\ref{6.7.1c})} and {\rm (\ref{6.7.1d})}
with the construction.
\end{proof}

We enumerate the $(X,x)$ of Proposition \ref{6.7.4}.  Let $\Pi_\gt$ be a
simple $\gt_\C$--roots system on $\gm_\C$ and $\Phi_\gt$ a subset of 
$\Pi_\gt$\,.  Let $\Pi$ be the simple $\gh_\C$--root system on $\gg_\C$
that contains $\Pi_\gt$ and induces the positive $\ga$--root
system used for construction of $P=MAN$.  Define 
$\Phi = \Phi_\gt \cup (\Pi \setminus \Pi_\gt)$.  The parabolic subalgebras
$\gq \subset \overline{\gg}_\C$ of Proposition \ref{6.7.4} are just the
$\gq_\Phi$\,.

\begin{corollary}\label{6.7.7} Given $G(x) \in X$ as in {\rm Proposition
\ref{6.7.4}}, $M^\dagger$ is the stabilizer 
$\{m \in M \mid mS_{[x]} = S_{[x]}\}$ of $S_{[x]}$ in $M$.   Thus
$U \subset M^\dagger$ and $N_{[x]} = M^\dagger AN$.
\end{corollary}

\begin{proof}
Let $M^1 = \{m \in M \mid mS_{[x]} = S_{[x]}\}$.  Then $M^\dagger \subset M^1$
because $M^\dagger = Z_M(M^0)M^0$ and $Z_M(M^0)$ acts trivially on $S_{[x]}$.
The isotropy subgroup $U$ of $M$ at $x$ is in $M^1$ and $M^0$ is transitive
on $S_{[x]}$\,, so $M^1 = UM^0$.  Let $u \in U$.  All Cartan subalgebras and 
all Weyl chambers of $\gu$ are $\Ad(U^0)$--conjugate, so we choose a
Weyl chamber $\gd \subset i\gt^*$ for $\gu$ and replace $u$ within $uU^0$
so that $\Ad(u)$ preserves $\gt$ and $\gd$.  Thus $\Ad(u)$ is an inner
automorphism of $\gg_\C$ that is the identity on $\gt$, so 
$u \in T \subset M^\dagger$.  We have shown $M^1 \subset M^\dagger$.
As $M^\dagger \subset M^1$ now $M^1 = M^\dagger$.
\end{proof}

\begin{corollary}\label{6.7.8} Given $G(x) \in X$ as in {\rm Proposition
\ref{6.7.4}} and $u \in U$, $\Ad(u)$ is an inner automorphism on $U^0$.
\end{corollary}

\section{Open Orbits and Discrete Series}
\label{sec7}
\setcounter{equation}{0}

Let $G$ be a reductive Lie group of our class specified in \S \ref{ssec3a}.  
We consider complex flag manifolds $X = \overline{G}_C/Q$ and open orbits 
$Y = G(x) \subset X$ such that $U = \{g \in G : g(x) = x\}$ is compact modulo
$Z$.  In \S \ref{ssec7a}  we see that these pairs $(X,x)$ exist precisely 
when $G$ 
has relative discrete series representations, that $U = Z_G(G^0)U^0$ with
$U^0 = U \cap G^0$, and that $Y$ has $|G/G^\dagger|$ topological components.
If $[\mu] \in \widehat{U}$ we show that the associated $G$--homogeneous
hermitian vector bundle $\cV_\mu \to Y$ has a unique $G$--homogeneous
holomorphic vector bundle structure.  That allows us to construct the
Hilbert spaces $H^{0,q}_2(Y;\cV_\mu)$ of square integrable harmonic 
$(0,q)$--forms on $Y$ with  values in $\cV_\mu$\,, and unitary representations 
$\pi^q_\mu$ of $G$ on $H^{0,q}_2(\cV_\mu)$.  The remainder of \S \ref{sec7} 
shows that the 
$[\pi^q_\mu]$, $q \geqq 0$ and $[\mu] \in \widehat{U}$, are the relative 
discrete series classes in $\widehat{G}$.
\medskip

Section \ref{ssec7b} is the formulation and history of our main result, 
Theorem 7.2.3.
Let $[\mu] \in \widehat{U}$.  Then $[\mu] = [\chi \otimes \mu^0]$ where
$[\chi] \in \widehat{Z_G(G^0)}$ and $[\mu^0] \in \widehat{U^0}$.  Let
$\Theta^{disc}_{\pi^q_\mu}$ denote the character of the discrete part of
$\pi^q_\mu$\,.  We prove
$$
\sum_{q \geqq 0} (-1)^q \Theta^{disc}_{\pi^q_\mu} =
	(-1)^{n + q(\lambda + \rho)} \Theta_{\pi_{\chi, \lambda + \rho}}
$$
where $\lambda$ is the highest weight of $\mu^0$, $n$ is the number of
positive roots, and $\rho$ is half the sum of the positive roots.  We
note that $H^{0,q}_2(\cV_\mu) = 0$ for $q \ne q(\lambda + \rho)$, and
we show that $[\pi^{q(\lambda + \rho)}_\mu] = [\pi_{\chi, \lambda + \rho}]
\in \widehat{G}_{disc}$\,.  Theorem 7.2.3 is proved in \S\S 7.3 through 7.7.
\medskip

We reduce the proof of Theorem 7.2.3 to the case $G = G^\dagger$ in \S 7.3, 
to the case $G = G^0$ in \S 7.4, and then further to the case where $Q$ is
a Borel subgroup of $\overline{G}_C$ in \S 7.5.  In \S 7.6 we use results
of Harish--Chandra and a method of W. Schmid to prove the alternating sum
formula for the $\Theta^{disc}_{\pi^q_\mu}$\,.  The vanishing statement comes 
out of work of Schmid \cite{S1976} cited above.  It combines with the
alternating sum formula to identify $[\pi_{\chi, \lambda + \rho}] \in 
\widehat{G}_{disc}$ as the discrete part of $[\pi^{q(\lambda + \rho)}_\mu]$\,.
This trick is due to Narasimhan and Okamoto.  Finally we use Corollary 5.7.2
of our Plancherel Theorem to show that $[\pi^{q(\lambda + \rho)}_\mu]$ has
no nondiscrete part, so $[\pi^{q(\lambda + \rho)}_\mu] = 
[\pi_{\chi, \lambda + \rho}]$, completing our proof in~\S 7.7.
\medskip
\medskip

\subsection{}\label{ssec7a}\setcounter{equation}{0}
$G$ is a Lie group of the class $\widetilde{\cH}$ of general real
reductive Lie groups defined in \S\ref{ssec3a}.  As explained in 
\S\ref{ssec6b}, $\overline{G} = G/Z_G(G^0)$ has
complexification $\overline{G}_C = \Int(\gg_C)$, and $G$ acts on the
complex flag manifolds of $\overline{G}_C$\,.  To realize the relative discrete
series of $G$ we work with 
\begin{equation}\label{7.1.1}
\begin{aligned}
&X = \overline{G}_C/Q \text{ complex flag manifold of }
	\overline{G}_C \\
&Y = G(x) \subset X \text{ open $G$--orbit such that} \\
&\phantom{XXX} \text{the isotropy subgroup $U$ of $G$ at $x$ 
	is compact modulo } Z.
\end{aligned}
\end{equation}

We collect some immediate consequences of (\ref{7.1.1}).

\begin{lemma}\label{7.1.2}
Suppose $(X,x)$ is given as in {\rm (\ref{7.1.1})}.  Then $U/Z$ contains 
a compact Cartan subgroup $H/Z$ of $G/Z$, so
$G$ has relative discrete series representations.  Further, the open orbit
$Y = G(x) \subset X$ is measurable and integrable, and $(X,x)$ is the
case $P = G$ of {\rm (6.7.1)}.  Finally, $U = Z_G(U^0)U^0$, $U \cap G^0 = U^0$,
$UG^0 = G^\dagger$, and $G/G^\dagger$ enumerates the topological
components of $Y$.
\end{lemma}

{\sl Remark.}\,\,\,  As a consequence of the second assertion, all 
possibilities for (\ref{7.1.1}) are enumerated in the paragraph following
Proposition \ref{6.7.4}.

\begin{proof}  Isotropy subgroups of $G$ on $X$ all contain Cartan
subgroups of $G$ by (\ref{6.2.3}).  Now the first assertion follows from
(\ref{7.1.1}) and Theorem \ref{3.5.8}.

$U$ acts on the tangent space at $x$ as $U/Z_G(G^0)$, which is compact by 
(\ref{7.1.1}).  Thus the orbit $Y = G(x)$ is measurable by 
Lemma \ref{6.4.5}.
As open orbits are integrable with $\gq_x + \tau\gq_x = \overline{\gg}_C$\,.
Now we have (\ref{6.7.1c}) and (\ref{6.7.1d}) with $P = M = G = N_{[x]}$\,.

Let $u \in U$.  Corollary \ref{6.7.8} says that $\Ad(u)$ is trivial on some
Cartan subalgebra of $\gu$, thus on a Cartan subalgebra of $\gg$.  Now
$\Ad(u)$ is an inner automorphism of $G^0$, i.e. $u \in G^\dagger$.
We have just seen $U \subset G^\dagger = Z_G(G^0)G^0$.  On the other
hand, open orbits are simply connected, so $U \cap G^0 = U^0$,  Thus
$U = Z_G(G^0)U^0$ and $UG^0 = G^\dagger$.  Since $UG^0$ is the
$G$--normalizer of $G^0(x)$, now $G/G^\dagger$ parameterizes the
components of $G(x)$.
\end{proof}
\medskip

The facts about $U$ in Lemma \ref{7.1.2} tell us 
\begin{equation}\label{7.1.3}
\begin{aligned}
&\widehat{U} = \{[\chi \otimes \mu^0] : [\chi] \in
	\widehat{Z_G(G^0)} \text{ and } [\mu^0] \in \widehat{U^0}\}; 
	\text{ so } \\
&\text{if } [\mu] \in \widehat{U} \text{ then its representation
	space } E_\mu \text{ has } \dim E_\mu < \infty \\
&\text{and we have } \E_\mu \to Y, \text{ a } G\text{--homogeneous hermitian 
	vector bundle.}
\end{aligned}
\end{equation}

\begin{lemma}\label{7.1.4} There is a unique complex
structure on $\E_\mu$ such that $\E_\mu \to Y$ is a $G$--homogeneous
holomorphic vector bundle.
\end{lemma}

\begin{proof}  The action of $G^0$ on $X$ maps $\gg_\C$ to a Lie
algebra of holomorphic vector fields.  Define
$$
\gl = \{\xi \in \gg_\C \mid \xi_x = 0\}, 
$$
isotropy subalgebra at $x$.  The homomorphism $G^0 \to \overline{G}$ induces
a homomorphism $\alpha$ of $\gg_C$ onto $\overline{\gg}_C$\,, and
$\gl = \alpha^{-1}(\gq_x)$.  Note that $\gu_C = \alpha^{-1}(\gq_x^r)$
reductive subalgebra of $\gl$.  Choose a linear algebraic group with Lie
algebra $\gg_\C$ and observe that $\alpha$ is a homomorphism of algebraic
Lie algebras.  Thus $\gu_\C$ is a maximal reductive subalgebra of $\gl$, 
and there is a nilpotent ideal $\gl^-$ such that $\gl =  \gu_\C + \gl^-$
semidirect sum.  Observe that $\Ad(u)\gl^- = \gl^-$ for all $u \in U$.

By extension of $\mu$ from U to $\gl$, we mean a (complex linear)
representation $\lambda$ of $\gl$ on $V_\mu$ such that
$$
\lambda|_\gu = \mu\,, \text{ i.e., } \lambda(\xi)
	= \mu(\xi) \text{ for all } \xi \in \gu,
$$
and 
$$
\mu(u)\lambda(\xi)\mu(u)^{-1} = \lambda(\Ad(u)\xi)
	\text{ for all } u \in U \text{ and } \xi \in \gl. 
$$
Let $\lambda$ be such an extension.  Then $\lambda(\gl^-)$ consists of
nilpotent linear transformations because $\lambda$ is an algebraic 
representation
of $\gl$, and that implies $\lambda(\gl^-)~=~0$ because $\mu$ is 
irreducible.  Thus there is just one extension of $\mu$ from $U$ to $\gl$\,;\,
it is given by $\lambda(\xi_1 + i \xi_2 + \eta) = \mu(\xi_1) + i\mu(\xi_2)$
where $\xi_1\,, \xi_2 \in \gu$ and~$\eta~\in~\gl^-$.

Our lemma now follows from the fact \cite[Theorem 3.6]{TW1970} that the
$G$--homogeneous holomorphic vector bundle structures on $\E_\mu \to Y$
are in bijective correspondence with the extensions of $\mu$ from $U$
to $\gl$. 
\end{proof}

Using (\ref{7.1.1}) we fix a $G$--invariant hermitian metric on the complex
manifold $Y$.  The unitary structure of $E_\mu$ specifies a
$G$--invariant hermitian metric on the fibers of $\E_\mu \to Y$.  Denote
\begin{equation}\label{7.1.5}
\begin{aligned}
&A^{p,q}(Y;\E_\mu) = \{ C^\infty (p,q)\text{--forms on }
	Y \text{ with values in } \E_\mu\} \text{ so we have }\\
&\text{Hodge--Kodaira maps 
	$A^{p,q}(Y;\E_\mu) \overset{\#}{\to} A^{n-p,n-q}(Y;\E_\mu^*)
	\overset{\widetilde{\#}}{\to} A^{p,q}(Y;\E_\mu)$}
\end{aligned}
\end{equation}
Here $n = \dim_\C Y$ and $\E_\mu^* = \E_{\mu^*}$ is the dual bundle.  If
$\alpha$, $\beta \in A^{p,q}(Y;\E_\mu)$ then $\alpha \wedge \#\beta \in
A^{n,n}(Y;\E_\mu \otimes \E_\mu^*)$.  The pairing $E_\mu \otimes E_\mu^* \to \C$
sends $\alpha \wedge \#\beta$ to an ordinary $(n,n)$--form on $Y$ that we
denote $\alpha \bar\wedge \#\beta$.  This gives us a pre Hilbert space
\begin{equation}\label{7.1.6}
A_2^{p,q}(\cV_\mu) = \left \{ \alpha \in A^{p,q}(Y;\E_\mu)
	\left | \int_Y \alpha \bar\wedge \#\alpha < \infty \right . \right \}\,,
\langle \alpha , \beta \rangle =
	\int_Y \alpha \bar\wedge \#\beta.
\end{equation}
The space of {\em square integrable} 
$(p,q)$--{\em forms} on $Y$ with values in $\E_\mu$ is
\begin{equation}\label{7.1.6c}
L_2^{p,q}(Y;\E_\mu): \text{ Hilbert space completion of } A_2^{p,q}(Y;\E_\mu).
\end{equation}

The operator $\overline{\partial}: A^{p,q}(Y;\E_\mu) \to A^{p,q+1}(Y;\E_\mu)$
is densely defined on $L_2^{p,q}(Y;\E_\mu)$ with formal adjoint
$\overline{\partial}^* = - \widetilde{\#}\,\overline{\partial}\,\#$. 
That gives us a second order elliptic operator 
\begin{equation}\label{7.1.7}
\square = (\overline{\partial} + \overline{\partial}^*)^2
	= \overline{\partial}\,\overline{\partial}^* + \overline{\partial}^*
	\,\overline{\partial}: \text{ Kodaira--Hodge--Laplacian.}
\end{equation}

The hermitian metric on $Y$ is complete by homogeneity, so the work of
Andreotti and Vesentini \cite{AV1965} applies.  First, it says that 
$\square$, with domain consisting of the compactly supported forms in 
$A^{p,q}(Y;\E_\mu)$,
is essentially self adjoint on $L_2^{p,q}(Y;\E_\mu)$.  We also write
$\square$ for the unique self adjoint extension, which coincides both with
the adjoint and the closure.  Its kernel
\begin{equation}\label{7.1.8}
H_2^{p,q}(Y;\E_\mu) = \{\omega \in L_2^{p,q}(Y;\E_\mu) \mid
	\square\, \omega = 0\}
\end{equation}
consists of the {\em square integrable harmonic}
$(p,q)$--{\em forms} on $Y$ with values in $\E_\mu$\,.
$H_2^{p,q}(Y;\E_\mu)$ is a closed subspace of $L_2^{p,q}(Y;\E_\mu)$.  It is
contained in $A_2^{p,q}(Y;\E_\mu)$  by ellipticity of $\square$.  Write
$c\ell$ for closure.  The Andreotti--Vesentini work shows that
$\overline{\partial}^*$ has closed range and gives us an orthogonal direct sum
\begin{equation}\label{7.1.9}
L_2^{p,q}(Y;\E_\mu) = 
	c\ell\, \overline{\partial} L_2^{p,q-1}(\cV_\mu)
	\oplus \overline{\partial}^* L_2^{p,q+1}(\cV_\mu)
	\oplus H_2^{p,q}(\cV_\mu)
\end{equation}
Here $\overline{\partial}$ has kernel 
$c\ell\,\overline{\partial} L_2^{p,q-1}(Y;\E_\mu)
\oplus H_2^{p,q}(Y;\E_\mu)$  and the kernel of 
$\overline{\partial}^*$ is the closed subspace 
$\overline{\partial}^* L_2^{p,q+1}(Y;\E_\mu)
        \oplus H_2^{p,q}(Y;\E_\mu)$. 
Thus $H_2^{p,q}(Y;\E_\mu)$ is a square integrable Dolbeault cohomology group.

The metrics and complex structures on $Y$ and $\E_\mu$ are invariant under 
the action of $G$.  Thus $G$ acts on $L_2^{p,q}(Y;\E_\mu)$ by a unitary
representation $\widetilde{\pi}_\mu^{p,q}$ that has a subrepresentation
\begin{equation}\label{7.1.10}
\pi_\mu^{p,q}: \text{ unitary representation of $G$ on $H_2^{p,q}(\E_\mu)$.}
\end{equation}
For convenience we also denote $\pi_\mu^q = \pi_\mu^{0,q}$.
The program here in \S \ref{sec7} is to represent the classes in 
$\widehat{G}_{disc}$ by the various $\pi_\mu^q$\,.
\medskip

\subsection{}\label{ssec7b}\setcounter{equation}{0}
Fix a compact Cartan subgroup $H/Z$ of $G/Z$ with $H \subset U$ as in 
Lemma \ref{7.1.2}.  Choose a system $\Pi$ of simple
$\gh_\C$--roots of $\gg_\C$ such that $\gq_x = \gq_\Phi$ where $\Phi \subset 
\Pi$.  Let $\Sigma^+$ denote the corresponding positive root system.  As usual
we pass to a $\Z_2$ extension if necessary so that $e^\rho$ and $\Delta$ are
well defined on $H^0$ and define
\begin{equation}\label{7.2.1} 
\begin{aligned}
&\rho = \frac{1}{2}\sum_{\phi \in \Sigma^+} \phi\,,\,\,\,\,
	\Delta = \prod_{\phi \in \Sigma^+} 
	(e^{\phi/2} - e^{-\phi/2})\,,\,\,\,\, \text{ and }\,\,
	\widetilde{\omega}(\cdot) = \prod_{\phi \in \Sigma^+}
	\langle \cdot, \phi\rangle, \\
&L = \{\lambda \in i\gh^* : e^\lambda \text{ is
	well defined on }
	H^0\} \text{ and } L' = \{\lambda \in L : \widetilde{\omega}(\lambda)
	\ne 0\}. 
\end{aligned}
\end{equation}
Let $\theta$ be the (unique) Cartan involution under which $H$ is
stable and $K = G^\theta$, so $H \subset U \subset K$.
Now $\Sigma^+ = \Sigma_k^+ \cup \Sigma_m^+$ (disjoint) where $\Sigma_k^+$
consists of the compact positive roots ($\overline{\gg}^\phi \subset \gk_\C$)
and $\Sigma_m^+$ consists of the noncompact positive roots
($\overline{\gg}^\phi \not\subset \gk_\C$).  If $\lambda \in L'$ we have
\begin{equation}\label{7.2.2}
q(\lambda) = |\{\phi \in \Sigma^+_k : \langle \lambda,
	\phi \rangle < 0\}| + |\{\phi \in \Sigma^+_m : \langle \lambda,
        \phi \rangle > 0\}|.
\end{equation}

Recall the statement of Theorem 3.5.9.  The main result of \S \ref{sec7} is

\begin{theorem}\label{7.2.3}  Let $[\mu] \in \widehat{U}$,
say $[\mu] = [\chi \otimes \mu^0]$ as in {\rm (7.1.3)}.  Let $\lambda$
be the highest weight of $\mu^0$ for the positive $\gh_\C$--root system
$\Sigma^+ \cap \Phi^r$ of $\gu_\C$.  Then $\lambda \in L$ and
$[\mu] \in \widehat{U}_\zeta$ where $\zeta \in \widehat{Z}$ coincides with
$e^\lambda$ on $Z \cap G^0$ and $[\chi] \in \widehat{Z_G(G^0)}_\zeta$\,.
Assume $\lambda + \rho \in L'$.  Then
$H^{0,q}_2(Y;\E_\mu) = 0$ whenever $q \ne q(\lambda + \rho)$, and
the natural action of $G$ on
$H^{0,q(\lambda + \rho)}_2(Y;\E_\mu)$ is the $\zeta$--discrete series
class $[\pi_{\chi,\lambda + \rho}]$\,.
\end{theorem}

Theorem \ref{7.2.3} gives a number of explicit geometric realizations
of the relative discrete series representations of $G$.  The case
where $G$ is a connected semisimple Lie group with finite center and
$U = H$ is due to W. Schmid (\cite{S1971}, \cite{S1976}); to some extent 
we follow his ideas.
The case where $G$ is a connected semisimple Lie group with finite center
and $Y = G(x)$ is a hermitian symmetric space was proved by
M. S. Narasimhan and K. Okamoto \cite{NO1970}.  Some results for groups with
possibly infinite center were proved by Harish--Chandra \cite{HC1956b} and
J. A. Tirao \cite{T1974}.  Also, W. Schmid (unpublished) and R. Parthasarathy 
(\cite{P1971}, \cite{P1972}) obtained realizations on spaces of square 
integrable harmonic spinors.  Finally, R. Hotta \cite{H1971} realized 
discrete series
representations of connected semisimple groups of finite center on
certain eigenspaces of the Casimir operator.

We carry out the proof of Theorem \ref{7.2.3} in \S\S \ref{ssec7c} 
through \ref{ssec7g}.  

\subsection{}\label{ssec7c}\setcounter{equation}{0}
We reduce Theorem \ref{7.2.3} to the case $G = G^\dagger$\,.

Choose a system $\{g_1, \cdots , g_r\}$ of coset representatives
of $G$ modulo $G^\dagger$\,.  According to Lemma \ref{7.2.1}, the topological
components of $Y = G(x)$ are the $Y_i = G^\dagger(g_ix)$.  Let $^i\pi_\mu^q$
denote the representation of $G^\dagger$ on
$$
H_2^{0,q}(Y;\E_\mu|_{Y_i}) = \{\omega \in H_2^{0,q}(Y;\E_\mu)\,
	:\, \omega \text{ is supported in } Y_i\}.
$$
Evidently $H_2^{0,q}(Y;\E_\mu) = H_2^{0,q}(Y;\E_\mu|_{Y_1}) \oplus
	\cdots \oplus H_2^{0,q}(Y;\E_\mu|_{Y_r})$ as orthogonal direct sum.
Thus $\pi_\mu^{0,q} = {^1\pi}_\mu^{0,q} \oplus \cdots \oplus 
{^r\pi}_\mu^{0,q}$\,.  Also, $\pi_\mu^q(g_i)$ sends $H_2^{0,q}(Y;\E_\mu|_{Y_j})$
to $H_2^{0,q}(Y;\E_\mu|_{Y_k})$ where $g_iY_k = Y_j$\,, i.e., where
$g_i^{-1}g_j \in g_kG^\dagger$.  In summary,

\begin{lemma}\label{7.3.1}
$\pi_\mu^q = \Ind_{G^\dagger}^G
({^i\pi}_\mu^{0,q})$ for $1 \leqq i \leqq r$.
\end{lemma}

We know from Theorem \ref{3.5.9} that $\widehat{G}_{disc}$ consists of the
classes $[\pi] = [\Ind_{G^\dagger}^G(\pi^\dagger)]$ where
$[\pi^\dagger] \in \widehat{G^\dagger}_{disc}$\,.  Further, $\Theta_\pi$
is supported in $G^\dagger$, where it coincides with 
$\Theta_{\pi_{G^\dagger}}$\,.  Now Lemma \ref{7.3.1} tells us that, if Theorem
\ref{7.2.3} holds for $G^\dagger$ with each of the $\E_{\mu|_{Y_i}}$\,.
Then Theorem \ref{7.2.3} is valid for $G$ with $\E_\mu$\,.  In summary,
\begin{lemma}\label{7.3.2}
In the proof of {\rm Theorem \ref{7.2.3}}
we may assume $G = G^\dagger$\,.
\end{lemma}

\subsection{}\label{ssec7d}\setcounter{equation}{0}
We reduce Theorem \ref{7.2.3} to the case where $G$
is connected.  Using Lemma \ref{7.3.2} we assume $G = G^\dagger$.  Thus
$G = Z_G(G^0)G^0$ and $Y = G(x)$ is connected.  Recall
$[\mu] = [\chi \otimes \mu^0]$ with $[\chi] \in \widehat{Z_G(G^0)}$
and $[\mu^0] \in \widehat{U^0}$, so $E_\mu = E_\chi \otimes E_{\mu^0}$\,.
Now $[\mu^0]$ specifies a $G^0$--homogeneous holomorphic vector
bundle $\E_{\mu^0} \to Y$.  Let $\pi^q_{\mu^0}$ denote the representation
of $G^0$ on $H^{0,q}_2(Y;\E_{\mu^0})$.

\begin{lemma}\label{7.4.1}
$\pi^q_{\mu} = \chi \otimes \pi^q_{\mu^0}$
for all $q \geqq 0$.
\end{lemma}

\begin{proof} $Z_G(G^0)$ acts trivially on $X$, so it acts trivially
on the bundle of ordinary $(0,q)$--forms over the orbit $Y \subset X$.
Thus $Z_G(G^0)$ acts on $L_2^{0,q}(Y;\E_\mu)$ as a type I primary 
representation $\omega_\chi$\,.  In particular $\pi_\mu^q|_{Z_G(G^0)}$ is 
a multiple of $\chi$.  But $\mu|_{U^0} = (\dim \chi)\mu^0$ so
$\pi_\mu^q|_{G^0} = (\dim \chi)\pi^q_{\mu^0}$\,.  We conclude $\pi_\mu^q
= \chi \otimes \pi^q_{\mu^0}$\,. 
\end{proof}

We know from Proposition \ref{3.5.2} that $\widehat{G^\dagger}_{disc}$ consists 
of the $[\chi \otimes \pi^0]$ where $[\chi] \in \widehat{Z_G(G^0)}$ and
$[\pi^0] \in \widehat{G^0}_{disc}$ agree on $Z_{G^0}$\,.  The distribution
character $\Theta_{\chi \otimes \pi^0} = (\trace \chi)\Theta_{\pi^0}$\,.
If Theorem \ref{7.2.3} holds for $G^0$ with $\E_{\mu^0}$ now 
Lemma \ref{7.4.1} ensures
the result for $G$ with $\E_\mu$\,.  In summary

\begin{lemma}\label{7.4.2}
In the proof of {\rm Theorem 7.2.3} we may 
assume $G$ is connected.
\end{lemma}

\subsection{}\label{ssec7e}\setcounter{equation}{0}
We reduce Theorem \ref{7.2.3} to the case where $G = G^0$
and $U = H^0$.

Choose a Borel subgroup $B \subset Q$ of $\overline{G}_\C$\,.  Denote 
$X' = \overline{G}_\C/B$ and consider the $G$-equivariant projection
$r: X' \to X$ defined by $r(\bar g B) = \bar g Q$.  Now choose a base
point $x' \in r^{-1}(x)$ defined by $\gb_{x'} = \gq_\emptyset$
relative to $(\overline{\gh}_\C, \Pi)$.
Since $\tau\phi = -\phi$ for every $\gh$-root, the isotropy subalgebra
of $\gg$ at $x'$ is just $\gh$.  Now 
$Y' = G(x')$ is open in $X'$ and
$H = \{g \in G : g(x') = x'\}$, and
$r:Y' \to Y$ is $G$-equivariant and holomorphic.

Following Lemma \ref{7.4.2}, we assume $G$ connected, so $U$ and $H$ are
connected.  Now $\lambda$ is the highest weight of $\mu$ and $e^\lambda 
\in \widehat{H}$ specifies
\begin{equation}\label{7.5.2}
\begin{aligned}
&\L_\lambda \to Y':\,\, G\text{--homogeneous holomorphic line bundle.} \\
&\pi_\lambda^q:\,\, \text{ representation of } G \text{ on } H_2^{0,q}
	(Y;\L_\lambda). 
\end{aligned}
\end{equation}
\begin{lemma}\label{7.5.3}
$[\pi_\lambda^q] = [\pi_\mu^q]$\,.
\end{lemma}
\begin{proof}
This is a Leray spectral sequence argument.
Let $\cO(Y';\L_\lambda)$ denote the sheaf of germs of holomorphic sections of
$\L_\lambda \to Y'$.  Each integer $s \geqq 0$ gives a sheaf 
$\cR^s(Y;\L_\lambda) \to Y$, associated to the presheaf that assigns the
sheaf cohomology group $H^s(Y' \cap r^{-1}D; \cO(Y';\L_\lambda))$ to
an open set $D \subset Y$.  Since $r: Y' \to Y$ is a holomorphic fiber
bundle, $\cR^s(\L_\lambda)$ is the sheaf of germs of holomorphic
sections of the holomorphic vector bundle over $Y$ whose fiber at
$y \in Y$ is $H^s(Y' \cap r^{-1}(y); \cO(\L_\lambda))$.

Recall our 
Borel-Weil Theorem from Proposition \ref{1.1.12} with $q_0 = 0$, and
apply it to $Y' \cap r^{-1}(y) = U(x') \cong U/H$.  That says
$H^0(Y' \cap r^{-1}(y); \cO(\L_\lambda)) = E_\mu$ as $U$-module and
$H^s(Y' \cap r^{-1}(y); \cO(\L_\lambda)) = 0$ for $s > 0$.  Now
$\cR^0(\L_\lambda) = \cO(\E_\mu)$ and $H^s(Y;\L_\lambda) = 0$ for $s > 0$.

Our analysis of the direct image sheaves $\cR^s(\L_\lambda)$ shows that
the Leray spectral sequence collapses for $r: Y' \to Y$, so each
$H^q(Y';\cO(\L_\lambda)) = H^q(Y;\cO(\E_\mu))$ as $G$-modules.  More
to the point, we carry the spectral sequence over from sheaf cohomology
to Dolbeault cohomology and use the Andreotti--Vesentini theory 
((\ref{7.1.8}) and (\ref{7.1.9}))
to restrict considerations to square integrable forms.  Then the
resultant square integrable Leray spectral sequence collapses and we
conclude that each $H_2^{0,q}(Y';\L_\lambda) = H_2^{0,q}(Y;\E_\mu)$
as $G$-modules.  
\end{proof}

As immediate consequence of Lemmas \ref{7.4.2} and \ref{7.5.3} we have

\begin{lemma}\label{7.5.4}
In the proof of {\rm Theorem \ref{7.2.3}} we may
assume that $G$ is connected, that $Q$ is a Borel subgroup of $\overline{G}_\C$
and that $U = H$.
\end{lemma}

\subsection{}\label{ssec7f}\setcounter{equation}{0}
Next, we prove the formula 
$\sum_{q \geqq 0} \Theta^{disc}_{\pi_\mu^q} =
 (-1)^{|\Sigma^+| + q(\lambda + \rho)}\Theta_{\pi_{\chi, \lambda + \rho}}$\,\,.
By Lemma \ref{7.5.4} we may assume that $G$ is connected, that
$Q$ is a Borel subgroup of $\overline{G}_\C$ and that $U = H$.
$K/Z$ is the maximal compact subgroup of $G/Z$ that contains the compact
Cartan subgroup $H/Z$.  If $[\pi] \in \widehat{G}_\zeta$ Lemma \ref{3.2.2},
and an argument \cite[\S 5]{HC1954b} of Harish--Chandra say that
\begin{equation}\label{7.6.1}
\begin{aligned}
&\pi|_K = 
	\sum_{\widehat{K}_\zeta} m_\kappa\cdot\kappa \text{ where }
	0 \leqq m_\kappa \leqq n_G (\dim\kappa) \\
&(\pi|_K)(f) = \int_K f(k)\pi(k)dk, f \in C^\infty_c(K),
	\text{ is of trace class, and} \\ 
&T_\pi : C^\infty_c(K) \to \C \text{ defined by }
	f \mapsto \trace(\pi|_K)(f) 
	\text{ is a distribution on } K.  
\end{aligned}
\end{equation}
Harish--Chandra's argument \cite[\S 12]{HC1956c} now shows that 
$T_\pi|_{K \cap G'}$
is a real analytic function on $K \cap G'$ and that 
$T_\pi|_{K \cap G'} = \Theta_\pi|_{K \cap G'}$\,\,.

Recall the Cartan involution $\theta$ of $G$ with fixed point set $K$.  Fix a
nondegenerate invariant bilinear form $\langle\phantom{x},\phantom{x}\rangle$ on
$\gg_\C$ that restricts to the Killing form on the derived algebra and is
negative definite on $\gk = (\gk \cap [\gg,\gg])\oplus \gc$\,.  That gives 
us a positive definite $\Ad(K)$-invariant hermitian inner product
$(u,v) = -\langle u, \theta\tau v\rangle$ 
on $\gg_\C$ where $\tau$ is complex conjugation of $\gg_\C$ over $\gg$\,.

Consider the nilpotent algebra $\gn = \sum_{\phi \in \Sigma^+} 
	\overline{\gg}^{-\phi} = \gq_x^u \subset \gg_\C$\,. 
Denote 
\begin{equation}\label{7.6.3}
\Lambda(\Ad^*) = \sum_{j \geqq 0} \Lambda^j(\Ad^*): 
	\text{ representation of } \gq_x \text{ on }
	\Lambda\gn^* = \sum_{j \geqq 0} \Lambda^j\gn^*\,. 
\end{equation}
The inner product $(\phantom{x},\phantom{x})$ gives $\gn$, thus $\gn^*$, 
thus also $\Lambda\gn^*$,
a Hilbert space structure; and $\Ad^*(\gh)$ acts by skew-hermitian
transformations.

Fix $[\pi] \in \widehat{G}_\zeta$ and write $\cH_\pi$ for its
representation space.  Let $\cH_\pi^0$  denote the space
of $K$-finite vectors in $\cH_\pi$\,. It is dense and consists of
analytic vectors, by (\ref{7.6.1}).  Now $\gh$ acts on 
	$\cH_\pi^0\otimes \Lambda\gn^*$ by $\pi \otimes \Ad^*$
by skew--hermitian transformations.  Let 
$\{y_1,\cdots , y_n\}$ be a basis of $\gn$, let $\{\omega^j\}$ be the dual
basis of $\gn^*$, and $e(\omega^j): \Lambda\gn^* \to \Lambda\gn^*$
the exterior product.  Then $\delta := 
	\sum \{(\pi(\omega_j) \otimes e(\omega^j)) + \frac{1}{2}1 \otimes
		(e(\omega_j)\Ad^*(y_j))\}$
is the coboundary operator $\cH_\pi^0 \to \cH_\pi^0$ of the Lie algebra
cohomology for the action of $\gh$.  It has formal adjoint 
$\delta^* = \sum \{(-\pi(\tau y_j)\otimes i(\omega^j)) + (\frac{1}{2}1 \otimes
		\Ad^*(y_j)^* i(\omega^j))\}$ 
where $i(\omega^j)$ denotes interior product.  Now $\delta + \delta^*$ is
a densely defined symmetric operator on $\cH_\pi \otimes \Lambda \gn^*$.

Choose a basis $\{z_i\}$ of $\gk_\C$ that is orthonormal relative to
$(\phantom{x},\phantom{x})$.  Then $\Omega_K = \sum z_iz_i \in \cU(\gk)$
is independent of choice of the basis $\{z_i\}$.  In particular
$\Omega_K$ is a linear combination, positive coefficients, of the
Casimir operators of the simple ideals of $\gk$ plus the Laplacian
on the center of $\gk$.  Thus (\ref{7.6.1})  $\pi(\Omega_K)$ is symmetric
non-negative on $\cH_\pi^0$ and has a unique self adjoint extension
$\pi(\Omega_K)$ to $\cH_\pi$\,.  Further $\cH_\pi$ is the discrete
direct sum of the (all non--negative) eigenspaces of $\pi(\Omega_K)$.  As
$$
\{ [\kappa] \in \widehat{K}_\zeta :
	\kappa(\text{ Casimir element of } \cU(\gk)) \leqq c\}
$$
is finite for every real $c$, (\ref{7.6.1}) also says that the sum of the 
eigenspaces of $\pi(\Omega_K)$ for eigenvalues $\leqq c$ has finite
dimension.  Thus 
$(1 + \pi(\Omega_K))^{-1}$ is a
	self adjoint compact operator on $\cH_\pi$\,.
With this preparation, Wilfried Schmid's arguments \cite[\S 3]{S1971}.
are valid in our situation.  We state the result.

\begin{lemma}\label{7.6.4}
The closure of $\delta + \delta^*$ from the
domain $\cH_\pi^0 \otimes \Lambda\gn^*$ is the unique self adjoint extension
of $\delta + \delta^*$ on $\cH_\pi \otimes \Lambda\gn^*$.  Each
$$
\cH^q(\pi): \text{ kernel of }
	\delta + \delta^* \text{ on } \cH_\pi \otimes \Lambda^q \gn^*
$$
is a finite dimensional $H$-module.  Define 
$f_\pi = \sum (-1)^q
	(\text{character of } H \text{ on } \cH^q(\pi))$.
Let $\Delta$ and $\rho$ be as in {\rm (\ref{7.2.1})} and $n = \dim_\C\gn = 
|\Sigma^+|$.  Then 
$$
f_\pi|_{H \cap G'} =
	(-1)^n \Delta e^\rho\cdot T_\pi|_{H\cap G'}.
$$
\end{lemma}

Let $d\pi$ denote Plancherel measure on $\widehat{G}_\zeta$\,, so
$L_2(G/Z,\zeta) = \int_{\widehat{G}_\zeta} \cH_\pi \widehat{\otimes}
\cH_\pi^*\, d\pi$.  We have the unitary $G$-module structure 
$L_2^{0,q}(Y;\L_\lambda) =
\int_{\widehat{G}_\zeta} \cH_\pi \widehat{\otimes}\{\cH_\pi^* \otimes
\Lambda^q \gn^* \otimes L_\lambda\}^H d\pi$
where $L_\lambda$ is the representation space of $e^\lambda$, where
$H$ acts on $\cH_\pi^* \otimes \Lambda^q \gn^* \otimes L_\lambda$ by
$\pi^* \otimes \Ad^* \otimes e^\lambda$, and where $(\cdot )^H$
denotes the fixed points of $H$ there.  Now
$\overline{\partial}: A^{0,q}(Y;\L_\lambda) \to A^{0,q+1}(Y;\L_\lambda)$
and its formal adjoint $\overline{\partial}^*$ act by \hfill\newline
\phantom{XX}$\overline{\partial}(f\cdot\omega^J
	\cdot \ell) = \sum_{1\leq k\leq n} (y_k(f)\cdot e(\omega^k)\omega^J
	\cdot \ell) + \frac{1}{2} \sum_{1\leq k\leq n} (f\cdot e(\omega^k)
	\Ad^*(y_k)\omega^J\cdot\ell)$\hfill\newline
and \hfill\newline
\phantom{XX} $ \overline{\partial}^*(f\cdot\omega^I
	\cdot \ell) = -\sum_{1\leq k\leq n}(\tau(y_k)f\cdot i(\omega^k)\omega^I
	\cdot\ell) + \frac{1}{2} \sum_{1\leq k\leq n} (f\cdot \Ad^*(y_k)^*
	i(\omega^k)\omega^I\cdot\ell)$ \hfill\newline
where $I$ and $J$ are multi-indices and $\ell \in L_\lambda$\,.
These correspond to the formulae for $\delta$ and $\delta^*$.  The
argument of \cite[Lemmas 5 and 6]{S1971} shows that
$[\pi] \mapsto \{\cH^q(\pi^*) \otimes L_\lambda\}^H$ 
is a measurable assignment of Hilbert spaces on $\widehat{G}_\zeta$\,, 
and that 
\begin{equation}\label{7.6.5}
\begin{aligned}
&\cH_2^{0,q}(Y;\L_\lambda) =
	\int_{\widehat{G}_\zeta} \cH_\pi \otimes \{\cH^q(\pi^*)
	\otimes L_\lambda\}^H d\pi, \text{ unitary $G$--module; i.e.,}\\
&\pi_\lambda^q = \int_{\widehat{G}_\zeta}
	\dim (\cH^q(\pi^*)\otimes L_\lambda)^H\cdot\pi\,d\pi
	\text{, which has discrete part}\\
&^0\pi_\lambda^q = {\sum}_{\widehat{G}_{\zeta-disc}}
	\dim (\cH^q(\pi^*)\otimes L_\lambda)^H\cdot\pi.
\end{aligned}
\end{equation}
The equation for $^0\pi_\lambda^q$ is summation over the 
discrete set $W_G\backslash \{\nu \in L' : e^{\nu - \rho}|_Z = \zeta\}$.

Write $\Theta_{\pi^q_\lambda}^{disc}$ for the formal sum of the 
characters of the irreducible (in this case,  $\zeta$-discrete) 
subrepresentations of $\pi^q_\lambda$\,.  Define 
$$
\begin{aligned}
F_\lambda = &{\sum}_{q \geqq 0} (-1)^q 
        \dim (\cH^q(\pi^*)\otimes L_\lambda)^H\cdot \Theta_\pi\\
	&= {\sum}_{\pi \in \widehat{G}_{\zeta\text{-disc}}}
	(\sum_{q \geqq 0} (-1)^q\,
	\dim (\cH^q(\pi^*)\otimes L_\lambda)^H\cdot \Theta_\pi \\
	&= {\sum}_{\widehat{G}_{\zeta-disc}}
        {\rm (coefficient\,\, of \,\,}e^{-\lambda} {\rm \,\,in\,\, } 
        f_{\pi^*})\, \Theta_\pi
\end{aligned}
$$
because $\dim (\cH^q(\pi^*)\otimes L_\lambda)^H$ is the 
multiplicity of $e^{-\lambda}$ for $H$ on $\cH^q(\pi^*)$.
Let $\nu \in L'$ with $e^{\nu - \rho}|_Z = \zeta$, so
$[\pi_\nu] \in \widehat{G}_{\zeta-disc}$\,.  By (\ref{7.6.1}),
$$
\Delta\, T_{\pi_\nu^*|_{H \cap G'}}
	= \Delta\, T_{\pi_{-\nu}}|_{H \cap G'} = (-1)^{q(\nu)}
	\sum_{w \in W_G} \det(w)e^{-w(\nu)}.
$$
Thus Lemma 7.6.4 says $f_{\pi_\nu^*} = (-1)^{n + q(\nu)}
	\sum_{w \in W_G} \det(w)e^{-w(\nu)}$. 
In particular the coefficient of $e^{-\lambda}$ in $f_{\pi_\nu^*}$ is $0$ if
$\lambda + \rho \not\in W_G(\nu)$, is $(-1)^{n + q(\nu)}\det(w)$ if $w(\nu)
= \lambda + \rho$ for some $w \in W_G$\,.  Now $F_\lambda = (-1)^{n + q(\nu)}
	\Theta_{\pi_{\lambda + \rho}}$\,.
In view of Lemma \ref{7.5.4}, we have just proved the alternating sum formula
\begin{equation}\label{7.6.7}
{\sum}_{q \geqq 0} \Theta^{disc}_{\pi_\mu^q} =
 (-1)^{|\Sigma^+| + q(\lambda + \rho)}\Theta_{\pi_{\chi, \lambda + \rho}}\,.
\end{equation}
This is a key step in the proof of Theorem 7.2.3.

\subsection{}\label{ssec7g}\setcounter{equation}{0}
We complete the proof of Theorem \ref{7.2.3}.  The crux of the matter is the
vanishing statement, 
$$
\text{ if $\lambda + \rho \in L'$ then
		$H^{0,q}_2(Y;\E_\mu) = 0$ for $q \ne q(\lambda + \rho)$ }
$$
combined with the alternating sum formula (\ref{7.6.7}).

The vanishing statement was proved by Griffiths and Schmid 
\cite[Theorem 7.8]{GS1969} for the case where $G$ is a connected semisimple 
Lie group with finite center, $Q$ is a Borel subgroup of $\overline{G}_\C$\,, 
and $\lambda + \rho$ is ``sufficiently'' far from the walls of the Weyl 
chamber that contains it.  Then the requirement of ``sufficiently''
far from the wall was eliminated by Schmid \cite{S1976} using methods not 
available earlier.  Both proofs go through without change in our case. 

Now we have $H_2^{0,q}(Y;\L_\lambda) = 0$ for $q \ne q(\lambda + \rho)$.  Using
the alternating sum formula (\ref{7.6.7}) and linear independence of the 
$\Theta_\pi$ for $[\pi] \in \widehat{G}_{\zeta-disc}$, we see that 
\begin{equation}\label{7.7.1}
[\pi_{\lambda + \rho}] \text{ is the discrete part }
	[{^0\pi}^{q(\lambda + \rho)}_\lambda] 
	\text{ of } [\pi^{q(\lambda + \rho)}_\lambda].
\end{equation}
Corollary \ref{5.7.3} applied to $\overline{\zeta}$, with
(\ref{7.6.1}), tells us that
$$
\{[\pi] \in \widehat{G}_\zeta \setminus \widehat{G}_{\zeta-disc} \mid T_\pi
\ne 0\} \text{ has Plancherel measure zero in } \widehat{G}_\zeta\,.
$$
Lemma \ref{7.6.4} says 
$f_{\pi^*} = 0$ for almost all $[\pi] \in \widehat{G}_\zeta \setminus 
\widehat{G}_{\zeta-disc}$\,.  If $q \ne q(\lambda + \rho)$ now (\ref{7.6.5}) 
and $H^{0,q}_2(Y;\E_\mu) = 0$ force $(\cH^q(\pi^*)\otimes L_\lambda)^H = 0$, so
$e^{-\lambda}$ has multiplicity $0$ in the representation of $H$ on 
$\cH^q(\pi^*)$.  If $f_{\pi^*} = 0$ then also $e^{-\lambda}$ has 
multiplicity $0$ in the representation 
of $H$ on $\cH^{q(\lambda + \rho)}(\pi^*)$, so 
$(\cH^{q(\lambda + \rho)}(\pi^*)\otimes L_\lambda)^H = 0$.  In
summary, 
\begin{equation}\label{7.7.2}
(\cH^{q(\lambda + \rho)}(\pi^*)\otimes L_\lambda)^H = 0
\text{ for almost all } 
[\pi] \in \widehat{G}_\zeta \setminus \widehat{G}_{\zeta-disc}\,.
\end{equation}

The measure $(\dim(\cH^{q(\lambda + \rho)}(\pi^*)\otimes L_\lambda)^Hd\pi$ on
$\widehat{G}_\zeta$ is concentrated on $\widehat{G}_{\zeta-disc}$ by 
(\ref{7.7.2}).  Now (\ref{7.6.5}) says that 
$[\pi_\lambda^{q(\lambda + \rho)}] = 
[{^0\pi}_\lambda^{q(\lambda + \rho)}]$, so
$[\pi_\lambda^{q(\lambda + \rho)}] = [\pi_{\lambda + \rho}]$ by
(\ref{7.7.1}).  That completes the proof of Theorem \ref{7.2.3}.

\section{Measurable Orbits and Nondegenerate Series}\label{sec8}
\setcounter{equation}{0}

Let $G$ be a reductive Lie group from the class $\widetilde{\cH}$ of general
real reductive Lie groups defined in \S \ref{ssec3a}.  If $\zeta \in
\widehat{Z}$ then Theorem \ref{5.1.6} shows that Plancherel measure on
$\widehat{G}_\zeta$ is supported by the constituents of $H$--series
classes that transform by $\zeta$, as $H$ runs over the conjugacy classes
of Cartan subgroups of $G$.  Here we work out geometric realizations for 
all these $H$--series classes.  Our method is a reduction to the special
case of the relative discrete series ($H/Z$ compact) that we studied
in \S \ref{sec7}.

Fix a Cartan subgroup $H = T \times A$ in $G$ and an associated
cuspidal parabolic subgroup $P = MAN$ of $G$.  We work over measurable
orbits $Y = G(x) \subset X = \overline{G}_\C /Q$ 
such that (i) the $G$--normalizer $N_{[x]}$ of the holomorphic arc
$S_{[x]}$ is open in $P$ and (ii) $U = \{m \in M : m(x) = x\}$ is compact
modulo $Z$.  In \S \ref{ssec8a} we first check that $G$ has isotropy subgroup 
$UAN$ at $x$, $U = Z_M(M^0)U^0$, and that $N_{[x]} = M^\dagger AN$.  
If $[\mu] \in \widehat{U}$ and $\sigma \in \ga^*$ we show that the 
$G$--homogeneous complex vector bundle 
$$
p: \E_{\mu,\sigma} \to G/UAN = Y
	\text{ associated to } \mu \otimes e^{\rho_\ga + i\sigma}
$$
is holomorphic over each holomorphic arc component of $Y$, in an essentially
unique manner. 

Let $K$ be the fixed point set of a Cartan involution that
preserves $H$.  Since $\mu$ is unitary we get a $K$--invariant hermitian
metric on $\E_{\mu,\sigma}$\,.  Since $U/Z$ is compact we get a $K$--invariant
assignment of hermitian metrics on the holomorphic arc components of $Y$.
This results in Hilbert spaces $\cH_2^{p,q}(Y;\E_{\mu,\sigma})$
of ``square integrable partially harmonic $(p,q)$--forms'' on $Y$ with
values in $\E_{\mu,\sigma}$: measurable $\omega$ such that (i) 
$\omega|_{S_{[kx]}}$ is a harmonic $(p,q)$--form on $S_{[kx]}$ with values
in $\E_{\mu,\sigma}|_{S_{[kx]}}$ and $L_2$ norm 
$||\omega|_{S_{[kx]}}|| < \infty$ for almost all $k \in K$ and (ii)
$\int_{K/Z} ||\omega|_{S_{[kx]}}||^2 d(kZ) < \infty$.  We end \S \ref{ssec8a}
by showing that the natural action of $G$ on $\cH_2^{p,q}(Y;\E_{\mu,\sigma})$
is a unitary representation.

The representation of $G$  on $\cH_2^{0,q}(Y;\E_{\mu,\sigma})$ is denoted
$\pi_{\mu,\sigma}^q$\,.  Let $\eta_\mu^q$ denote the representation of
$M^\dagger$ on $\cH_2^{0,q}(S_{[x]};\E_{\mu,\sigma}|_{S_{[x]}})$; we studied 
these in Section \ref{sec7}.  Now we have a representation 
$$
\eta_{\mu,\sigma}^q(man) = e^{i\sigma}(a)
	\eta_\mu^q(m) \text{ of } N_{[x]} = M^\dagger AN.
$$
In \S \ref{ssec8b} we prove 
$[\pi_{\mu,\sigma}^q] = [\Ind_{N_{[x]}}^G(\eta_{\mu,\sigma}^q)]$.

Our main result is Theorem \ref{8.3.4}.  Split $[\mu] = [\chi\otimes\mu^0]$ 
where
$[\chi] \in \widehat{Z_M(M^0)}$ and $[\mu^0] \in \widehat{U^0}$, 
where $[\mu^0]$ has highest
weight $\nu$ such that $\nu + \rho_\gt$ is $\gm$--regular. Then the
$H$--series constituents of $\pi_{\mu,\sigma}^q$ are just its irreducible
subrepresentations.  Their sum ${^H\pi}^q_{\mu,\sigma}$ has distribution
character $\Theta^H_{\pi_{\mu,\sigma}^q}$\,\,.  Further  
$$
{\sum}_{q \geqq 0} (-1)^q
	\Theta^H_{\pi_{\mu,\sigma}^q} 
	= (-1)^{|\Sigma^+_\gt| + q_M(\nu + \rho_\gt)} 
		\Theta_{\pi_{\chi,\nu + \rho_\gt,\sigma}}\,.
$$
Also, if $q \ne q_M(\nu + \rho_\gt)$ then
$\cH_2^{0,q}(S_{[x]};\E_{\mu,\sigma}) = 0$.  This combines with the
alternating sum formula and some consequences of the Plancherel Theorem,
yielding 
$$
[\pi_{\mu,\sigma}^{q_M(\nu + \rho_\gt)}]
	= [\pi_{\chi, \nu + \rho_\gt,\sigma}], 
	\text{ $H$--series class.}
$$
The proof is a matter of applying the results from Section \ref{sec7} to 
every holomorphic arc component of $Y$ and combining those results by means 
of the induced representation theorem of \S \ref{ssec8b}.

\subsection{}\label{ssec8a}\setcounter{equation}{0}
$G$ is a general real reductive Lie group from our class $\widetilde{\cH}$ 
defined in \S \ref{ssec3a}.  As noted at the end of \S \ref{ssec6b}, 
$\overline{G} = G/Z_G(G^0)$ is a linear semisimple group with complexification 
$\overline{G}_\C = \Int(\gg_\C)$, and $G$ acts on the complex 
flag manifolds of $\overline{G}_\C$\,.

For the remainder of Section \ref{sec8} we fix a Cartan subgroup $H = T\times A$
of $G$ and an associated cuspidal parabolic subgroups $P=MAN$.
In order to realize the $H$--series of $G$ we work with a complex flag
manifold $X = \overline{G}_\C/Q$ and a measurable $G$--orbit 
$Y = G(x) \subset X$
such that the $G$--normalizers of the holomorphic arc components  of $Y$
in $X$ have the property 
\begin{equation}\label{8.1.2c}
N_{[x]} = \{g \in G : gS_{[x]} = S_{[x]}
	\text{ has Lie algebra } \gp.
\end{equation}
Since the orbit $Y = G(x)$ is measurable, it is partially complex and of
flag type.  Thus $S_{[x]}$ is an open $M^0$--orbit on the smaller flag manifold
$\overline{M}_\C(x)$ where $\overline{M} = M/Z_G(G^0)$ and $AN$ acts
trivially on $S_{[x]}$\,.  The isotropy group of $G$ at $x$ is $UAN$
where $T \subset U \subset M$.  We require that 
\begin{equation}\label{8.1.2d}
U/Z_G(G^0) = \{m \in M : m(x) = x\}/Z_G(G^0) \text{ is compact.}
\end{equation}

The $G$--orbits discussed studied in Theorem \ref{6.7.4} 
form the special case in which the orbit is integrable  We obtain a 
number of examples of that class from the construction in the paragraph
after the proof of Theorem \ref{6.7.4}.

\begin{lemma}\label{8.1.3}
Suppose that $G(x) \subset X$ is a measurable orbit, that
$N_{[x]} = \{g \in G : gS_{[x]} = S_{[x]}$ has Lie algebra $\gp$, and 
that the isotropy group of $G$ at $x$ is $UAN$ with $U/Z_G(G^0)$
compact. 
Then the open orbit $M(x) \subset \overline{M}_\C(x)$ is measurable and
integrable.  Further $U = Z_M(M^0)U^0$\,, $U \cap M^0 = U^0$\,,
$UM^0 = M^\dagger$\,, and $M/M^\dagger$ generates the topological
components of $M(x)$.  Finally, $N_{[x]} = M^\dagger AN$, and $G/M^\dagger G^0$
enumerates the topological components of $Y = G(x)$.
\end{lemma}

\begin{proof}
The open orbit $M(x) \subset \overline{M}_\C(x) =
\overline{M}_\C/(Q \cap \overline{M}_\C)$ satisfies (\ref{7.1.1}).  Applying
Lemma \ref{7.1.2} to it, we get the first two assertions.  For the third,
$N_{[x]} = UN_{[x]}^0 = UM^0 AN = M^\dagger AN$, and the $G$--normalizer
of $G^0(x)$ is $UG^0 = UM^0G^0 = M^\dagger G^0$. 
\end{proof}

\begin{remark}\label{8.1.3x}{\rm
$G^\dagger \subset M^\dagger G^0$ in general, but one
can have $G^\dagger \ne M^\dagger G^0$.  For example let $G = SL(2,R) 
\cup \left ( \begin{smallmatrix} 1 & 0 \\ 0 & -1 \end{smallmatrix}
\right ) SL(2,R)$ and $\gh = \{ \left ( \begin{smallmatrix} a & 0 \\ 
0 & -a \end{smallmatrix} \right ) : a \text{ real }\}$.  Then
$M^\dagger = M = \{\pm \left ( \begin{smallmatrix}1 & 0 \\ 0 & 1 
\end{smallmatrix}\right )\,,\, \left ( \pm \begin{smallmatrix}1 & 0 \\ 0 & -1 
\end{smallmatrix} \right ) \}$ so $M^\dagger G^0 = G \ne G^\dagger = G^0$.}
\end{remark}

Fix $[\mu] \in \widehat{U}$ and $\sigma \in \ga^*$,
	so $[\mu \otimes e^{i\sigma}] \in \widehat{U\times A}$.
As usual, $\Sigma_\ga^+$ is the positive $\ga$--root system on $\gg$
such that $\gn$ is the sum of the negative $\ga$--root spaces, and
$\rho_\ga = \frac{1}{2}\sum_{\phi \in Sigma_\ga^+} (\dim \gg^\phi)\phi$, 
so $\ga$ acts on $\gn$ with trace $-2\rho_\ga$\,.  Now $UAN$ acts on the
representation space $V_\mu$ of $\mu$ by
$$
\gamma_{\mu,\sigma}(uan) = e^{\rho_a + i\sigma}(a)\mu(u).
$$
That specifies the associated $G$--homogeneous complex vector bundle
\begin{equation}\label{8.1.4}
p:\E_{\mu,\sigma} \to G/UAN = G(x) = Y.
\end{equation}

\begin{lemma}\label{8.1.5}
There is a unique assignment of complex 
structures to the parts $p^{-1}S_{[gx]}$ of $\E_{\mu,\sigma}$ over
the holomorphic arc components of $Y$, such that each restriction
$\E_{\mu,\sigma}|_{S_{[gx]}} \to S_{[gx]}$ is an $N_{[gx]}$--homogeneous
holomorphic vector bundle.  The assignment is a $G$--equivariant 
real analytic tangent space distribution on $\E_{\mu,\sigma}$.
\end{lemma}

\begin{proof} Lemma 7.1.4 says that $p^{-1}S_{[gx]}$ has a unique complex 
structure for which $\E_{\mu,\sigma}|_{S_{[gx]}} \to S_{[gx]}$ is an
$\Ad(g)M^\dagger$--homogeneous holomorphic vector bundle.  Each
$\Ad(g)(an)$ is trivial on $S_{[gx]} = gS_{[x]}$ and multiplies all fibers of
$\E_{\mu,\sigma}|_{S_{[gx]}}$ by the same scalar $e^{\rho_\ga + i\sigma}(a)$.
Now the complex structure on $p^{-1}S_{[gx]}$ is invariant by the action of
$\Ad(g)N_{[x]} = N_{[gx]}$\,, so $\E_{\mu,\sigma}|_{S_{[gx]}} \to S_{[gx]}$
is an $N_{[gx]}$--homogeneous holomorphic vector bundle.  Finally, the
assignment of complex structures to the $p^{-1}S_{[gx]}$ is $G$--invariant
by uniqueness, thus also real analytic. 
\end{proof}

If $z \in Y = G(x)$ we have the holomorphic tangent space $T_z$ to $S_{[z]}$
at $z$. 
Evidently $\{T_z\}_{z \in Y}$ is a $G$--invariant complex tangent space
distribution on $Y$, so it is real analytic.  Thus 
$\T := \bigcup_{z \in Y} T_z$
is a $G$--homogeneous real analytic sub--bundle of the complexified
tangent bundle of $Y$.  Given non-negative integers $p$ and $q$, the
space of {\em partially smooth} $(p,q)$--{\em forms} on $Y$ with values
in $\E_{\mu,\sigma}$ is 
\begin{equation}\label{8.1.6c}
\begin{aligned}
A^{p,q}(Y;\E_{\mu,\sigma}): &\text{ measurable sections $\alpha$ of }
	\E_{\mu,\sigma}\otimes \Lambda^p\T^* \otimes \Lambda^q\overline{\T^*}\\
&\text{ where $\alpha$ is $C^\infty$ on each holomorphic arc component of $Y$.} 
\end{aligned}
\end{equation}
If $\alpha \in A^{p,q}(Y;\E_{\mu,\sigma})$ and $z \in Y$ then 
$\alpha|_{S_{[z]}}$ is a smooth $(p,q)$--form on $S_{[z]}$ with values in
$\E_{\mu,\sigma}|_{S_{[z]}}$\,, in the ordinary sense.  The
$\overline{\partial}$ operator of $X$ specifies operators
$\overline{\partial}: A^{p,q}(Y;\E_{\mu,\sigma}) \to 
A^{p,q+1}(Y;\E_{\mu,\sigma})$.

We need hermitian metrics for the harmonic theory.  Let $\theta$ be a
Cartan involution of $G$ with $\theta(H) = H$ and denote $K = \{g \in G \mid
\theta(g) = g\}$ as usual.  Then $K \cap N_{[x]} = K \cap M^\dagger$
can be assumed to contain $U$, and we have an $M^\dagger$--invariant hermitian 
metric on the complex manifold $S_{[x]}$\,.  Every holomorphic arc component of
$G(x)$ is an $S_{[kx]}$\,, $k \in K$.  Give $S_{[kx]}$ the hermitian
metric such that the $k: S_{[x]} \to S_{[kx]}$ are hermitian isometries.
In other words, we have a $K$--invariant hermitian metric on the fibers of
the bundle $\T \to Y$.  Similarly the unitary structure of $E_\mu$ specifies
an $M^\dagger$--invariant hermitian metric on the fibers of 
$\E_{\mu,\sigma} \to Y$.  Now we have $K$--invariant hermitian metrics on
the fibers of the bundles $\E_{\mu,\sigma}\otimes \Lambda^p\T^* \otimes
 \Lambda^q\overline{\T^*} \to Y$.  As in (\ref{7.1.5}), that specifies
Hodge--Kodaira operators 
\begin{equation}\label{8.1.7a}
A^{p,q}(Y;\E_{\mu,\sigma}) \overset{\#}{\rightarrow}
	A^{n-p,n-q}(Y;\E^*_{\mu,\sigma}) \overset{\widetilde{\#}}{\rightarrow}
	A^{p,q}(Y;\E_{\mu,\sigma})
\end{equation}
where $n = \dim_\C S_{[x]}$\,.  It also specifies a pre Hilbert space
\begin{equation}\label{8.1.7b}
A_2^{p,q}(Y;\E_{\mu,\sigma}) = \left \{ \alpha \in
	A^{p,q}(Y;\E_{\mu,\sigma}) \left | \int_{K/Z}\left ( \int_{S_{[kx]}}
	\alpha \bar{\wedge} \#\alpha \right ) d(kZ) < \right . \infty \right \}
\end{equation}
whose inner product is $\langle \alpha , \beta \rangle = 
\int_{K/Z}\left ( \int_{S_{[kx]}} \alpha \bar{\wedge} \#\beta\right ) d(kZ)$.

We define {\em square integrable partially-{$(p,q)$}-form} on $Y$ with
values in $\E_{\mu,\sigma}$ to mean an element of 
\begin{equation}\label{8.1.8}
L_2^{p,q}(Y;\E_{\mu,\sigma}): \text{ Hilbert
	space completion of } A_2^{p,q}(Y;\E_{\mu,\sigma}).
\end{equation}
$\overline{\partial}$ is densely defined on $L_2^{p,q}(\E_{\mu,\sigma})$
with formal adjoint $\overline{\partial}^* = - \widetilde{\#}\overline{\partial}
\#$; this follows from the corresponding standard fact (\ref{7.1.7}) over each
holomorphic arc component.  The analogue of the Hodge--Kodaira--Laplacian
is
\begin{equation}\label{8.1.8a}
\square = (\overline{\partial} +
	\overline{\partial}^*)^2 = \overline{\partial}\overline{\partial}^*
	+ \overline{\partial}^*\overline{\partial}\,, 
\end{equation}
which is elliptic and essentially self adjoint over every holomorphic
arc component.  Now $\square$ is essentially self adjoint on
$L_2^{p,q}(Y;\E_{\mu,\sigma})$ from the domain consisting of $C^\infty$ 
forms with
support compact modulo $Z$.  We write $\square$ for the closure, which is the
unique self--adjoint extension on $L_2^{p,q}(Y;\E_{\mu,\sigma})$.  The kernel
\begin{equation}\label{8.1.8b}
H_2^{p,q}(Y;\E_{\mu,\sigma}) = \{\omega \in L_2^{p,q}(Y;\E_{\mu,\sigma}) \mid
	\square\,\omega = 0\}
\end{equation}
is the space of {\em square integrable partially harmonic $(p,q)$--forms}
on $Y$ with values in $\E_{\mu,\sigma}$.  $H_2^{p,q}(Y;\E_{\mu,\sigma})$ is the
subspace of $A_2^{p,q}(Y;\E_{\mu,\sigma})$ consisting of all elements $\omega$ 
such that $\omega|_{S_{[kx]}}$ is harmonic a.e. in $K/Z$.  It is a closed
subspace of $L_2^{p,q}(Y;\E_{\mu,\sigma})$ and there is an orthogonal
direct sum decomposition 
\begin{equation}\label{8.1.8c}
L_2^{p,q}(Y;\E_{\mu,\sigma}) = 
     c\ell\, \overline{\partial}L_2^{p,q-1}(Y;\E_{\mu,\sigma}) \oplus
	\overline{\partial}^*L_2^{p,q+1}(Y;\E_{\mu,\sigma}) \oplus
	H_2^{p,q}(Y;\E_{\mu,\sigma})
\end{equation}
obtained by applying (\ref{7.1.9}) to each holomorphic arc component.

\begin{lemma}\label{8.1.9}
The action $[\widetilde{\pi}^{p,q}_{\mu,\sigma}\alpha](z)
= g(\alpha(g^{-1}z))$ of $G$ on $L_2^{p,q}(Y;\E_{\mu,\sigma})$ is a
unitary representation.
\end{lemma}
\begin{proof}
$\E_{\mu,\sigma}\otimes \Lambda^p\T^* \otimes \Lambda^q\overline{\T}^*$ has 
fiber $E_\mu^{p,q} := E_\mu \otimes \Lambda^p T_x^* \otimes 
\Lambda^q\overline{T}_x^*$ over $x$.  If $\mu^{p,q}$ denotes the 
representation of $U$ on $E_\mu^{p,q}$ then $UAN$ acts on $E_\mu^{p,q}$ by
$$
\gamma_{\mu,\sigma}^{p,q}(uan) = e^{\rho_\ga + i\sigma}(a)\mu^{p,q}(u)
	= e^{\rho_\ga}(a)\cdot {'\gamma}_{\mu,\sigma}^{p,q}(uan)
$$
where $'\gamma_{\mu,\sigma}^{p,q} = \mu^{p,q}\otimes e^{i\sigma}$ is
unitary. Since $e^{\rho_\ga}(a)$ is the square root of the determinant of
$uan$ on the real tangent space $\gg/(\gu + \ga + \gn)$ to $Y$ at $x$,
now $\widetilde{\pi}^{p,q}_{\mu,\sigma}$ is the unitarily induced 
representation $\Ind_{UAN}^G({'\gamma}_{\mu,\sigma}^{p,q})$.
\end{proof}

The representation $\widetilde{\pi}^{p,q}_{\mu,\sigma}$ commutes
with $\overline{\partial}$, hence also with $\overline{\partial}^*$, so
$H_2^{p,q}(Y;\E_{\mu,\sigma})$ is a closed $G$--invariant subspace of
$L_2^{p,q}(Y;\E_{\mu,\sigma})$.  Thus we have
\begin{equation}\label{8.1.10}
{\pi}^{p,q}_{\mu,\sigma}:\,\text{ unitary representation
	of $G$ on } H_2^{p,q}(Y;\E_{\mu,\sigma})\,\,\text{ and }\,\,
	{\pi}_{\mu,\sigma}^q = {\pi}^{0,q}_{\mu,\sigma}
\end{equation}
The program of Section \ref{sec8} is to represent the various $H$--series of
unitary representation classes of $G$ by the various ${\pi}_{\mu,\sigma}^q$\,.

\subsection{}\label{ssec8b}\setcounter{equation}{0}
We set up ${\pi}_{\mu,\sigma}^q$ as an induced representation from
$N_{[x]} = M^\dagger AN$.  Write 
$\E_\mu = \E_{\mu,\sigma}|_{S_{[x]}} \to S_{[x]}$\,.  It is the 
$M^\dagger$--homogeneous hermitian holomorphic vector bundle defined by
$[\mu] \in \widehat{U}$ as in Lemma \ref{7.1.4}.  That gives us the unitary
representations $\eta_\mu^q$ of $M^\dagger$ on $H_2^{0,q}(S_{[x]};\E_\mu)$.
The formula $\eta_{\mu,\sigma}^q(man) = e^{i\sigma}(a)\eta_\mu^q(m)$
defines a unitary representation of $N_{[x]} = M^\dagger AN$ on
$H_2^{0,q}(S_{[x]};\E_\mu)$.

\begin{theorem}\label{8.2.2} $[\pi_{\mu,\sigma}^q] = 
[\Ind_{N_{[x]}}^G(\eta_{\mu,\sigma}^q)]$\,.
\end{theorem}

\begin{proof}
Let $\widetilde{\pi} = \widetilde{\pi}_{\mu,\sigma}^{0,q}$\,, the
representation of $G$ on $L_2^{p,q}(Y;\E_{\mu,\sigma})$.  Let $'\gamma$ denote 
the representation of $UAN$ on $E_\mu \otimes \Lambda^q(\overline{T}_x^*)$\,;
it is the ${'\gamma}_{\mu,\sigma}^{0,q}\otimes e^{i\sigma}$ of the proof of
Lemma \ref{8.1.9}.  That lemma was proved (if $p = 0$) by showing
$[\widetilde{\pi}] = [\Ind_{UAN}^G({'\gamma})]$.

Let $\widetilde{\eta}$ denote the representation of $M^\dagger$ on
$L_2^{0,q}(S_{[x]};\E_\mu)$, and $'\eta$ the representation of
$M^\dagger AN$ by ${'\eta}(man) = e^{i\sigma}(a)\widetilde{\eta}(m)$.  Then
$[\widetilde{\eta}] = [\Ind_U^{M^\dagger}(\mu^{0,q})]$, and so $[{'\eta}]
= [\Ind_{UAN}^{M^\dagger AN}(\gamma)]$.

Induction by stages now says that $\widetilde{\pi}$ is unitarily 
equivalent to $\Ind_{M^\dagger AN}^{G}('\eta)$.  We need the equivalence.
Let $f$ be in the representation space of 
$\Ind_{M^\dagger AN}^{G}('\eta)$.  In other words 
$f: G \to L_2^{0,q}(S_{[x]};\E_\mu)$ is Borel measurable,
$f$ transforms by
$f(gman) = e^{-\rho_\ga}(a)\cdot {'\eta}(man)^{-1}f(g)$ for $g \in G$ and
$man \in M^\dagger AN$, and we have global norms 
$\int_{K/Z} ||f(k)||^2 d(kZ) < \infty$.  For almost all $g \in G$ we may view
$f(g) \in L_2^{0,q}(S_{[x]}; \E_\mu)$ as a Borel--measurable function
$M^\dagger AN \to E_\mu^q = E_\mu\otimes \Lambda^q(\overline{T}_x^*)$ such that
$$
\begin{aligned}
f(g)(puan) = &{'\gamma}(uan)^{-1}[f(g)(p)] \text{ for } p \in M^\dagger AN\,, 
	uan \in UAN \\
&\text{ and } \int_{M^\dagger/U}||f(g)(m)||^2d(mU) < \infty.
\end{aligned}
$$
Now define
\begin{equation}\label{8.2.3}
F = \Gamma(f): G \to E_\mu^q = E_\mu\otimes \Lambda^q(\overline{T}_x^*) 
	\text{ by } F(g) = f(g)(1).
\end{equation}
Then $F$ is Borel measurable.  Use ${'\eta}=\Ind_{UAN}^{M^\dagger AN}('\gamma)$
to compute
$$
\begin{aligned}
F(guan) &= f(guan)(1) = [e^{-\rho_\ga}(a)\cdot {'\eta}(uan)^{-1}f(g)](1) \\
&= e^{-\rho_\ga}(a)\cdot  {'\gamma}(uan)^{-1}[f(g)(1)] 
		= e^{-\rho_\ga}(a)\cdot {'\gamma}(uan)^{-1}F(g)
\end{aligned}
$$
and
{\small
$$
\int_{K/Z}\Bigl \{ \int_{M^\dagger /U} ||F(km)||^2\,d(mU)\Bigr \} d(kZ)
= \int_{K/Z}\Bigl \{ \int_{M^\dagger /U} ||f(k)(m)||^2\, d(mU)\Bigr \} d(kZ)
< \infty.
$$
}
Thus $f \mapsto \Gamma(f) = F$ is the desired equivalence
$\widetilde{\pi} \simeq \Ind_{M^\dagger AN}^{G}('\eta)$.

In the construction just above, $f$ is in the representation space of
$\Ind_{N{[x]}}^g(\eta^q_{\mu,\sigma})$ precisely when almost every $f(g)$ 
is annihilated by the Hodge--Kodaira--Laplace operator of 
$\E_\mu \to S_{[x]}$\,.  That is equivalent to $\square\,\Gamma(f) = 0$.
Thus the equivalence $\Gamma$ of (\ref{8.2.3}) restricts to the equivalence
asserted in Theorem \ref{8.2.2}.
\end{proof}

\subsection{}\label{ssec8c}\setcounter{equation}{0}
We now come to the geometric realization of the various $H$--series
of unitary representations of $G$.  Note that these are the standard
induced representations.  They are unitary, in fact tempered.

We construct a particular positive $\gh_\C$--root system $\Sigma^+$\,.
The choice $P = MAN$ us a choice $\Sigma_\ga^+$ of positive $\ga$--root
system on $\gg$.  Let $\overline{M} = M/Z_G(G^0)$, and choose a simple
$\gt_\C$--root system $\Pi_\gt$ on $\gm_\C$ such that the parabolic
subalgebra $\gq\cap \overline{\gm}_\C$ of $\overline{\gm}_\C$ is 
specified by $\Pi_\gt$ and a subset $\Phi_\gt$\,.  $\Sigma^+$ will be
the positive $\gh_\C$--root system on $\gg_\C$ determined by $\Sigma_\ga^+$
and $\Sigma_\gt^+$\,, and $\Pi$ is the simple root system for $\Sigma^+$.
From Proposition \ref{6.7.4} the measurable orbit $G(x)$ is integrable
exactly when $\gq = \gq_\Phi$ with 
$\Phi = \Phi_\gt \cup (\Pi \setminus \Pi_\gt)$.  As we had done before 
for $\gg$ let
$$
\rho_\gt = \tfrac{1}{2}\sum_{\phi\in\Sigma_\gt^+} \phi\,,
\Delta_{M,T} = \prod_{\gt^+}(e^{\phi/2}-e^{-\phi/2})\,,
\varpi_\gt(\nu) = \prod_{\gt^+} \langle\nu,\phi\rangle.
$$
Replacing $G$ be a $\Z_2$ extension if necessary, Lemma \ref{4.3.6}
ensures that $e^{\rho_\gt}$ and $\Delta_{M,T}$ are well defined on $T$.
As we did before for $(G,H)$ denote
\begin{equation}\label{8.3.2}
L_\gt = \{\nu\in i\gt^*\mid e^\nu \text{ is well defined on } T^0\} \text{ and }
L_\gt'' = \{\nu \in L_\gt \mid \varpi_\gt(\nu) \ne 0\}.
\end{equation}
Then $\rho_\gt \in L_\gt''$.  If $\nu+\rho_\gt \in L_\gt''$ then $\Sigma_\gt^+$ 
specifies $q_M(\nu+\rho_\gt) \in \Z$ as in (\ref{7.2.2}).

Since $U = Z_M(M^0)U^0$ and $U\cap M^0 = U^0$, $\widehat{U}$ consists of all
$[\chi\otimes\mu^0]$ with $\chi \in \widehat{Z_M(M^0)}$ consistent with
$[\mu^0] \in \widehat{U^0}$. 

We review some aspects of our basic setup and then come to our main result.

$G$ is a general real reductive Lie group as defined in 
{\rm \S \ref{ssec3a}}, $Q$ is a parabolic subgroup of $\overline{G}_\C$\,,
and $Y = G(x) \subset X = \overline{G}_\C$ is a measurable integrable orbit
partially complex orbit of flag type as described in {\rm \S \ref{ssec6e}}.  
$H = T\times A \in \Car(G)$, $P=MAN$ is an associated cuspidal 
parabolic subgroup of $G$, and we suppose that
$U = \{m \in M \mid m(x)=x\}$ is compact modulo $Z_G(G^0)$.
$S_{[x]}$ is the holomorphic arc component of $Y$ through $x$ and its
$G$--normalizer is $N_{[x]} = M^\dagger AN$ denotes the $G$--normalizer of 
the holomorphic arc component $S_{[x]}$. 

Recall that $\pi^q_{\mu,\sigma}$ denotes the unitary
representation of $G$ on  $H_2^{0,q}(Y;\E_{\mu,\sigma})$ and
that $^H\pi_{\mu,\sigma}^q$ denotes the sum of its irreducible 
subrepresentations.

\begin{theorem}\label{8.3.4}
Let $[\mu] \in \widehat{U}$, say $[\mu] = [\chi\otimes\mu^0]$ as above.
Let $\nu$ be the highest weight of $\mu^0$ in the $\gt_\C$--root system
$\Sigma_\gt^+\cap\Phi_\gt^r$ of $\gu_\C$.  Then $\nu \in L_\gt$ and
$\mu \in \widehat{U}_\zeta$ where $\zeta \in \widehat{Z}$ agrees with $e^\nu$ 
on $Z\cap M^0$ and where $[\chi] \in \widehat{Z_M(M^0)}_\zeta$\,.  
Let $\sigma \in \ga^*$ and suppose that $\nu + \rho_\gt \in L_\gt''$\,.  

{\rm 1.} If $q \ne q(\nu+\rho_\gt)$ then $H_2^{0,q}(Y;E_{\mu,\sigma}) = 0$,
so $\pi^q_{\mu,\sigma}$ does not occur

{\rm 2.} If  $q = q(\nu+\rho_\gt)$ then $G$ acts on
$H_2^{0,q}(Y;E_{\mu,\sigma})$ by the $H$--series representation
$\pi_{\chi,\nu + \rho_\gt, \sigma}$.  Every $H$--series representation
is obtained in this way.
\end{theorem}

\begin{proof}
Theorem \ref{7.2.3} says that the representation $\eta_\mu^q$ of $M^\dagger$ on
$H_2^{0,q}(S_{[x]};\E_\mu)$ is trivial if $q \ne q_M(\nu + \rho_\gt)$, and if 
$q = q_M(\nu + \rho_\gt)$ it is equivalent to the relative discrete series 
representation $\eta_{\chi,\nu + \rho_\gt}$ of $M^\dagger$.  Then the 
representation of $M$ on $H_2^{0,q}(S_{[x]};\E_\mu)$ is the relative
discrete series representation (which we temporarily denote 
$\eta'_{\chi,\nu + \rho_\gt}$) $\Ind_{M^\dagger}^M(\eta_{\chi,\nu + \rho_\gt})$
of $M$.  Finally, if $\sigma \in \ga^*$ then $G$ acts on 
$H_2^{0,q}(Y;E_{\mu,\sigma})$ by the $H$--series representation
$\Ind_{MAN}^G \bigr ( \Ind_{M^\dagger AN}^{MAN}
(\eta_{\chi,\nu + \rho_\gt}\otimes e^{i \sigma})\bigl ) =
\pi_{\chi,\nu + \rho_\gt, \sigma}$ by Theorem \ref{8.2.2}.
\end{proof}

\subsection{}\label{ssec8d}\setcounter{equation}{0}
Theorem \ref{8.3.4} gives explicit geometric realizations for the
standard tempered representations, i.e., for the various $H$--series
classes of unitary representations of $G$.  Theorem \ref{7.2.3} is the 
special case of the relative discrete series.  In view of the Plancherel 
Theorem \ref{5.1.6} we now have, for every $\zeta \in \widehat{Z}$,
explicit geometric realizations for a subset of $\widehat{G}_\zeta$ that
supports Plancherel measure there.
\vfill\pagebreak

\end{document}